\documentclass[11pt]{amsart}

\usepackage{latexsym}
\usepackage{amssymb}
\usepackage{amsmath}
\usepackage{color}

\usepackage{amsmath,amstext,amssymb,amsfonts,amsthm,amsxtra,amscd}
\usepackage{latexsym}
\usepackage{mathrsfs}
\usepackage{color}
\usepackage[all]{xy}
\usepackage{bbm}
\usepackage{hyperref} 
\usepackage{setspace}

\newtheorem{theorem}{Theorem}[section]
\newtheorem{prop}[theorem]{Proposition}
\newtheorem{cor}[theorem]{Corollary}
\newtheorem{lemma}[theorem]{Lemma}
\newtheorem{question}[theorem]{Question}

\newcommand\ran{\mathop{\rm ran}}

\newcommand\Span{\mathop{\rm span}}

\newcommand\rank{\mathop{\rm rank}}

\newcommand\conv{\mathop{\rm conv}}
\newcommand\her{\mathop{\rm her}}

\newcommand\ap{\mathop{\rm ap}}
\newcommand\cp{\mathop{\rm cp}}
\newcommand\thet{\mathop{\rm th}}
\newcommand\fp{\mathop{\rm fp}}
\newcommand\vp{\mathop{\rm vp}}
\newcommand\fvp{\mathop{\rm fvp}}

\newcommand\thab{\mathop{\rm thab}}

\newcommand{\cl}[1]{\mathcal{#1}}
\newcommand{\bb}[1]{\mathbb{#1}}

\newcommand{\To}{\rightarrow}
\newcommand{\ep}{\epsilon}

\newcommand\Tr{\mathop{\rm Tr}}

\newcommand{\ip}[2]{\ensuremath{\left\langle #1 , #2\right\rangle}}
\theoremstyle{definition}
\newtheorem{definition}[theorem]{Definition}
\theoremstyle{remark}
\newtheorem{example}[theorem]{Example}

\newtheorem{remark}[theorem]{Remark}

\usepackage{graphicx}

\begin{document}

\title[Parameters of non-commutative graphs]
{Information theoretic parameters of non-commutative graphs and convex corners}

\author[G. Boreland]{Gareth Boreland}
\address{Mathematical Sciences Research Centre,
Queen's University Belfast, Belfast BT7 1NN, United Kingdom}
\email{gboreland01@qub.ac.uk}

\author[I. G. Todorov]{Ivan G. Todorov}
\address{
School of Mathematical Sciences, University of Delaware, 501 Ewing Hall,
Newark, DE 19716, USA, and 
Mathematical Sciences Research Centre,
Queen's University Belfast, Belfast BT7 1NN, United Kingdom}
\email{todorov@udel.edu}

\author[A. Winter]{Andreas Winter}
\address{
ICREA and Grup d'Informaci\'o Qu\`antica, Departament de F\'isica, 
 Universitat Aut\`onoma de Barcelona, 08193 Bellaterra 
 (Barcelona), Spain}
\email{andreas.winter@uab.cat}


\date{3 January 2021}

\maketitle

\begin{abstract}
We establish a second anti-blocker theorem for non-com-mutative convex corners, 
show that the anti-blocking operation is continuous on bounded sets of convex corners, 
and define optimisation parameters for a given convex corner 
that generalise well-known graph theoretic quantities. 
We define the entropy of a state with respect to a convex corner, characterise its maximum value
in terms of a generalised fractional chromatic number and establish entropy splitting results 
that demonstrate the entropic complementarity between a convex corner and its anti-blocker.
We identify two extremal tensor products of convex corners and examine the behaviour of the introduced parameters 
with respect to tensoring.
Specialising to non-commutative graphs, we obtain
quantum versions of the fractional chromatic number and the clique covering number,
as well as a notion of non-commutative graph entropy of a state, which we show to be  
continuous with respect to the state and the graph. 
We define the Witsenhausen rate of a non-commutative graph and compute the values of 
our parameters in some specific cases. 
\end{abstract}

\tableofcontents


\section{Introduction}\label{s_intro}

The importance of graphs in information theory was recognised by Shannon in the early stages 
of its formation. The underlying idea, which he pioneered in \cite{Shannon2}, is to use the adjacency relation between 
the vertices of a graph as signifying the confusability between the symbols from an alphabet, 
transmitted via a noisy information channel. 
This led to the definition of the \emph{zero-error capacity} of a channel as an asymptotic parameter, 
depending on the behaviour of the independence numbers of the iterated strong products of the graph. 
In a similar vein, Witsenhausen \cite{wit} identified the optimal rate of transmission via a channel 
with side information, nowadays known as the \emph{Witsenhausen rate}.

In the coding problem for a source, 
K\"{o}rner \cite{Korner} employed 
the asymptotic behaviour of the chromatic numbers of the conormal graph products
to define the \emph{graph entropy} $H(G,p)$ of the source $p$, the optimal compression rate
in the presence of ambiguity captured by the graph $G$. 
Very importantly from a computational viewpoint, he expressed $H(G,p)$ 
as the solution of an optimisation problem over a convex polytope in $\bb{R}^d$, 
canonically associated with $G$. 
Graph entropy has since attracted a considerable attention in the literature, 
see e.g. \cite{Csis, relax, Korner2, Kornersimonyituza, Kornerandsimonyi, Symmgraphs, Simonyi1, Simonyi2}. 
The role similar subsets of $\bb{R}^d$, canonically associated with the graph $G$, play in information theoretical questions was 
emphasised by Gr\"{o}tschel, Lov\'{a}sz and Schrijver \cite{relax} (see also their monograph \cite{lovbook}, 
as well as Knuth's survey \cite{knuth}), 
who defined \emph{convex corners} in $\bb{R}^d$ as a unifying concept, capturing a number of 
previously considered contexts. It was thus possible to see graph entropy as a 
special case of a much more general entropic quantity, attributed to any convex corner, 
leading, among others, to probabilistic versions \cite{Marton} of the fundamental Lov\'{a}sz number \cite{lovasz}. 

Confusability in quantum information was examined in \cite{cubitt-chen-harrow, cubitt-smith, duan2, duan},
which identify a suitable quantum analogue of graphs. 
\emph{Non-commutative graphs} are simply \emph{operator systems} in the space $M_d$ 
of all complex $d$ by $d$ matrices, that is, linear subspaces closed under the adjoint operation and 
containing the identity matrix \cite{paul}. Every graph $G$ on $d$ vertices gives rise to a 
canonical operator system $\cl S_G \subseteq M_d$, which remembers $G$ up to a graph isomorphism 
\cite{Paulsen}. This led to defining and studying a number of graph parameters with relevance in 
information theory in the non-commutative setting, initiating what can be called 
\emph{non-commutative combinatorics}. Fruitful quantum versions of, among others, the Lov\'{a}sz number \cite{duan}, 
the chromatic number \cite{Paulsen}, the clique and fractional clique number \cite{BTW}, 
the minimum semi-definite rank and the intersection number \cite{LPT}, the Sandwich Theorem \cite{LPT} 
(see \cite{knuth}) and
a Ramsey-type theorem \cite{Weaver} have thereafter been found. 
In \cite{BTW}, the authors introduced a non-commutative version of convex corners; 
however, a further development was impeded at that stage by the absence of 
a \emph{second anti-blocker theorem}, a fundamental result that holds for classical convex corners \cite{knuth}.

In the present paper, we fill this gap by proving a quantum version of the anti-blocker theorem.
This allows us, in particular, to establish the equality  between the 
fractional chromatic number $\chi_{\rm f}(\cl S)$ 
and fractional clique number $\omega_{\rm f}(\cl S)$ of a non-commutative graph $\cl S$, extending the 
well-known duality result for classical graphs. 
We define the non-commutative graph entropy $H(\cl S,\rho)$ of a state $\rho$ with respect to a 
non-commutative graph $\cl S$, which reduces to the von Neumann entropy $H(\rho)$ of $\rho$ in case 
$\cl S$ coincides with the complete non-commutative graph $M_d$, and extends classical graph entropy in that 
$H(\cl S_G,p) = H(G,p)$, when $G$ is a graph on $d$ vertices and the probability distribution $p$ on its vertex set is 
viewed as a diagonal quantum state in $M_d$. 
The parameter $H(\cl S,\rho)$ is a special case of the entropy parameter $H_{\cl A}(\rho)$ 
attached to any non-commutative convex corner $\cl A$.
Another application of the second anti-blocker theorem yields an optimisation result, identifying 
the maximum entropy of a convex corner $\cl A$ in terms of a generalised fractional chromatic number of $\cl A$, 
defined also in the present paper as an extension of the fractional chromatic number of a non-commutative 
graph. 
We define the Witsenhausen rate of a non-commutative graph, and study the behaviour of $H_{\cl A}(\rho)$ as a 
function on $\cl A$, obtaining continuity results which are new also in the classical case. 

The paper is organised as follows: in Section \ref{s_crdc}, we recall the basic notions from the theory of 
classical convex corners and introduce several parameters used subsequently that can be thought of
as continuous versions of combinatorial parameters associated to graphs, such as the 
independence number, the fractional chromatic number and others. 
In Section \ref{s_ccabMd}, we examine non-commutative convex corners in $M_d$
as a quantum version of classical convex corners in $\bb{R}^d$.
These are closed convex subsets $\cl A$ of positive semi-definite matrices in $M_d$, 
possessing a natural hereditarity property. 
We extend the parameters from Section \ref{s_crdc} as solutions of optimisation problems 
over $\cl A$, define the
non-commutative anti-blocker $\cl A^{\sharp}$ of $\cl A$ 
and consider some examples. The latter are
used in Section \ref{sect_AB} in establishing the second anti-blocker theorem, stating that convex corners in $M_d$
satisfy the relation $\cl A^{\sharp\sharp} = \cl A$. In addition, we prove the continuity of the anti-blocking operation.
In Section \ref{sssde}, we show that a classical convex corner in $\bb{R}^d$ possesses two extremal quantisations, 
which are distinct provided $d > 1$. 

In Section \ref{sssmde}, we introduce the entropy $H_{\cl A}(\rho)$ of a state $\rho$ with respect to a given 
convex corner $\cl A\subseteq M_d$ and identify its maximum value
in terms of the optimisation parameters defined in Section \ref{s_ccabMd} (Theorem \ref{noncomm13}). 
This can be thought of as a continuous and quantum version of the 
corresponding facts \cite{Simonyi1, Marton} 
for the vertex packing polytope and the Gr\"{o}tschel-Lov\'{a}sz-Schrijver convex corner of a graph \cite{relax}, 
and is new in this generality even in the commutative case. 
We examine the continuity of $H_{\cl A}(\rho)$ both as a function on $\cl A$ and as a function on $\rho$,
and obtain quantum versions of the entropy splitting results from \cite{Csis} (Theorem \ref{noncomm23}). 

In Section \ref{ss_pcc}, we define two extremal tensor products of non-commutative convex corners 
and establish relations between the value of our parameters on a tensor product and the values on 
its components.
This leads to inequalities for the entropy of an entangled state with respect to a tensor product convex corner,
new also in the commutative case (Theorem \ref{th_enprod}). 

In Section \ref{s_ccfncg}, we consider three convex corners associated with a non-commutative graph $\cl S$:
the abelian projection corner ${\rm ap}(\cl S)$, the clique projection corner ${\rm cp}(\cl S)$
and the full projection corner ${\rm fp}(\cl S)$. 
Viewing projections as quantum versions of sets, we have that ${\rm ap}(\cl S)$ is 
a quantum version of the vertex packing polytope ${\rm vp}(G)$, while 
${\rm cp}(\cl S)$ and ${\rm fp}(\cl S)$ are quantum versions of the fractional vertex packing polytope
${\rm fvp}(G)$ of a graph $G$ \cite{relax, lovbook}. 
Several parameters for a non-commutative graph are thus defined as a specialisation of 
the optimisation parameters from Section \ref{s_ccabMd} to 
the corners ${\rm ap}(\cl S)$, ${\rm cp}(\cl S)$ and ${\rm fp}(\cl S)$ and their anti-blockers. 
The non-commutative graph entropy $H(\cl S,\rho)$ of $\cl S$ is defined in Section \ref{s_nge},
and its maximum value is identified in terms of the fractional chromatic number $\chi_{\rm f}(\cl S)$ 
of $\cl S$, while the clique and the clique covering number of $\cl S$ are examined in Section \ref{ss_ccn}. 
Section \ref{s_wit} contains some multiplicativity properties of 
the chromatic, the fractional chromatic, the clique and the clique covering numbers of a non-commutative 
graph that lead to the definition of its Witsenhausen rate. Finally, in Section \ref{ch4examples} 
we identify the values of our parameters in several specific examples.

\smallskip

\subsection{Notation}

For $d\in \bb{N}$, write $[d] = \{1,2,\dots,d\}$.
We denote by $\bb{R}^d_+$ the cone of all real $d$-vectors with non-negative entries,
and write $\cl P_d$ for its subset of probability distributions. 
Sometimes we work with the extended real line $\bb{R} \cup \{\infty\}$ and use the conventions 
$\frac{1}{0} = \infty$, $\frac{1}{\infty} =0$ and $0\log 0 = 0$. 
For $u,v\in \bb{R}^d$, we write $u\leq v$ when $v - u\in \bb{R}^d_+$.

Let $\{e_1, \ldots, e_d\}$ be the canonical basis of $\bb C^d$ and 
$M_d$ be the algebra of all complex $d \times d$ matrices. 
For $u,v\in \bb{C}^d$, let $uv^*$ be the rank one operator in $M_d$, given by $uv^*(w) = \langle w,v\rangle u$, $w\in \bb{C}$. 
Here, and in the sequel, we use the notation $\langle \cdot, \cdot\rangle$ to refer to both inner product 
(assumed linear on the first variable) and bilinear duality.
For a vector $v\in \bb{C}^d$, let $v_i = \langle v,e_i\rangle$, $i\in [d]$. 
We set $\cl D_d = \Span(\{e_ie_i^*: i \in [d] \})$; 
thus, $\cl D_d$ is the subalgebra of $M_d$ of all diagonal matrices.  
We write $M_d^h$ (resp. $M_d^+$) for the set of all 
Hermitian (resp. positive) matrices in $M_d$, 
and we set $\cl D_d^+= \cl D_d \cap M_d^+$. 
We call a matrix in $M_d$ \emph{strictly positive} if it is positive and invertible; we denote by 
$M_d^{++}$ the set of all strictly positive matrices in $M_d$ and, 
for a set $\cl A\subseteq M_d$, write $\cl A^{++} = \cl A\cap M_d^{++}$.
Similarly, we call a vector $v\in \bb{C}^d$ \emph{strictly positive} if $v_i > 0$ for every $i\in [d]$. 
For $M = (m_{i,j})$ and $N = (n_{i,j}) \in M_d$, we 
let $\ip{M}{N} = \Tr (M N) = \sum_{i,j=1}^d m_{i,j} n_{j,i}$.
The Hilbert-Schmidt (resp. operator) norm of a matrix $M\in M_d$ will be denoted by 
$\|M\|_2$ (resp. $\|M\|$).
For $\delta > 0$, we write ${\rm B}(M, \delta)$ for the open ball with centre $M$ and radius $\delta$ 
with respect to $\|\cdot\|_2$.
Given an orthonormal basis $V$ of a Hilbert space $H$ of dimension $d$, we make a (relative to $V$) identification 
$\cl L(H) \equiv M_d$. We will often write $M_d$ in place of $\cl L(H)$ even if we have not specified a particular basis.
For an orthogonal projection $P\in \cl L(H)$, we write $P^\perp = I - P$.


\section{Convex $\bb{R}^d$-corners}\label{s_crdc}

In this preliminary section, we recall relevant concepts and facts regarding classical convex corners
and formalise some parameters, implicitly used in the literature, which will be used throughout. 
A \emph{convex $\bb{R}^d$-corner} \cite{relax} is a non-empty closed convex subset $\cl A$ of $\bb{R}^d_+$ such that 
$$v\in \cl A, \ 0\leq u\leq v \ \Longrightarrow \ u\in \cl A.$$
The latter property will be referred to as \emph{hereditarity}. 
A convex $\bb{R}^d$-corner is called \emph{standard} if it is bounded and has non-empty 
topological interior.

\begin{lemma}\label{nonemptyint}
Let $\cl A$ be a convex $\bb R^d$-corner.  The following are equivalent:
\begin{itemize}
\item[(i)] $\cl A$ has a non-empty interior;
\item[(ii)] there exists $r > 0$ such that $r \mathbbm 1 \in \cl A$;
\item[(iii)] $\cl A$ contains a strictly positive element.
\end{itemize}
\end{lemma}

\begin{proof}
(i)$\Rightarrow$(ii)
Suppose that $\cl A$ has  non-empty interior.
Let $a \in \cl A$ and $\delta > 0$ be such that 
${\rm B}(a, \delta) \subseteq \cl A$.
Then $a + \frac{1}{2\sqrt{d}} \delta \mathbbm 1\in \cl A$ and, since $\cl A$ is hereditary and $a \ge 0$, we have $\frac{1}{2\sqrt{d}}\delta \mathbbm 1 \in \cl A$.

(ii)$\Rightarrow$(iii)
is trivial.

(iii)$\Rightarrow$(i) Let $b\in \cl A$ be strictly positive. Setting $r= \min_{i \in [d]} b_i$,
we have that $c \le b$ for all $c \in {\rm B}(0,r)$. By the hereditarity of $\cl A$ it follows that $\bb R^d_+ \cap{\rm B}(0,r) \subseteq \cl A$. 
It is trivial to verify that 
$${\rm B}\left( \frac{r}{2 \sqrt{d}} \mathbbm 1, \frac{r}{2 \sqrt{d}}\right) \subseteq \bb R_+^d \cap{\rm B}(0,r) \subseteq \cl A.$$ 
\end{proof}

The \emph{anti-blocker} of a non-empty subset $\cl A\subseteq \bb{R}^d_+$ is given by 
$$\cl A^{\flat} = \{v\in \bb{R}^d_+ : \langle v,u\rangle \leq 1 \mbox{ for all } u\in \cl A\}.$$
It is clear that $\cl A^{\flat}$ is a convex $\bb{R}^d$-corner.
Moreover, the following \emph{second anti-blocker theorem} holds:

\begin{theorem}\cite[Lemma, p. 35]{knuth} \label{th_sabt}
A non-empty subset $\cl A\subseteq \bb{R}^d_+$ is a convex corner if and only if 
$\cl A^{\flat\flat} = \cl A$.
\end{theorem}

We note that Theorem \ref{th_sabt} was formulated in \cite{knuth} only for 
standard convex corners, but a direct verification shows that the same proof remains valid in our generality.

We define the \emph{unit $\bb R^d$-corner} $\cl C_d$
and the \emph{unit $\bb R^d$-cube} $\cl B_d$
by letting 
$$\cl C_d = \{v \in \bb R^d_+ : \|v\|_1 \leq 1\} \ \mbox{ and } \ \cl B_d = \{v \in \bb R^d_+ : \|v\|_{\infty} \le 1\}.$$
It is clear that $\cl C_d$ and $\cl B_d$ are standard convex $\bb R^d$-corners; moreover, if $\lambda > 0$ then
$$\left(\lambda\cl B_d\right)^\flat = \frac{1}{\lambda} \cl C_d \mbox{ and } 
\left(\lambda \cl C_d\right)^\flat= \frac{1}{\lambda} \cl B_d .$$ 
It follows easily that a non-empty subset $\cl A\subseteq \bb{R}^d_+$ 
is a standard convex corner if and only if $\cl A^\flat$ is so.

For a bounded convex $\bb R^d$-corner $\cl A$, we set 
$$\gamma(\cl A)= \max \{ \ip{u}{\mathbbm 1}: u \in \cl A\}.$$
It is clear that $\gamma(\cl A)=0$ if and only if $\cl A = \{0\}.$ 
If the convex $\bb R^d$-corner $\cl A$ is unbounded, we set $\gamma(\cl A) = \infty.$ 

If $\cl A$ is a convex $\bb{R}^d$-corner with $\cl A \ne \bb R^d_+$, 
then the set $\{ \beta \in \bb R_+: \beta \mathbbm 1 \in \cl A\}$ is bounded, and 
we set 
$$N(\cl A) = \max \{ \beta: \beta \mathbbm 1 \in \cl A\}.$$

\noindent We write $N(\bb R_+^d)= \infty.$
By Lemma \ref{nonemptyint}, $N( \cl A)=0$ if and only if $\cl A$ has empty interior.

For a convex $\bb R^d$-corner $\cl A$ with non-empty interior, 
we set
$$M(\cl A) = \inf \left\{\sum_{i=1}^k \lambda_i : \lambda_i > 0 \mbox{ and } \exists \ v_i\in \cl A, i \in [k], 
\mbox{ s.t. } \sum_{i=1}^k \lambda_i v_i \ge \mathbbm 1 \right\}.$$
If  $\cl A$ has empty interior, we set  $M(\cl A)= \infty.$ 
Note that $M(\bb R_+^d)=0.$  

\begin{lemma} \label{MAmin} 
If $\cl A$ is a standard convex corner, the infimum in the definition of $M(\cl A)$ is attained.   
In fact,
$$M(\cl A) = \min \left\{\mu \in \bb{R}_+ : \mbox{there exists } v \in \cl A \mbox{ s.t. } \mu v \ge \mathbbm 1 \right\}.$$
\end{lemma}  

\begin{proof} 
Let $m$ be the right hand side of the displayed identity (its existence is a consequence of the compactness of $\cl A$). 
It is clear that $M(\cl A)\leq m$. 
For $n \in \bb N$, let $x_n = \sum_{i = 1}^{k_n} \lambda_i^{(n)}v_i^{(n)} \ge \mathbbm 1$, 
with $v_i^{(n)} \in \cl A$ and 
$\lambda_i^{(n)} > 0$, be such that $M(\cl A) \le \sum_{i=1}^{k_n} \lambda_i^{(n)} \le M(\cl A) + 1/n$. 
Thus, $\sum_{i = 1}^{k_n} \lambda_i^{(n)} \To_{n \To \infty} M(\cl A)$.  
By convexity, $\left(\sum_{i = 1}^{k_n} \lambda_i^{(n)}\right)^{-1} x_n \in \cl A$ for all $n \in \bb N$.  
Let $v \in \cl A$ be a cluster point of the sequence 
$\left( \left(\sum_{i=1}^{k_n} \lambda_i^{(n)}\right)^{-1} x_n \right)_{n \in \bb N}$.
Then $M(\cl A) v \ge \mathbbm 1$; this shows that $m\leq M(\cl A)$ and hence $m = M(\cl A)$. 
\end{proof}

\begin{prop} \label{NMgamma}
Let $\cl A $ be a convex $\bb R^d$-corner. Then 
$$M(\cl A)= \frac{1}{N (\cl A)} = \gamma( \cl A^\flat).$$  
\end{prop}

\begin{proof}  
We first consider the case where $\cl A \ne \bb{R}_+^d$ and $\cl A$ has non-empty interior.  
Set $y = N(\cl A) \mathbbm 1$, and observe that $y \in \cl A$ and $N(\cl A)>0$.  
We have $\frac{1}{N(\cl A)}y = \mathbbm 1$ and hence $M(\cl A) \le \frac{1}{N(\cl A)}$.  
By Lemma \ref{MAmin}, there exists $v \in \cl A$ satisfying $M(\cl A)v \ge \mathbbm 1$.  
This gives $\frac{1}{M (\cl A)}\mathbbm 1 \le v$ and thus $\frac{1}{M (\cl A)}\mathbbm 1 \in \cl A$ by hereditarity. 
It follows that $N(\cl A) \ge \frac{1}{M(\cl A)}$ and the first equality is proved. 

It is easy to see that $$ \cl A = \bb R_d^+ \iff \cl A^\flat = \{0\} \iff \gamma (\cl A^\flat)=0,$$ 
and that $\cl A^\flat$ is bounded when $\cl A$ has non-empty interior.  Thus when $\cl A \ne \bb{R}_+^d$ has non-empty interior, $0 < \gamma(\cl A^\flat) < \infty.$
To prove the second equality in this case, let $w \in \cl A^\flat$ satisfy $\ip{w}{\mathbbm 1} = \gamma (\cl A^\flat)$.  
Then
$1 \ge \ip{w}{N(\cl A) \mathbbm 1} = N(\cl A) \gamma (\cl A^\flat)$, and so $\gamma(\cl A^\flat) \le \frac{1}{N(\cl A)}.$  For the reverse inequality, set $v= \frac{1}{\gamma(\cl A^\flat)} \mathbbm 1.$ For all $u \in \cl A^\flat$, we have $\ip{v}{u} = \frac{1}{\gamma(\cl A^\flat)} \ip{\mathbbm 1}{u} \le 1$. This shows that $v \in \cl A^{\flat \flat}$, and so $v \in \cl A$ by Theorem \ref{th_sabt}.  
Thus, $N(\cl A) \ge \frac{1}{\gamma(\cl A^\flat)},$ as required.

In the case where $\cl A= \bb R^d_+$, the statement holds with $M(\cl A)=0$, $N(\cl A)= \infty$ and $\gamma(\cl A^\flat)=0$.

Finally, suppose that $\cl A$ has empty interior. By Lemma \ref{nonemptyint}, there exists $i \in [d]$ such that 
$v_i = 0$ for all $v \in \cl A$ and hence $u_i$ can be arbitrarily large for $u \in \cl A^\flat$,
implying that $\cl A^\flat$ is unbounded.  
The statement thus holds with $M(\cl A) = \infty$, $N(\cl A) = 0$ and $\gamma(\cl A^\flat) = \infty$.
\end{proof}


\section {Convex corners and anti-blockers in $M_d$}\label{s_ccabMd}

\subsection{Definitions and basic properties}

We begin by defining several concepts that will play an essential role in the sequel.

\begin{definition} 
A non-empty subset $\cl A \subseteq M_d^+$ will be called a \emph{convex $M_d$-corner} 
(or just a \emph{convex corner} where the context allows), 
if $\cl A$ is closed, convex, and 
\begin{equation}\label{eq_herncc}
B\in \cl A, \ 0\leq A\leq B \ \Longrightarrow \ A\in \cl A.
\end{equation}  
A convex $M_d$-corner will be called \emph{standard} if it is bounded and has non-empty relative interior.
\end{definition}

We will refer to property (\ref{eq_herncc}) as \emph{hereditarity}. 

\begin{remark}\label{ccintersect}
The intersection of an arbitrary family of convex $M_d$-corners is a convex $M_d$-corner. 
Thus, given a non-empty subset $\cl G\subseteq M_d^+$, there exists a smallest 
convex corner ${\rm C}(\cl G)$ containing $\cl G$. 
\end{remark}

Recall that $(e_i)_{i=1}^d$ is the standard basis of $\bb{R}^d$, and 
let $\phi : \bb R_+^d \To \cl D_d^+$ be the one-to-one map given by 
\begin{equation} \label{phi} 
\phi \left( \sum_{i=1}^d \lambda_i e_i \right) = \sum_{i=1}^d \lambda_i e_i e_i^*.
\end{equation}

\begin{definition}  
A non-empty subset $\cl B \subseteq \cl D_d^+$ is 
called a \emph{diagonal convex $M_d$-corner}
(or simply a \emph{diagonal convex corner} when the context allows it),
if $\phi^{-1}(\cl B)$ is a convex $\bb R^d$-corner. 
A diagonal convex corner $\cl B \subseteq \cl D_d^+$ is called \emph{standard} if
$\phi^{-1}(\cl B)$ is standard.
\end{definition}

It is often convenient to identify
the convex $\bb R^d$-corner $\cl A$ with the diagonal convex $M_d$-corner $\phi(\cl A)$.

\begin{definition}  \label{dcc2}
Let $\cl A \subseteq M^+_d$ be a non-empty subset. The \emph{anti-blocker} of $\cl A$ is the set
$$\cl A^\sharp = \{N \in M_d^+ : \ip{N}{M} \le 1 \mbox{ for all }M \in \cl A \}.$$  
If $\cl B \subseteq \cl D_d^+$ is non-empty, its \emph{diagonal anti-blocker} is given by 
$\cl B^\flat := \cl D_d \cap \cl B^\sharp$. 
\end{definition}

\begin{remark}  \label{dcc1} 
{\rm It is clear that, if $\cl A \subseteq M_d^+$ is a convex corner then $\cl D_d \cap \cl A$ is a  diagonal convex corner. 
Theorem \ref{th_sabt} and the fact that 
$\cl B^{\flat} =  \phi\left(\phi^{-1}(\cl B)^{\flat}\right)$ 
implies that if $\cl B \subseteq M_d^+$ is a diagonal convex corner then $\cl B^{\flat \flat} = \cl B$.  }
\end{remark}

Let $\cl B \subseteq \cl D_d^+$ be a diagonal convex corner. We set 
$$\gamma( \cl B) := \gamma (\phi^{-1}(\cl B)), \ 
N(\cl B): = N(\phi^{-1}(\cl B)) \mbox{ and }  M(\cl B) := M (\phi^{-1}(\cl B)).$$
Note that
$$\gamma(\cl B)= \max \{ \Tr T : T \in \cl B\} \ \mbox{ and } \ N(\cl B)= \max \{ \beta : \beta I \in \cl B\}.$$
Proposition \ref{NMgamma} shows that, if $\cl B$ is a diagonal convex corner then
$$M(\cl B) = \frac{1}{N(\cl B)} = \gamma (\cl B^\flat).$$

\begin{definition}\label{d_reflcc}
A non-empty subset $\cl A \subseteq M_d^+$ is called \emph{reflexive} if $\cl A = \cl A^{\sharp \sharp}$.
\end{definition}

\begin{lemma} \label{ccinM_d} 
Let $\cl A$ and $\cl C$ be non-empty subsets of $M^+_d$ with $\cl A\subseteq \cl C$.
Then
\begin{enumerate} 
\item [(i)] $\cl C^\sharp \subseteq \cl A^\sharp$;  
\item [(ii)] $\cl A \subseteq \cl A^{\sharp \sharp}$;
\item [(iii)] $\cl A^\sharp$ is a reflexive convex $M_d$-corner;
\item[(iv)] If $\{\cl B_{\alpha}\}_{\bb A}$ is a non-empty family of non-empty subsets of $M_d^+$
then  
$\left(\cup_{\alpha \in \bb A} \cl B_\alpha \right)^\sharp = \cap_{\alpha \in \bb A} \cl B_\alpha^\sharp.$
\item[(v)] The intersection of a non-empty family of reflexive convex corners is a reflexive convex corner.  
\end{enumerate}
\end{lemma}

\begin{proof} 
Set $\cl B = \cl A^{\sharp}$. Properties (i) and (ii), as well as  
the convexity and the closedness of $\cl B$, are trivial.
Let $B \in \cl A^\sharp$ and $C\in M_d^+$ be such that $C\leq B$.
Then $\Tr (CA) \le \Tr (BA) \le 1$ for every $A\in \cl A$, and so $C \in \cl B$; thus,
$\cl B$ is hereditary and hence a convex corner.
By (ii),
$ \cl B \subseteq  \cl B^{\sharp \sharp}$.  
However, $\cl A \subseteq  \cl A^{\sharp \sharp}$ and so, by (i), 
$\cl B^{\sharp \sharp} = \cl A^{\sharp \sharp \sharp} \subseteq \cl B.$

(iv) For each $\beta \in \bb A$ we have $\cl B_{\beta} \subseteq \cup_{\alpha \in \bb A} \cl B_\alpha$, 
and (i) gives 
$\left( \cup_{\alpha \in \bb A} \cl B_\alpha \right)^\sharp \subseteq \cl B_{\beta}^\sharp$; thus, 
$\left( \cup_{\alpha \in \bb A} \cl B_\alpha \right)^\sharp \subseteq \cap_{\alpha \in \bb A} \cl B_\alpha^\sharp.$
The reverse inclusion is equally straightforward.

(v) Let $\bb{A}$ be a non-empty set, $\cl A_{\alpha} \subseteq M_d^+$ be a  reflexive convex corner, $\alpha \in \bb A$,
and $\cl A = \cap_{\alpha \in \bb A} \cl A_{\alpha}$. 
By Remark \ref{ccintersect}, $\cl A$ is a convex corner. 
By (iv) and the reflexivity of $\cl A_\alpha$, we have 
$$\cl A =  \cap_{\alpha \in \bb  A} \cl A_{\alpha}^{\sharp \sharp} =  
\left(\cup_{\alpha \in \bb  A} \cl A_{\alpha}^{\sharp} \right)^\sharp.$$ 
By (iii), $\cl A ^{\sharp \sharp} = \cl A$.  
\end{proof}

We isolate for future reference two straightforward statements. 

\begin{lemma} \label{PSD2} 
Let $\{v_i: i \in [d]\}$ be an orthonormal basis of  $\bb C^d$ and 
$M = \sum_{i,j =1}^d m_{i,j} v_iv^*_j$ be a positive matrix.
Then $$|m_{i,j}| \le \sqrt{m_{i,i} m_{j,j}} \le \max \{m_{i,i}, m_{j,j}\}.$$
Thus, if $m_{i,i}=0$ for some $i\in [d]$ then $m_{i,j} = m_{j,i} = 0$ for all $j \in [d]$.
\end{lemma}

\begin{lemma} \label{PSD4}  
The following are equivalent for a non-empty subset $ \cl A \subseteq M_d^+$:
\begin{enumerate} 
\item [(i)] the set $\cl A$ is bounded;
\item [(ii)] the set $\{\Tr M : M \in \cl A\}$ is bounded;
\item [(iii)] the set $\{\ip{u}{Mu}: M \in \cl A, u \in \bb C^d, \|u\|=1\}$ is bounded. 
\end{enumerate} 
\end{lemma}

\begin{lemma}\label{l_ri}
Let $\cl A\subseteq M_d^+$ be a convex corner. The following are equivalent:
\begin{itemize}
\item[(i)] $\cl A$ has a non-empty relative interior;
\item[(ii)] there exists $r > 0$ such that $rI \in \cl A$;
\item[(iii)] for every non-zero vector $v\in \bb{C}^d$ there exists $s > 0$ such that $s vv^* \in \cl A$;
\item[(iv)] $\cl A$ contains a strictly positive element.
\end{itemize}
\end{lemma}

\begin{proof}
(i)$\Rightarrow$(ii)
Let $A \in \cl A$ and $\delta > 0$ be such that ${\rm B}(A, \delta) \cap {M}_d^+ \subseteq \cl A$.
Then $A + \frac{1}{2\sqrt{d}}\delta I \in \cl A$;
since $\frac{1}{2\sqrt{d}}\delta I \le A + \frac{1}{2\sqrt{d}}\delta I$, we have that 
$\frac{1}{2\sqrt{d}}\delta I \in \cl A$.

(ii)$\Rightarrow$(iii)
Suppose that $r > 0$ is such that $r I \in \cl A$ and let $v\in \bb{C}^d$ be a non-zero vector.
Since $\frac{r}{\|v\|^2} vv^* \leq r I$, the hereditarity of $\cl A$ implies that $\frac{r}{\|v\|^2} vv^* \in \cl A$.

(iii)$\Rightarrow$(ii)
Let $\{v_i\}_{i=1}^d$ be an orthonormal basis of $\bb{C}^d$ and, for each $i\in [d]$, let $s_i > 0$ be such that
$s_i v_i v_i^*\in \cl A$. Since $\cl A$ is convex, $A=\sum_{i=1}^d \frac{s_i}{d} v_i v_i^*\in \cl A$.
Letting $s = \min_{i \in [d] } \frac{s_i}{d}$, we have that $s > 0$ and $sI \le A$; 
by hereditarity, $s I \in \cl A$.

(ii)$\Rightarrow$(iv) is trivial.

(iv)$\Rightarrow$(i) 
By hereditarity, there exists $r > 0$ such that $rI\in \cl A$. It follows that 
any $M\in M_d^+$ with $\|M\| < r$ is in $\cl A$. 
\end{proof}

\begin{remark} \label{dcc1s} 
{\rm 
Lemmas \ref{nonemptyint} and \ref{l_ri} imply that if $\cl A$ is a standard convex $M_d$-corner then 
$\cl D_d \cap \cl A$ is a standard diagonal convex corner. 
}
\end{remark}

\begin{prop} \label{cc1} 
Let $\cl A \subseteq M_d^+$ be a convex corner.
\begin{itemize}
\item[(i)] $\cl A$  has non-empty relative interior if and only if
$\cl A^\sharp$ is bounded;
\item[(ii)] $\cl A$ is bounded if and only if $\cl A^\sharp$ has non-empty relative interior;
\item[(iii)] $\cl A$ is standard if and only if $\cl A^\sharp$ is standard.
\end{itemize}
\end{prop}

\begin{proof} 
(i) If $\cl A$ has non-empty relative interior then, by Lemma \ref{l_ri}, $ r I \in \cl A$ for some $r >0$. 
Then $r \Tr M = \ip{M}{r I} \le 1$ for all $M \in \cl A^\sharp.$  Thus, $\Tr M \le 1/r$ for all  $M \in \cl A^\sharp$ and, by 
Lemma \ref{PSD4}, $\cl A^\sharp$ is bounded.

Suppose that $\cl A$ has empty relative interior.  
By Lemma \ref{l_ri}, $\cl A$ contains no strictly positive element.
Let $\cl E =\{A_i\}_{i\in \bb{N}}$ be a countable dense subset of $\cl A$.  
Write
$$\cl K_m = \left\{v \in \bb C^d : \|v\|=1, A_i v = 0, i \in [m]\right\}.$$  
It is clear that $\cl K_m$ is compact and $\cl K_{m+1} \subseteq   \cl K_m$, $ m \in \bb N$.  
Set $B_m = \frac{1}{m} \sum_{i=1}^m A_i$.  By convexity, $B_m \in \cl A$; by assumption, $B_m$ is not strictly positive.  
Thus there exists $v \in \bb C^d$ such that $B_m v = 0$, and
hence $A_i v = 0$ for all $i =1, \ldots, m$; in other words, $ \cl K_m$ is non-empty for all $m \in \bb N$. 
It follows that $\bigcap_{i=1}^\infty \cl K_i \neq \emptyset$, that is,
there exists a unit vector $v \in \bb C^d$ such that $Av = 0$ for all $A \in \cl E$.  
Since $\cl E$ is dense, $M v = 0$ for all $M \in \cl A$.
But then $\Tr (Mvv^*) = 0$ for all $ M \in \cl A$; thus,
$\lambda v v^* \in \cl A^\sharp$ for all $\lambda \ge 0$, showing that $\cl A^\sharp$ is unbounded.

(ii) 
If $\cl A^\sharp$ has non-empty relative interior then, by (i), 
$\cl A^{\sharp \sharp}$ is bounded.  By Lemma \ref{ccinM_d}, $\cl A$ is bounded.  
Conversely, suppose that $\cl A$ is bounded. By Lemma \ref {PSD4}, 
there exists $c > 0$ such that $\Tr M \le  c$ for all $M \in \cl A$.  
Thus $\ip{\frac{1}{c}I}{M} \le 1$ for all $M \in \cl A$, that is, $\frac{1}{c}I \in \cl A^\sharp$.
By  Lemma \ref{l_ri}, $\cl A^\sharp$ has non-empty relative interior.

(iii) is immediate from (i) and (ii). 
\end{proof}

\begin{definition} \label{def_her} 
Let $\cl B$ be a non-empty subset of $M_d^+$.
The \emph{hereditary cover} of $\cl B$ is the set
$$\her (\cl B) = \left\{M \in M_d^+: \mbox {there exists } N \in \cl B \mbox { such that } M \le N\right\}.$$  
\end{definition}

\begin{prop}\label{l_gen}
Let $\cl G \subseteq M_d^+$ be non-empty. The following hold:
\begin{itemize}
\item[(i)] 
If $\cl G$ is bounded then
${\rm C}(\cl G) = \her(\overline{\rm conv}(\cl G));$

\item[(ii)] 
$\cl G^\sharp = \her(\cl G)^\sharp = \rm C(\cl G)^\sharp$.
\end{itemize} 
\end{prop}

\begin{proof}
(i) 
Set $\cl A = \her (\overline{\rm conv}(\cl G))$. It is clear that $\cl A$ is a hereditary and bounded (non-empty) subset of $M_d^+$.
Let $A,B \in \cl A$, $\lambda\in [0,1]$, and choose 
$C,D \in \overline{\conv} (\cl G)$ with $A \le C$ and $B \le D.$  
Then $\lambda C +( 1- \lambda)D \in \overline{\conv}(\cl G)$;
since $\lambda A +(1- \lambda)B \le \lambda C +(1- \lambda)D$,
we have that $\lambda A +(1- \lambda)B \in \cl A$. 
It follows that $\cl A$ is convex. 

To show that $\cl A$ is closed, suppose that $(T_n)_{n\in \bb{N}}\subseteq \cl A$
and $T_n \to_{n\to\infty} T$. Let $C_n\in \overline{\rm conv}(\cl G)$ be such that $T_n\leq C_n$, $n\in \bb{N}$.
Since $\overline{\rm conv}(\cl G)$ is compact, $(C_n)_{n\in \bb{N}}$ has a cluster point $C$ in $\overline{\rm conv}(\cl G)$.
Then $T\leq C$ and hence $T \in \cl A$. 
Thus, $\cl A$ is a convex corner containing $\cl G$. Its minimality is straightforward. 

(ii) 
Since $\cl G \subseteq \her(\cl G) \subseteq \rm C(\cl G)$, Lemma \ref{ccinM_d} gives 
$$\rm C(\cl G)^\sharp \subseteq \her(\cl G)^\sharp \subseteq \cl G^\sharp.$$ 
Let $M \in \cl G^\sharp$ and 
$Q = \sum_{i =1}^n \lambda_i A_i$, with $A_i \in \cl G$ and  $\lambda_i \in \bb R_+$, $i\in [n]$, satisfying 
$\sum_{i=1}^n \lambda_i =1$. Then $\Tr (MQ) \le \sum_{i =1}^n \lambda_i=1$; thus, 
$M \in \overline{{\rm conv}}(\cl G)^{\sharp}$.
Finally, if $N'\in \overline{\rm conv}(\cl G)$ and $0\leq N\leq N'$ then 
$\Tr (MN) \le \Tr (MN')\leq 1$, and hence $M \in \rm C(\cl G)^\sharp$ as required. 
\end{proof}

We list some immediate consequences of Proposition \ref{l_gen} and Lemma \ref{l_ri}.

\begin{cor} \label{55555} 
Suppose $\cl G \subseteq M_d^+$ is  bounded and $\overline{\rm conv}(\cl G)$ contains a strictly positive element.  
Then ${\rm C}(\cl G)$ s a standard convex corner. 
\end{cor}

\begin{cor} \label{555555} 
If $\cl C$ is a diagonal convex corner then ${\rm C}(\cl C) = \her(\cl C)$. 
\end{cor}

\begin{cor} \label{dcc3} 
If $\cl C$ is a bounded (resp. standard)  diagonal convex corner, then $\her(\cl C)$ is a bounded (resp. standard) 
convex corner.  
\end{cor}


\subsection{Examples of convex $M_d$-corners}

In this subsection we consider some examples of convex $M_d$-corners that will be used subsequently. 
For  $C\in M_d^h$ and  $\lambda \in \bb R$, let
\begin{equation} \label{ccAX} 
\cl A_{C,\lambda} = \left\{M \in M_d^+ : \Tr(MC) \le \lambda \right\} 
\end{equation}  
and $\cl A_{C} = \cl A_{C,1}$.
Further, let 
$$\cl N_{C} = \left\{M \in M_d^+ : \Tr(MC) =0 \right\}$$
and   
\begin{equation} \label{ccAX3} 
\cl B_C = \{ M \in M_d^+ : M \le C \}. 
\end{equation}
It is clear that if $\lambda  > 0$ then $\cl A_{C,\lambda} = \cl A_{(1/\lambda)C}.$  
Note that, if $C \ge 0$ then $ \cl N_C = \cl A_{C,0}$.  

\begin{lemma} \label{ccinM_d2} 
Let $C \in M_d^+$ and $\lambda >0$. Then
\begin{itemize}
\item[(i)] $\cl B_C$ is a reflexive convex corner and $\cl B_C ^\sharp = \cl A_C$;
\item[(ii)] $\cl A_{C, \lambda}$ is a reflexive convex corner and $\cl A_{C, \lambda} ^\sharp = \cl B_{(1/ \lambda)C}$.
\end{itemize}
\end{lemma}

\begin{proof} 
(i) 
It is clear that $\cl B_C$ is a convex corner.  
Suppose that 
$0 \le M \le C$ and $N  \in \cl A_C$. 
Then $0 \le \Tr (MN) \le \Tr (CN) \leq 1$, and hence $N \in \cl B_C^\sharp$. 
Thus, $\cl A_C \subseteq \cl B_C ^\sharp$.
Conversely,  if $N \in \cl B_C^\sharp$, then $\Tr (CN) \le 1$, giving that $N \in \cl A_C$; 
thus, $\cl B_C ^\sharp = \cl A_C$. 

By Lemma \ref{ccinM_d}, in order to show that $\cl B_C$ is reflexive, 
it suffices to prove that  $\cl B_C^{\sharp\sharp}\subseteq \cl B_C$. 
Suppose that $Q \ge 0$ and $Q \notin \cl B_C$;  
then $C-Q \notin M_d^+$.  
Write 
$$C-Q = \sum_{i=1}^d \lambda_i v_i v_i^*,$$ 
where $\{ v_1, \ldots, v_d \}$ is an orthonormal basis of $\bb C^d$,
$\lambda_1, \ldots, \lambda_d \in \bb R$, $i \in [d]$, and $\lambda_j <0$ for some $j \in [ d]$.
Let $D= \alpha v_j v_j^*$ with $\alpha > 0$ to be fixed shortly.  
We have $D \ge 0$ and $\Tr (CD) = \alpha \Tr (C v_j v_j^*)= \alpha \ip{Cv_j}{v_j} \ge 0$ as $C \ge 0$.
On the other hand,
\begin{equation} \label{Tr}
\Tr\big( (C-Q)D\big) = \Tr \left(\sum_{i=1}^d \alpha \lambda_i (v_i v_i^*) (v_j v_j^*)\right) = \lambda_j \alpha. 
\end{equation}  
We will show that $Q \notin \cl B_C^{\sharp \sharp}$; we consider two cases:

\smallskip

\noindent {\it Case 1. } $\ip{Cv_j}{v_j} = 0$. 
Set $\alpha= -2/ \lambda_j$.  
Then  $\Tr (CD)=0$, and so $D \in \cl A_C= \cl B_C^{\sharp}.$   
By \eqref{Tr}, $\Tr (QD) = \Tr (CD) - \lambda_j \alpha = 2$, and hence 
$Q \notin \cl B_C^{\sharp \sharp}$.

\smallskip

\noindent {\it Case 2. } $\ip{Cv_j}{v_j} > 0$.  
Set $\alpha = \ip{Cv_j}{v_j}^{-1}$; then $\Tr (CD)=1 $ and so $D\in  \cl A_C =\cl B_C^\sharp$.  
On the other hand, 
$\Tr (QD) = \Tr (CD) - \lambda_j \alpha  >1$, and hence $Q \notin \cl B_C ^{\sharp \sharp}$, completing the proof of (i).

(ii) 
By (i), 
$$\cl A_{C, \lambda}= \cl A_{(1/ \lambda)C} = \cl B_{(1/ \lambda)C}^\sharp.$$
Applying anti-blockers and using (i), we get $\cl A_{C, \lambda} ^\sharp = \cl B_{(1/ \lambda)C}$.
\end{proof}

\begin{prop}
Let $C\in M_d^h$. 
\begin{itemize}
\item[(i)] If $C \in M_d^{++}$ then $\cl A_C$ and $\cl A_C^\sharp$ are standard convex corners;
\item[(ii)] If $C \in M_d^+ \backslash M_d^{++}$ then $\cl A_C$ and $\cl A_C^\sharp$ are  convex corners, but neither 
of them is standard;
\item[(iii)] If $-C \in M_d^+$, then $\cl A_C = M_d^+$;
\item[(iv)] If $\pm C \notin M_d^+$ then $\cl A_C$ is not a convex corner and $\cl A_C^\sharp= \{0\}$.
\end{itemize}
\end{prop}

\begin{proof}
By Lemmas \ref{ccinM_d} and \ref{ccinM_d2}, 
$\cl A_C$ and $\cl A_C^\sharp$ are convex corners.
Write $C= \sum_{i=1}^d\mu_i v_i v_i^*$ for some orthonormal basis $\{v_i, \ldots, v_d\}$ of eigenvectors of 
$C$ and some $\mu_1, \ldots, \mu_d \in \bb R$. 

(i)  
Let $M \in M_d^+$ and write $M= \sum_{i,j=1}^d m_{i,j}v_iv_j^* $.
Then $\Tr (MC) = \sum_{i=1}^d \mu_i m_{i,i}$.
If $M \in \cl A_C$ then $0 \le m_{i,i} \le \max_{j\in [d]} \frac{1}{\mu_j}$ for each $i\in [d]$, 
and hence $\cl A_C$ is a bounded convex corner by Lemma \ref{PSD4}. 
By Lemma \ref{ccinM_d2}, 
$\cl A_C^\sharp = \cl B_C$, and hence $\cl A_C^\sharp$  is bounded.  
By Proposition \ref{cc1}, $\cl A_C$ and $\cl A_C^\sharp$ have non-empty relative interiors.

(ii)  If $C \in M_d^+\backslash M_d^{++}$ then $\mu_i \ge 0$ for all $i\in [d]$ and  $\mu_j=0$ for some $j$.  
Then $\alpha v_j v_j^* \in \cl A_C$ for all $\alpha \ge 0$, and $\cl A_C$ is unbounded.   
By Lemma \ref{ccinM_d2}, Lemma \ref{l_ri} and 
Proposition \ref{cc1}, $\cl A_C^\sharp = \cl B_C$ has empty relative interior.

(iii)   In this case, $\mu_i \le 0$ for each $i\in [d]$ and hence 
$\Tr(MC) \le 0$ for all $M \in M_d^+$, giving $\cl A_C = M^+_d$. 
 
(iv) Write $C = \sum_{i=1}^d \mu_i v_iv^*_i$ with $\mu_j <0$ for some $j$ and $\mu_k>0$ for some $k$. 
Let $M = -\frac{2}{\mu_j}v_jv_j^*+ \frac{2}{\mu_k}v_kv_k^*$ and $N = \frac{2}{\mu_k}v_kv_k^*$. 
Then $0 \leq N \leq M$, $\Tr (MC) = 0$ and $\Tr (NC) =2$; thus, $M \in \cl A_C$ while
$N \notin \cl A_C$. It follows that the set $\cl A_C$ is not hereditary. 

Let $A \in M_d^+$ and $\lambda = \ip{A}{C}$.  If $\lambda \le 1$, then 
$A \in \cl A_C \subseteq \her(\cl A_C)$.  
If $\lambda > 1$, then $A \le A'$ where $A' := A - \frac{\lambda}{\mu_j} v_jv_j^*$ satisfies $\ip{A'}{C} = 0$.  
Thus $A' \in \cl A_C$ and therefore $A \in \her(\cl A_C)$, showing that $\her( \cl A_C)= M_d^+$. 
By Proposition \ref{l_gen}, $\cl A_C^\sharp = \{0\}$.
\end{proof}

We complete a similar analysis for the sets $\cl B_C$ and $\cl N_C$.  

\begin{prop} 
Let $C\in M_d^h$. 
\begin{itemize}
\item[(i)] If $C \in M_d^{++}$, then $\cl B_C$ is a reflexive standard convex corner;
\item[(ii)] If $C \in M_d^+ \backslash M_d^{++}$, then $\cl B_C$ is a reflexive convex corner with empty relative interior;
\item[(iii)] If $C \in M_d^h \backslash M_d^{+}$, then $\cl B_C = \emptyset$.  
\end{itemize}
\end{prop}  

\begin{proof} 
The set $\cl B_C$ is clearly bounded for any $C \in M_d^h$. 
By Lemma \ref{ccinM_d2}, if $C \ge 0$ then 
$\cl B_C$ is a reflexive convex corner satisfying $\cl B_C^\sharp = \cl A_C$.  

(i) By Lemma \ref{l_ri}, if $C \in M_d^{++}$ then $\cl B_C$ 
has non-empty relative interior and is hence a standard convex corner.

(ii) Let $u \in \bb C^d$ be a non-zero vector with $\ip{u}{Cu} = 0$.  
Then $\ip{u}{Mu} = 0$ whenever $M \in \cl B_C$.  
By Lemma \ref{l_ri}, $\cl B_C$ has empty relative interior.

(iii) is trivial.
\end{proof}

\begin{prop}\label{ccinM_d4}  
Let $C \in M_d^h$.
\begin{itemize}
\item[(i)] If $C \in M_d^+$ and $P$ is the projection onto $\ran(C)$ then 
\begin{equation}\label{eq_NCP}
\cl N_C = \left\{M \oplus  0_P: M \in \cl L(P^\perp \bb{C}^d)^+\right\}
\end{equation}
and
\begin{equation}\label{eq_NCPsha}
\cl N_C^\sharp = \left\{0_{P^\perp} \oplus N:  N \in \cl L(P\bb{C}^d)^+\right\};
\end{equation}
thus, $\cl N_C$ and $\cl N_C^{\sharp}$ are reflexive convex corners with empty relative interior;
\item[(ii)] If $C \in M_d^{++}$, then $\cl N_C = \{ 0 \}$;
\item[(iii)] If $\pm C \notin M_d^+$, then $\cl N_C$ is not a convex corner and $\cl N_C^\sharp = \{0\}$.
\end{itemize}
\end{prop}

\begin{proof} 
(i) 
It is clear that $\cl N_C$ is a convex corner.  
Write $C = \sum_{i=1}^d \lambda_i v_i v_i^*$ where $\{v_i\}_{i=1}^d$ is an orthonormal basis of $\bb C^d$;
by assumption, $\lambda_i \ge 0$, $i\in [d]$.  
For $M \in M_d^+$, write $M = \sum_{i,j=1}^d \alpha_{i,j}v_iv_j^*$, where $\alpha_{i,j} \in \bb C$, $i,j\in [d]$;
then $\Tr (MC) = \sum_{i=1}^d \lambda_i \alpha_{i,i}$.  
Suppose that $M \in \cl N_C$.  Then 
$\alpha_{i,i} = 0$ whenever $\lambda_i > 0$ and, by Lemma \ref{PSD2},  
$\alpha_{i,j}=0$ whenever $\lambda_i >0$ or $\lambda_j >0$. 
Thus, 
$$\cl N_C= \left \{ M \in M_d^+ : M = \sum_{i,j=1}^d \alpha_{ij}v_iv_j^* \mbox{ with } \alpha_{ij}=0 \mbox{ when } \lambda_i>0 \mbox{ or } \lambda_j>0 \right \}.$$
This shows (\ref{eq_NCP});  equation (\ref{eq_NCPsha}) is now straightforward.
Clearly, $\cl N_C^{\sharp\sharp} = \cl N_C$.

(ii) In this case, $P = I$ and the claim follows from (i). 

(iii) 
Write $C = \sum_{i=1}^d \lambda_i v_i v_i^*$, where $\{v_i\}_{i=1}^d$ is an orthonormal basis of $\bb{C}^d$ 
and $\lambda_i\in \bb{R}$, $i\in [d]$. 
Let $j,k\in [d]$ be such that 
$\lambda_j > 0$ and $\lambda_k < 0$. 
Let $M = \lambda_j v_kv_k^* - \lambda_k v_jv_j^*$. Note 
that $M \ge 0$ and $\Tr (MC) =0$, giving that $\alpha M \in \cl N_C$ for all $\alpha \ge 0$.  
Thus, $\cl N_C$ is unbounded.  
Since $M \ge \lambda_j v_kv_k^* \notin \cl N_C$, we have that $\cl N_C$ lacks hereditarity. 

Let $N = \sum_{r,s=1}^d \alpha_{r,s} v_r v_s^* \in \cl N_C^\sharp$, where 
$\alpha_{r,s} \in \bb C$, $r,s\in [d]$.  
Then $\alpha_{i,i} \in \bb R_+$, $i \in [d]$.  
We have that 
$$\alpha ( \lambda_j \alpha_{k,k}- \lambda_k \alpha_{j,j})  = \Tr (\alpha MN) \le 1, \ \ \ \alpha \ge 0;$$
thus, $\alpha_{j,j} = \alpha_{k,k} = 0$. 
It follows that $\alpha_{i,i}=0$ whenever $\lambda_i \neq 0$.  
On the other hand, if $\lambda_m = 0$ and $\alpha_{m,m} > 0$, then 
$\Tr\left(\frac{2}{\alpha_{m,m}} (v_mv_m^*) C\right) = 0$, 
and so $\frac{2}{\alpha_{m,m}}v_m v_m^* \in \cl N_C$.  
However, $\Tr\left(\frac{2}{\alpha_{m,m}} (v_mv_m^*) N\right) = 2$, a contradiction.
Thus, $\alpha_{m,m}=0$.  
By Lemma \ref{PSD2}, $N = 0$ and the proof is complete.
\end{proof}

\begin{remark} 
{\rm Note that $\cl N_{-C}= \cl N_C$, so the case  $-C \in M_d^+$ does not require separate consideration
in Proposition \ref{ccinM_d4}. }
\end{remark}


\section{Reflexivity of convex $M_d$-corners} \label{sect_AB}

In this section, we show the reflexivity of convex $M_d$-corners and note some of its consequences.


\subsection {The second anti-blocker theorem}
The next lemma is certainly well-known, but we include its proof for the convenience of the reader.

\begin{lemma} \label{range} 
Let $u_1, \ldots, u_n$ be linearly independent vectors in $\bb C^d$.   
Then 
$$\ran \left(\sum_{i=1}^n u_i u_i^*\right) = \Span \left \{u_1 ,\ldots, u_n\right\}.$$ 
\end{lemma}

\begin{proof} 
Set $M = \sum_{i=1}^n u_i u_i^*$, 
$\cl U = \Span \left \{ u_1 ,\ldots, u_n\right \}$ and, for $k \in [n]$, write
$\cl U_k = \Span \{u_i : i\neq k\}$. 
It is clear that $\ran(M) \subseteq \cl U$. 
Since $u_1 , \ldots, u_n$ are linearly independent, 
$\cl U_k\neq \cl U$.
Let $v_k$ be a non-zero vector in $\cl U\cap \cl U_k^{\perp}$.
Then 
$Mv_k = \langle u_k,v_k\rangle u_k$, and hence $u_k\in \ran(M)$.
\end{proof}

In the following, we fix a convex $M_d$-corner $\cl A$.  
Let 
\begin{equation} \label{eqclUv} 
\cl U = \left\{v \in \bb C^d:\mbox{there exists } r > 0 \mbox{ such that } rvv^* \in \cl A\right\}
\end{equation}  
and $P$ be the projection onto $\Span (\cl U)$.

\begin{lemma}  \label{2ndABc}  
The set $\cl U$ is a subspace. Moreover, there exists $r > 0$ such that $rP \in \cl A$.  
\end{lemma}  

\begin{proof} 
Let $\{u_i\}_{i=1}^k \subseteq \cl U$ be a linear basis of $\Span (\cl U)$.  
By the definition of $\cl U$, there exists $r_i > 0$ such that $r_iu_iu_i^* \in \cl A$, $i\in [k]$. 
Since $\cl A$ is convex, 
$R := \frac{1}{k} \sum_{i=1}^k r_i u_iu_i^* \in \cl A$.  
Letting $r_0 = \frac{1}{k} \min_{i \in [k]}  r_i$ and $Q = r_0 \left (\sum_{i=1}^k u_iu_i^* \right) $, we have  $0 \le Q \le R$.  
By hereditarity, $Q \in \cl A$.   By Lemma \ref{range}, 
$\ran (Q) = \ran(P)$.  
Let $r$ be the smallest positive eigenvalue of $Q$. 
Then $rP \le Q$ and hence $rP \in \cl A$, again by hereditarity.  

Suppose that $u \in \Span (\cl U)$; there exists $t > 0$ such that $uu^* \le tP$.  
By the previous paragraph, 
$rP \in \cl A$ and thus $\frac{r}{t}uu^* \in \cl A$.  
It follows that $u \in \cl U$, and so $\cl U = \Span (\cl U)$.  
\end{proof}

\begin{lemma} \label{2ndABd} 
Let $\cl A$ be a convex corner. The following hold: 
\begin{itemize} 
\item [(i)]  $PMP = M$ for every $M \in \cl A$; 
\item [(ii)] $\ip{M}{P^\perp}=0$ for all $M \in \cl A$;  
\item [(iii)] $\ip{M}{P^\perp} > 0$ for all $M \in M_d^+$ satisfying $PMP \ne M$.  
\end{itemize} 
\end{lemma}

\begin{proof}   
(i) follows from the fact that if $v \in \bb C^d$ is an eigenvector of $M \in \cl A$ 
corresponding to a positive eigenvalue then $v \in \cl U$.

(ii) is a direct consequence of (i). 

(iii) 
Suppose that $M\in M_d^+$ and $PMP \ne M$. Then
$M$ has an eigenvector $v \notin \cl U$ whose eigenvalue $\lambda$ is positive; 
note that $P^\perp v \ne 0$.
Thus, 
$$\ip{M}{P^{\perp}} \geq \lambda \ip{vv^*}{P^\perp}
= \lambda  \|P^\perp v \|^2  > 0.$$ 
\end{proof}

Set $k = \rank(P)$ and let
$$M_d^P= \{ M \in M_d^+: PMP = M \mbox{ and } \rank (M)= k\}.$$ 
Note that 
\begin{equation}\label{eqMdP}  
M_d^P= \{ M \in M_d^+: \mbox{ there exist } s>r>0 \mbox{ such that } rP \le M \le sP \}.
\end{equation}

\begin{lemma} \label{2ndABe2}  
Let $\cl A \subseteq M_d^+$ be a convex corner and let $P$ be the projection onto $\cl U$ as defined in \eqref{eqclUv}.  Set
$$ \cl A_0= \{ A \in M_d^P \cap \cl A: (1+\ep) A \notin \cl A \mbox{ for all } \ep >0\}.$$
There exists a set 
$\left\{R_A \in M_d^+ :  A \in \cl A_0\right\}$ 
such that
$$\cl A = \bigcap_{A \in \cl A_0}  \cl A_{R_A} \cap \cl N_{P^\perp}.$$  
\end{lemma}  

\begin{proof}
Let $A \in \cl A_0$ and $A_n = \left(1 + \frac{1}{n}\right)A$;
thus, $A_n \notin \cl A$, $n \in \bb N$.  
By the Hahn-Banach Theorem, 
there exist $Q_{A_n} \in M_d$ and $\gamma \in \bb R$ such that  
$${\rm Re} \ip{M}{Q_{A_n}}  < \gamma < {\rm Re}  \ip{A_n}{Q_{A_n}}, \ \ \  M \in \cl A.$$  
After replacing $Q_{A_n}$ by $\frac{1}{2}(Q_{A_n} + Q_{A_n}^*)$, we may assume that 
$Q_{A_n} \in M_d^h$. 
Since $0 \in \cl A$, we have that $\gamma > 0$. 
After further replacing $Q_{A_n}$ by $\frac{1}{\gamma}Q_{A_n}$, we may assume that $\gamma = 1$, that is, 
\begin{equation} \label{2ndAB3'}  
\ip{M}{Q_{A_n}} \le 1 < \ip{A_n}{Q_{A_n}}, \ \ \  M \in \cl A.  
\end{equation}    

Note that $A, A_n \in M_d^P$, $n\in \bb{N}$.
By Lemma \ref{2ndABd} (i), 
$$\ip{M}{Q_{A_n}} = \ip{PMP}{Q_{A_n}} = \ip{M}{PQ_{A_n}P}, \ \ \ M \in \cl A;$$ 
similarly, $\ip{A_n}{Q_{A_n}}=\ip{A_n}{PQ_{A_n}P}$.  
We may thus assume that $Q_{A_n} = PQ_{A_n}P$ and hence that
the eigenvectors of $Q_{A_n}$, corresponding to non-zero eigenvalues, are contained in $\cl U$.

Fix $A \in \cl A_0$.
We claim that the set $\left\{Q_{A_n}: n \in \bb N\right\}$ is bounded.  
Write $Q_{A_n} = \sum_{i=1}^k \lambda^{(n)}_i v_i^{(n)}v_i^{(n)*}$, where 
$\left\{ v_i^{(n)}: i \in [d]\right\}  \subseteq \cl U$ is an orthonormal set.  
Using Lemma \ref{2ndABc}, let $r > 0$ be such that $rP \in \cl A$; by hereditarity,
$r v_i^{(n)}v_i^{(n)*} \in \cl A$ for all $i \in [k]$ and all $n \in \bb N$.  
By \eqref{2ndAB3'}, 
$\ip{rv_i^{(n)}v_i^{(n)*} }{Q_{A_n}} \le 1$, and so   
\begin{equation} \label{UB1'} 
\lambda_i^{(n)} \le \frac{1}{r}, \ \ \  i \in [k], \ n \in \bb N. 
\end{equation}  
By \eqref{2ndAB3'}, 
$$\left(1+\frac{1}{n}\right)\ip{A}{Q_{A_n}} > 1, \ \ \ n \in \bb N,$$ 
and hence 
\begin{equation} \label{2ndAB4'} 
\sum_{i=1}^k \lambda_i^{(n)} \ip{A}{v_i^{(n)}v_i^{(n)*}} = \ip{A}{Q_{A_n}} > \frac{1}{2}, \ \ \ n \in \bb N. 
\end{equation}
Since $A  \in M_d^P$, there exists $t > 0$ such that $A \ge tP$.
We have that 
$$t \le \ip{A}{v_i^{(n)}v_i^{(n)*}} \le \|A\|, \ \ \  i \in [k], \ n \in \bb N.$$  
Suppose $\lambda_j^{(n)} < 0$. 
Then (\ref{UB1'}) and (\ref{2ndAB4'}) give 
\begin{eqnarray*}
\lambda_j^{(n)} 
& \geq & 
\frac{\lambda_j^{(n)}}{t} \ip{A}{v_j^{(n)}v_j^{(n)*}} 
>
\frac{1}{2t} - \sum_{i \ne j} \frac{\lambda_i^{(n)}}{t} \ip{A}{v_i^{(n)}v_i^{(n)*}}\\ 
& \ge & \frac{1}{2t} - \frac{d-1}{rt} \|A\|
> 
- \frac{d-1}{rt} \|A\|.
\end{eqnarray*}
Together with \eqref{UB1'}, this shows that $\left\{ \lambda_i^{(n)}: i \in [k], n \in \bb N\right\}$ is bounded, 
and hence the set $\{ Q_{A_n}: n \in \bb N\}$ is bounded as claimed. 

Let $R_A \in M_d^h$ be a cluster point of the sequence $(Q_{A_n})_{n \in \bb N}$;
clearly, 
\begin{equation}\label{RAP} 
R_A= PR_AP.
\end{equation}
By \eqref{2ndAB3'}, 
\begin{equation} \label{2ndAB5'} 
\ip{M}{R_A} \le 1, \ \ \  M \in \cl A,
\end{equation}  
and 
\begin{equation} \label{2ndAB5''} 
\ip{\left( 1+ \frac{1}{n} \right)A}{Q_{A_n}} >1, \ \ \ n \in \bb N.
\end{equation}
Since $A \in \cl A$, \eqref{2ndAB5'} and \eqref{2ndAB5''} 
show that 
\begin{equation}\label{eq_ARA}
\ip{A}{R_A}=1.
\end{equation}

We claim that $R_A \ge 0$ for all $A \in \cl A_0.$ 
Suppose, towards a contradiction, 
that there exists $A \in \cl A_0$ for which $R_A$ has an eigenvalue $\lambda < 0$. 
By (\ref{RAP}), an associated unit eigenvector $v$ of $\lambda$ lies in $\cl U$. 
Since $A \in M_d^{P}$, there exists $t > 0$ with $A \ge tP$, and hence 
$0 \le A -tvv^* \le A$,
giving $A - tvv^* \in \cl A$ by hereditarity. 
However, $$\ip{A-tvv^*}{R_A}=1- \lambda t >1,$$ contradicting \eqref{2ndAB5'}.

Set $\cl C = \bigcap_{A \in \cl A_0} \cl A_{R_A} \cap \cl N_{P^\perp}$.  
We complete the proof by showing that $\cl C = \cl A$. 
By \eqref{2ndAB5'}, $\cl A \subseteq \cl A_{R_A}$ for all $A \in \cl A_0$.  
By Lemma \ref{2ndABd}, $\cl A \subseteq \cl N_{P^\perp}$, and thus $\cl A \subseteq \cl C$. 
Fix $M \notin \cl A$; we will show that $M \notin \cl C$.  
Let $r > 0$ be such that $rP \in \cl A$ (such $r$ exists by Lemma \ref{2ndABc}).
We identify four cases.

\smallskip

\noindent {\it  Case 1. } $M \notin M_d^+.$
Since $\cl C \subseteq M_d^+$, we have $M \notin \cl C.$

\smallskip

\noindent {\it  Case 2. } $M \in M_d^{P}$.
Let $\mu = \max \{\lambda \in \bb R_+ : \lambda M \in \cl A\}$.  
By \eqref{eqMdP}, $0 < \mu < 1$. 
Setting $A = \mu M$ we have $A \in \cl A_0$.  Then
$\cl C \subseteq \cl A_{R_{A}}$.
By (\ref{eq_ARA}), $\ip{M}{R_{A}} = \frac{1}{\mu} >1,$ and so $M \notin \cl C.$

\smallskip

\noindent {\it  Case 3. } 
$M = PMP \in M_d^+ \backslash M_d^{P}$.
Since the sets $\cl A_{R_A}$ and $ \cl N_{P^\perp}$ are convex, $\cl C$ is convex.  
By Case 2,
\begin{equation} \label{case3note} 
M_d^{P} \cap \cl A= M_d^{P} \cap \cl C. 
\end{equation}  
Suppose, towards a contradiction, that $M \in \cl C$. 
Letting $M_n = \left(1 - \frac{1}{n}\right)M + \frac{r}{n}P$, 
the convexity of $\cl C$ gives that  $M_n \in \cl C$ for all $n \in \bb N$.  
Since $M = PMP \ge 0$ and $rP \in M_d^P$, 
we have that $M_n  \in M_d^P$ for all $n \in \bb N$.
By \eqref{case3note}, $M_n \in \cl A$, $n \in \bb N$. 
Since $M_n \to_{n\to\infty} M$ and $\cl A$ is closed, $M \in \cl A$, the required contradiction.

\smallskip

\noindent {\it  Case 4. } $M \in M_d^+$, and $ PMP \ne M.$
By Lemma \ref{2ndABd} we have  $M \notin \cl N_{P^\perp}$, and hence $M \notin \cl C$.  
\end{proof}

We can now prove the non-commutative version of Theorem \ref{th_sabt}.

\begin{theorem} \label{2ndABe} 
A non-empty set $\cl A \subseteq M_d^{+}$ is reflexive if and only if $\cl A$ is a convex corner.  
\end{theorem}

\begin{proof} 
Let $\cl A$ be a convex corner. 
By Lemma \ref{2ndABe2}, $\cl A$ is the intersection of 
convex corners of the form $\cl A_R$ and $\cl N_P$, where $R$ is positive and $P$ is a projection. 
By Proposition \ref{ccinM_d4} and Lemma \ref{ccinM_d2}, such $\cl A_R$ and $\cl N_P$ are 
reflexive. Lemma \ref{ccinM_d} now implies that $\cl A$ is reflexive. 
Conversely, if $\cl A$ is reflexive then $\cl A = \cl A^{\sharp\sharp}$ and now 
Lemma \ref{ccinM_d} shows that $\cl A$ is a convex corner. 
\end{proof}

Theorem \ref{2ndABe} and Lemma \ref{ccinM_d} have the following immediate consequence.

\begin{cor}\label{reflexiveiff} 
If $\cl A$ and $\cl B$ are convex $M_d$-corners then 
\begin{itemize}
\item[(i)] 
$\cl A \subseteq \cl B$ if and only if $\cl A^\sharp \supseteq \cl B^\sharp$;
\item[(ii)]
$\cl A = \cl B$ if and only if  $\cl A^\sharp = \cl B^\sharp$;  
\item[(iii)]
$\cl A \subsetneq \cl B$ if and only if $\cl A^\sharp \supsetneq \cl B^\sharp$.
\end{itemize}
\end{cor}


\subsection{Consequences of reflexivity}\label{ss_conref}

In this subsection we give some corollaries of the reflexivity of convex $M_d$-corners.

\begin{theorem} \label{ccinM_d7} 
Let $\cl A$ be a non-empty subset of $M_d^+$. 
Then $\rm C (\cl A) = \cl A^{\sharp\sharp}$.  
\end{theorem}

\begin{proof} 
By Proposition \ref{l_gen}, $\cl A^\sharp = \rm C (\cl A)^\sharp$; 
Theorem \ref{2ndABe} yields 
$\cl A^{\sharp\sharp} = \rm C (\cl A)^{\sharp \sharp} = \rm C(\cl A)$.
\end{proof}

\begin{cor}  \label{2315} 
If $\cl A \subseteq M_d^+$ is a diagonal convex corner then $\cl A^{\sharp\sharp} = \her(\cl A)$. 
\end{cor}

\begin{proof}  
If $\cl A$ is a diagonal convex corner, then $\overline{\conv} (\cl A) = \cl A$.  
By Proposition \ref{l_gen}, $\rm C (\cl A) = \her(\cl A)$, and now the claim 
follows from Theorem \ref{ccinM_d7}.  
\end{proof}

\begin{prop} \label{capsharp} 
Let $\bb A$ be a non-empty set and $\cl B_{\alpha}$ be a  convex corner, $\alpha \in \bb A$.  
Then
$$\left( \bigcap_{\alpha \in \bb A} \cl B_{\alpha} \right) ^\sharp 
= \rm C \left( \bigcup_{\alpha \in \bb A} \cl B_{\alpha}^\sharp \right).$$ 
\end{prop}

\begin{proof}  
By Theorems \ref{2ndABe} and \ref{ccinM_d7} and Lemma \ref{ccinM_d}, 
$$\left( \bigcap_{\alpha \in \bb A} \cl B_{\alpha} \right) ^\sharp
=\left( \bigcap_{\alpha \in \bb A} \cl B_{\alpha}^{\sharp\sharp} \right) ^\sharp
= \left(\bigcup_{\alpha \in \bb A} \cl B_{\alpha}^\sharp\right)^{\sharp \sharp}
= \rm C \left( \bigcup_{\alpha \in \bb A} \cl B_{\alpha}^\sharp \right).$$    
\end{proof}

By analogy with convex $\bb R^d$-corners, we introduce several parameters for convex $M_d$-corners.
Recall that a set $(P_i)_{i=1}^k \subseteq M_d$ of projections is called a 
\emph{projection-valued measure (PVM)} 
if $\sum_{i=1}^k P_i = I$. 
Let $\cl A$ be a convex $M_d$-corner.

\smallskip

{\bf (a)} \ If $\cl A$ is bounded, let 
$$\gamma (\cl A) = \max \left\{ \Tr A : A \in \cl A\right\};$$
If $\cl A$ is unbounded, set $\gamma (\cl A)= \infty$.  

\smallskip

{\bf (b)} \ If $\cl A \ne M_d^+$, let  
$$N( \cl A) = \max \left\{ \beta  : \beta I \in \cl A\right\}.$$
We set $N( M_d^+) = \infty.$ 

\smallskip

{\bf (c)} \  If  $\cl A$ has non-empty relative interior, let 
$$M(\cl A) = \inf \left \{ \sum_{i=1}^k \lambda_i: \exists \ k \in \bb N,\, A_i \in \cl A,\, \lambda_i > 0, i\in [k], \mbox{  s.t. } 
\sum_{i=1}^k \lambda_i A_i \ge I \right \}.$$
If $\cl A$ has empty relative interior, set $M(\cl A)= \infty$.   


\smallskip

{\bf (d)} \ If $\cl A_{I_d}\subseteq \cl A$, let
$$\Gamma(\cl A) = \min\left\{k\in \bb{N} : \mbox{ there exists a PVM } (P_i)_{i=1}^k \subseteq \cl A\right\};$$
otherwise, set $\Gamma(\cl A) = \infty$;

\smallskip

{\bf (e)} \ If $\cl A_{I_d}\subseteq \cl A$, let
$$\Gamma_{\rm f}(\cl A) = 
\inf \left \{ \sum_{i=1}^k \lambda_i : \exists \ k \in \bb N, \mbox{ proj. } P_i \in \cl A, \lambda_i > 0,  \mbox{  s.t. } 
\sum_{i=1}^k \lambda_i P_i \ge I \right \};$$ 
otherwise, set $\Gamma_{\rm f}(\cl A) = \infty$.

\begin{remark}\label{r_spcnd}
{\rm 
\begin{itemize}
\item[(i)] We have that $\gamma(\cl A)=0$ if and only if $\cl A = \{0\}$;
\item[(ii)] By Lemma \ref{l_ri}, $N(\cl A) = 0$ if and only if $\cl A$ has empty relative interior;
\item[(iii)] The parameter $\Gamma_{\rm f}$ can be thought of as a real relaxation of $\Gamma$. In particular, it is clear
that $\Gamma_{\rm f}(\cl A)\leq \Gamma(\cl A)$. 
\end{itemize}
}
\end{remark}


\begin{theorem} \label{NMgamma2} 
Let $\cl A$ be a convex $M_d$-corner, 
$\cl P \subseteq M_d$ be a non-empty set of non-zero projections and $\cl B = {\rm C}(\cl P)$. 
Then
\begin{itemize}
\item[(i)] 
$M(\cl A) = \inf \left \{ \mu \in \bb{R}_+ : \exists \ A \in \cl A \mbox{  s.t. } \mu A \ge I \right \}$;

\item[(ii)] 
$M(\cl A) = \frac{1}{N(\cl A)} = \gamma(\cl A^\sharp)$;

\item[(iii)] 
$M(\cl B) = \Gamma_{\rm f}(\cl B)$; 

\item[(iv)] 
$\Gamma(\cl B)\gamma(\cl B) \geq d$.
\end{itemize}
\end{theorem}

\begin{proof}
The proof of (i) is similar to that of Lemma \ref{MAmin}, and the proof of (ii) to that of
Proposition \ref{NMgamma}, using Theorem \ref{2ndABe} instead of Theorem \ref{th_sabt}.

(iii) Since $\cl P\subseteq \cl B$, we have that $M(\cl B) \leq \Gamma_{\rm f}(\cl B)$.
Set $R = \Gamma_{\rm f}(\cl B)$. 
Let $\ep > 0$, $\lambda \in \bb{R}_+$ and $A \in \cl B$ be such that 
$$\lambda A \ge I  \ \mbox{ and } \ \lambda \le M(\cl B) + \ep.$$
Let $\delta > 0$ be such that $1-\lambda \delta > 0$.
By Proposition \ref{l_gen}, there exists $B \in \overline{\conv}(\cl P)$ such that 
$A \le B$ and, hence, a sequence 
$(B^{(j)})_{j \in \bb N} \subseteq \conv ( \cl P)$ such that $B^{(j)} \To_{j \To \infty} B$.
Let $n \in \bb N$ be such that $B^{(n)} + \delta I \ge B$.
Then $\lambda B^{(n)} \ge (1 - \lambda \delta) I$ and hence
\begin{equation} \label{eq5tf}  
\frac{\lambda}{1 - \lambda \delta} B^{(n)} \ge I.
\end{equation}
Write 
$B^{(n)} = \sum_{l = 1}^{m} \mu_l P_{l}$ with $P_{l} \in \cl P$ and 
$\mu_l \in \bb R_+$ satisfying $\sum_{l = 1}^{m} \mu_l = 1$.
By \eqref{eq5tf},   
$$R \le \sum_{l=1}^{m} \frac{1}{1 - \lambda \delta}\, \lambda \mu_l \le \frac{1}{1-\lambda  \delta} \big(M(\cl B) + \ep \big).$$  
Letting $\delta\to 0$, we obtain $R \le M(\cl B) + \ep$; 
letting $\ep \to 0$, we conclude that $R \le M(\cl B)$.

(iv) Suppose that $(P_i)_{i=1}^k$ is a PVM contained in $\cl B$. 
Then 
$$d = \sum_{i=1}^k {\rm rank} (P_i) \leq k \gamma(\cl B).$$
Minimising over $k$ implies the statement. 
\end{proof}

We next show the continuity of the anti-blocker. 
We use a classical concept of convergence due to Kuratowski. 
Let $\cl X$ be a topological space. For a sequence $(F_n)_{n\in \bb{N}}$ of subsets of $\cl X$, 
set 
$$\liminf_{n\in \bb{N}} F_n = 
\left\{\lim\mbox{}_{n\to\infty} x_n : 
(x_n)_{n\in \bb{N}} \in \Pi_{n\in \bb{N}}  F_n
\mbox{ a convergent sequence}\right\}$$
and 
$$\limsup_{n\in \bb{N}} F_n = 
\left\{x :  \mbox{ a cluster point of a sequence }
(x_n)_{n\in \bb{N}} \in \Pi_{n\in \bb{N}}  F_n\right\}.$$
We say that the sequence $(F_n)_{n\in \bb{N}}$ converges to the subset $F\subseteq \cl X$, 
and write $F = \lim_{n\to\infty} F_n$, 
if $F = \liminf_{n\in \bb{N}} F_n = \limsup_{n\in \bb{N}} F_n$.

\begin{prop}\label{p_ccc}
Let $\cl A, \cl A_n$, $n\in \bb{N}$, be 
convex $M_d$-corners such that $\cup_{n\in \bb{N}}\cl A_n$ is bounded.
\begin{itemize}
\item[(i)]
$\limsup_{n\in \bb{N}} \cl A_n \subseteq \cl A$ if and only if
$\cl A^{\sharp}\subseteq \liminf_{n\in \bb{N}} \cl A_n^{\sharp}$;
\item[(ii)]
$\cl A\subseteq \liminf_{n\in \bb{N}} \cl A_n$ if and only if
$\limsup_{n\in \bb{N}} \cl A_n^{\sharp} \subseteq \cl A^{\sharp}$;
\item[(iii)]
$\cl A = \lim_{n\in \bb{N}} \cl A_n$ if and only if
$\cl A^{\sharp} = \lim_{n\in \bb{N}} \cl A_n^{\sharp}$.
\end{itemize}
\end{prop}

\begin{proof}
(i)-(ii) 
By \cite[Lemma 6.9]{BTW}, 
\begin{equation}\label{eq_limsup}
\limsup_{n\in \bb{N}} \cl A_n \subseteq \cl A \ \Longrightarrow \ 
\cl A^{\sharp}\subseteq \liminf_{n\in \bb{N}} \cl A_n^{\sharp}.
\end{equation}
Suppose, on the other hand, that 
$\cl A\subseteq \liminf_{n\in \bb{N}} \cl A_n$. 
Let $(B_{n_k})_{k\in \bb{N}}\subseteq M_d^+$ be a sequence with limit $B$ 
such that $B_{n_k}\in \cl A_{n_k}^{\sharp}$, $k\in \bb{N}$. 
Let $A\in \cl A$, and $(A_n)_{n\in \bb{N}}\subseteq M_d^+$ be a 
sequence, such that $A_n\in \cl A_n$, $n\in \bb{N}$, and $\lim_{n\to\infty} A_n = A$. 
Then 
$$\langle B,A\rangle = \lim_{k\to\infty} \langle B_{n_k},A_{n_k}\rangle \leq 1,$$
and thus $B\in \cl A^{\sharp}$. 
Hence 
\begin{equation}\label{eq_liminf}
\cl A\subseteq \liminf_{n\in \bb{N}} \cl A_n  \ \Longrightarrow \ 
\limsup_{n\in \bb{N}} \cl A_n^{\sharp} \subseteq \cl A^{\sharp}.
\end{equation}
Now suppose that 
$\cl A^{\sharp}\subseteq \liminf_{n\in \bb{N}} \cl A_n^{\sharp}$.
By (\ref{eq_liminf}) and Theorem \ref{2ndABe}, 
$$\limsup_{n\in \bb{N}} \cl A_n = 
\limsup_{n\in \bb{N}} \cl A_n^{\sharp\sharp} \subseteq \cl A^{\sharp\sharp} = \cl A.$$
Similarly, if 
$\limsup_{n\in \bb{N}} \cl A_n^{\sharp} \subseteq \cl A^{\sharp}$ then, 
by (\ref{eq_limsup}) and Theorem \ref{2ndABe}, 
$$
\cl A = \cl A^{\sharp\sharp}\subseteq \liminf_{n\in \bb{N}} \cl A_n^{\sharp\sharp}
= \liminf_{n\in \bb{N}} \cl A_n.$$

(iii) is immediate from (i) and (ii).
\end{proof}

\begin{cor}\label{c_cpa}
\begin{itemize}
\item[(i)]
The parameters $M$, $N$ and $\gamma$ are continuous on bounded sets of convex corners. 

\item[(ii)]
Let $\cl P$ and $\cl P_n$ be non-empty sets of projections in $M_d$, 
$\cl B = {\rm C}(\cl P)$, and $\cl B_n = {\rm C}(\cl P_n)$, $n\in \bb{N}$.
Suppose that $\cl B$ has non-empty relative interior. 
If $\limsup_{n\in \bb{N}}\cl P_n \subseteq \cl P$ (resp. $\cl P \subseteq \liminf_{n\in \bb{N}} \cl P_n$) 
then $\Gamma_{\rm f}(\cl B) \leq \liminf_{n\in \bb{N}} \Gamma_{\rm f}(\cl B)$
(resp. $\limsup_{n\in \bb{N}} \Gamma_{\rm f}(\cl B)\leq \Gamma_{\rm f}(\cl B)$). 
\end{itemize}
\end{cor}

\begin{proof}
(i) By Propositions \ref{NMgamma2} and \ref{p_ccc}, it suffices to show the continuity of $N$. 
Suppose that $\cl A, \cl A_n$, $n\in \bb{N}$, are convex $M_d$-corners such that $\cup_{n\in \bb{N}}\cl A_n$ is bounded
and $\limsup_{n\in \bb{N}} \cl A_n \subseteq \cl A$. 
Let $\mu_n = N(\cl A_n)$; then $\mu_n I\in \cl A_n$, $n\in \bb{N}$. Selecting a convergent subsequence 
$(\mu_{n_k})_{k\in \bb{N}}$ with limit $\mu$, the assumption implies that $\mu I\in \cl A$, and hence 
$N(\cl A)\geq \mu$, showing that $\limsup_{n\in \bb{N}} N(\cl A_n) \leq N(\cl A)$. 

Now suppose that $\cl A\subseteq \liminf_{n\in \bb{N}} \cl A_n$ and let $\mu\in \bb{R}_+$ be such that 
$\mu I \in \cl A$. Let $A_n\in \cl A_n$, $n\in \bb{N}$, be such that $A_n\to_{n\to\infty} \mu I$.
By the continuity of the spectrum, there exist $\mu_n\in \bb{R}_+$ with $\mu_n I \leq A_n$, $n\in \bb{N}$, 
such that $\mu_n\to_{n\to\infty} \mu$. 
It follows that $N(\cl A)\leq \liminf_{n\in \bb{N}} N(\cl A_n)$.

(ii) 
By Theorem \ref{NMgamma2} and the proof of (i), it suffices to show that that 
$$\cl P \subseteq \liminf_{n\in \bb{N}}\cl P_n  \Rightarrow \cl B \subseteq \liminf_{n\in \bb{N}}\cl B_n
\mbox{ and } 
\limsup_{n\in \bb{N}}\cl P_n \subseteq \cl P  \ \Rightarrow \ \limsup_{n\in \bb{N}}\cl B_n \subseteq \cl B.$$
Suppose that $\cl P \subseteq \liminf_{n\in \bb{N}}\cl P_n$. It is clear that 
$$\overline{{\rm conv}}(\cl P) \subseteq \liminf_{n\in \bb{N}}\overline{{\rm conv}}(\cl P_n).$$ 
Suppose that $0 < A\leq B$ for some $B\in \overline{{\rm conv}}(\cl P)$, and let $B_n\in \overline{{\rm conv}}(\cl P_n)$, 
$n\in \bb{N}$, converge to $B$. Then 
$A_n := B_n - (B-A) \to_{n\to \infty} A$, $A_n \leq B_n$ for each $n$ and,
eventually, $A_n \geq 0$. It follows that $A\in \liminf_{n\in \bb{N}}\cl B_n$. 
Since $\cl B$ has non-empty interior, 
Lemma \ref{l_ri} shows that any $A\in \cl B$ is the limit of strictly positive elements of $\cl B$, 
and the first implication is proved. 

Suppose that $\limsup_{n\in \bb{N}}\cl P_n \subseteq \cl P$. 
Using the Carath\'{e}odory Theorem, we can express
every element of ${\rm conv}(\cl P_n)$ as a convex combination of at most 
$2d^2 + 1$ elements of $\cl P_n$. It readily follows that
$\limsup_{n\in \bb{N}}{\rm conv}(\cl P_n) \subseteq \overline{{\rm conv}}(\cl P)$, 
and hence $\limsup_{n\in \bb{N}}\overline{{\rm conv}}(\cl P_n) \subseteq \overline{{\rm conv}}(\cl P)$.
Let $A_k\in \cl B_{n_k}$, $k\in \bb{N}$, converge to $A\in M_d^+$, and $B_k\in \overline{{\rm conv}}(\cl P_{n_k})$, 
with $A_k\leq B_k$, $k\in \bb{N}$. 
Passing to a subsequence if necessary, we can assume that $(B_k)_{k\in \bb{N}}$ converges to 
an element $B$ of $\overline{{\rm conv}}(\cl P)$. Now $A\leq B$ and hence $A\in \cl B$. 
\end{proof}


\section{Non-commutative lifts} \label{sssde}

In this section, we discuss the connection between 
convex $\bb{R}^d$-corners and convex $M_d$-corners. We show that, 
a given convex $\bb{R}^d$-corner has two extremal quantisations and 
establish several results that will be used in the next section.

For an orthonormal basis $V= \{v_1, \ldots, v_d\}$ of $\bb C^d$, 
we let 
$$\cl D_V= \Span\left\{v_iv_i^*: i \in [d]\right\}$$ 
be the algebra of matrices diagonal with respect to $V$.  
We write $\cl D_V^+=\cl D_V \cap M_d^+$, and set 
$$\Delta_V(A) =  \sum_{i=1}^d \ip{Av_i}{v_i}v_iv_i^*, \ \ \  A\in M_d;$$   
thus, $\Delta_V : M_d\to \cl D_V$ is the \emph{diagonal expectation} with respect to $V$.
We write $\Delta$ for the diagonal expectation with respect to the canonical basis $\{e_1, \ldots, e_d\}$.

\begin{definition}\label{d_nclift}
Let $\cl A$ be a diagonal convex corner in $M_d$. 
The convex $M_d$-corner $\cl B$ is called a \emph{non-commutative lift} of $\cl A$ if 
$\Delta(\cl B) = \cl D_d\cap \cl B = \cl A$.
\end{definition}

\begin{remark} \label{diagon} 
{\rm 
Let $V = \{v_1, \ldots , v_d\}$ be an orthonormal basis of $\bb{C}^d$. The following hold:
\begin{itemize}  
\item [(i)] 
If $M,N \in M_d$ then $\Tr ((\Delta_V(M)) N) = \Tr(M \Delta_V (N))$;   
\item [(ii)] 
If $\cl B$ is a convex $M_d$-corner, we have that 
$\cl D_V \cap \cl B = \Delta_V (\cl B)$ if and only if $\Delta_V (\cl B) \subseteq \cl B$, if and only if 
$\Delta_V ( \cl B^\sharp) \subseteq  \cl B^\sharp$;  
\item [(iii)]
If $\cl B$ is a convex $M_d$-corner and $\cl A = \Delta(\cl B)$, then $\gamma(\cl A)= \gamma (\cl B)$. 
\end{itemize} 
}
\end{remark}

\begin{proof}
(i) is straightforward.

(ii) The first equivalence is trivial. Assume $\Delta_V (\cl B) \subseteq \cl B$ and let $B \in \cl B^\sharp$.
Then
$\ip{\Delta_V(B)}{A} = \ip{B}{\Delta_V(A)} \le 1$ for all $A \in \cl B$.  
This shows that $\Delta_V(B) \in \cl B^\sharp$, and hence $\Delta_V ( \cl B^\sharp) \subseteq  \cl B^\sharp$.  
The converse implication now follows from Theorem \ref{2ndABe}.  

(iii)
We have 
$$\gamma( \cl B) 
= 
\max \{ \Tr T : T \in \cl B \} = \max \{ \Tr (\Delta (T)): T \in \cl B\} = \gamma (\cl A).$$
\end{proof}

\begin{lemma} \label{Deltasharp1}  
Let $V$ be an orthonormal basis and $\cl B$ be a non-empty subset of $M_d^+$. 
\begin{itemize}
\item[(i)]
If $\cl D_V \cap \cl B = \Delta_V (\cl B)$ then
\begin{equation} \label{AB3} 
\cl D_V \cap \left(\Delta_V (\cl B)\right)^\sharp = \cl D_V \cap \cl B^\sharp =  \Delta_V(\cl B^\sharp).
\end{equation}  
\item[(ii)]
Suppose that $\cl B$ is a convex corner. Then 
$\cl D_V \cap \cl B = \Delta_V (\cl B)$ if and only if $\cl D_V \cap \cl B^\sharp =  \Delta_V \left(\cl B^\sharp\right)$.
\end{itemize}
\end{lemma}

\begin{proof} 
(i) 
Write $\cl A = \Delta_V (\cl B)$, and suppose that $\cl D_V \cap \cl B = \cl A$. 
Then $\cl A \subseteq \cl B$; thus, $ \cl B^\sharp \subseteq  \cl A^\sharp$ and so 
$\cl D_V \cap  \cl B^\sharp \subseteq \cl D_V \cap \cl A^\sharp$.   
Let $T \in \cl D_V \cap \cl A^\sharp$ and $N \in  \cl B$. 
Using Remark \ref{diagon}, we have 
$$\Tr (TN) = \Tr \left((\Delta_V (T))N\right) = \Tr \left(T\Delta_V (N)\right) \le 1,$$ 
and so $T \in \cl D_V \cap \cl B^\sharp$.  Thus  $\cl D_V \cap \cl A^\sharp \subseteq  \cl D_V \cap  \cl B^\sharp,$ 
and  the first equality in \eqref{AB3} is proved.

Let $R \in  \cl B^\sharp$, $M = \Delta_V (R)$ and $Q \in  \cl B$.  
By assumption, $\Delta_V (Q) \in  \cl B$ and hence
$$\Tr \left(MQ\right) = \Tr \left((\Delta_V (R))Q\right) = \Tr \left(R \Delta_V(Q)\right) \le 1.$$ 
Thus, $M \in  \cl B^\sharp$ and so $\Delta_V (\cl B^\sharp) \subseteq   \cl B^\sharp$; 
\eqref{AB3} now follows from Remark \ref{diagon}.

(ii)
Suppose that $\cl B$ is a convex corner such that 
$\cl D_V \cap \cl B^\sharp =  \Delta_V(\cl B^\sharp)$. 
By (i), 
$\cl D_V \cap \cl B^{\sharp \sharp} =  \Delta_V ( \cl B^{\sharp \sharp})$; 
now Theorem \ref{2ndABe} implies $\cl D_V \cap \cl B= \Delta_V (\cl B)$. 
\end{proof}

\begin{prop} \label{Deltasharp}
Let $\cl A$ be a diagonal convex corner, and $\cl B$ be a convex corner, in $M_d$. 
The following are equivalent:
\begin{itemize}
\item[(i)] 
$\cl B$ is a non-commutative lift of $\cl A$;
\item[(ii)]  
$\cl B^{\sharp}$ is a non-commutative lift of $\cl A^{\flat}$. 
\end{itemize}
\end{prop}

\begin{proof}
(i)$\Rightarrow$(ii) 
By Lemma \ref{Deltasharp1}, 
$\cl D_d \cap \cl B^\sharp =  \Delta(\cl B^\sharp)$, while the equality 
$\cl A^\flat = \cl D_d \cap  \cl B^\sharp$ is immediate from the definitions of the anti-blockers.

(ii)$\Rightarrow$(i) 
By the previous paragraph,
$\cl A^{\flat \flat} = \cl D_d\cap \cl B^{\sharp\sharp} = \Delta(\cl B^{\sharp\sharp})$. 
The claim now follows from Theorem \ref{2ndABe}. 
\end{proof}

\begin{lemma} \label{flatsharp=} 
Let $\cl A$ be a non-zero diagonal convex corner in $M_d$. Then
\begin{itemize} 
\item[(i)]
$(\cl A^\flat)^\sharp = \left\{M \in M_d^+ : \Delta (M) \in \cl A\right\};$

\item[(ii)]
If $\cl A$ is bounded
then 
$$\her(\cl A) \subseteq (\cl A^\flat)^\sharp=(\her (\cl A^\flat))^\sharp$$  
and, if $d > 1$, the inclusion is proper.
\end{itemize}
\end{lemma}

\begin{proof} 
(i) 
Let $A \in \cl A^\flat$ and $M \in M_d^+$ with $\Delta (M) \in \cl A$.  
Then 
$$\ip{A}{M}= \ip{\Delta (A)}{M}= \ip{A}{\Delta (M)} \le 1$$ 
and so $M \in (\cl A^\flat)^\sharp$.
On the other hand, suppose that $M \in (\cl A^\flat)^\sharp$. 
If $A\in \cl A^{\flat}$ then 
$$\ip{\Delta(M)}{A} = \ip{M}{\Delta(A)} = \ip{M}{A}\leq 1;$$
thus, $\Delta(M)\in \cl A^{\flat\flat}$. By Remark \ref{dcc1}, 
$\Delta (M) \in \cl A$.

(ii)
By Proposition \ref{l_gen}, $(\cl A^\flat)^\sharp = (\her(\cl A^\flat))^\sharp$. 
Clearly, $\cl A^\flat \subseteq \cl A^\sharp$;
Corollary \ref{2315} and Lemma \ref{ccinM_d} imply
$$\her(\cl A)  = \cl A^{\sharp \sharp} \subseteq  (\cl A^\flat)^\sharp.$$
We show that if $d > 1$ then $\her(\cl A) \subsetneq \her (\cl A^\flat)^\sharp$.  
By Corollary \ref{reflexiveiff}, it suffices to show that 
$\her(\cl A)^\sharp \supsetneq \her (\cl A^\flat)$. 
By Proposition \ref{l_gen}, $\her(\cl A)^\sharp = \cl A^\sharp$ so, to prove 
the latter inequality, we seek $M \in \cl A^\sharp$ such that $M \notin \her (\cl A^\flat)$.  
By assumption, $\cl A \ne \{0\}$; thus,
$\cl A^\sharp \ne M_d^+$ and so $N(\cl A ^\sharp) \ne \infty.$  
Since $\cl A$ is bounded,  by Proposition \ref{cc1}, $\cl A^\sharp$ has non-empty relative interior and, 
by Lemma \ref{l_ri}, $N( \cl A^\sharp) > 0$.  
Set $\mu = N(\cl A^\sharp)$. 
Since $\cl A^\flat = \cl D_d \cap \cl A^\sharp$, we have that 
\begin{equation}\label{maxmu}
\mu = N(\cl A^\flat).
\end{equation}
If $M \in M_d$ and $A \in \cl A$ then, by Remark \ref{diagon}, 
$\Tr (MA) = \Tr (M\Delta (A))= \Tr (\Delta (M)A)$.  
Thus, if $M \ge 0$ and $\Delta (M) \in \cl A^\sharp$ then $M \in \cl A^\sharp.$  
It follows that, if $J$ is the matrix in $M_d$ with all entries equal to one, then 
$\mu J \in \cl A^\sharp$.   
We show that $\mu J \notin \her(\cl A^\flat)$. 
By way of contradiction, suppose that
$$\mu J \le N = \begin{pmatrix} \mu_1 & 0 & \ldots & 0 \\ 0 & \mu_2 & \ldots & 0 \\ \vdots & \vdots & \ddots & \vdots \\ 0 & 0 & \ldots & \mu_d  \end{pmatrix} \in \cl A^\flat.$$  
Let $Q = (q_{i,j})_{i,j} = N - \mu J$; then $\mu_i \ge \mu$, $i \in [d]$.  
But if $\mu_i = \mu$, then $q_{i,i}=0$ and Lemma \ref{PSD2} implies that 
$-\mu  = q_{i,j} = q_{j,i} = 0$ for all $j \ne i$, contradicting the fact that $\mu > 0$.   
Thus there exists $\epsilon > 0$ such that $\mu_i \ge \mu + \epsilon$, $i \in [d]$.  
Then $(\mu + \epsilon)I \le N$; by hereditarity, $(\mu + \epsilon)I \in \cl A^\flat$ contradicting  \eqref{maxmu}.  
\end{proof}

We can now prove the main result of this section. 
It provides a characterisation of the non-commutative lifts of a given diagonal convex corner, 
showing that there are two extreme such lifts which, in the case where $d > 1$, do not  coincide.

\begin{theorem} \label{diagdelta} 
Let $\cl A$ be a diagonal convex corner,
$\cl B_1= \her(\cl A)$ and $\cl B_2 = (\cl A^\flat)^\sharp$. Then $\cl B_1$ and $\cl B_2$ are 
convex $M_d$-corners. Moreover,  
the following are equivalent for a convex $M_d$-corner $\cl B$: 
\begin{itemize}
\item[(i)] $\cl B$ is a non-commutative lift of $\cl A$;
\item[(ii)] $\cl B_1 \subseteq \cl B \subseteq \cl B_2$.  
\end{itemize}
\end{theorem}

\begin{proof} 
By Corollary \ref{555555} and Lemma \ref{ccinM_d}, $\cl B_1$ and $\cl B_2$ are convex corners. 

(ii)$\Rightarrow$(i) 
Trivially, $\cl A \subseteq \cl D_d \cap \cl B_1 \subseteq \Delta (\cl B_1)$.  
Let $T \in \cl B_1$ and $N \in \cl A$ be such that $0 \le T \le N$.  
Then
$0 \le \Delta (T) \le \Delta  (N) = N$.   
It follows that $\Delta (T) \in \cl A$ by the hereditarity of $\cl A$.    
Thus, 
\begin{equation}\label{eq_B1}
\cl A = \Delta (\cl B_1) = \cl D_d \cap \cl B_1. 
\end{equation}

Using reflexivity and Proposition \ref{l_gen}, we have
$\cl B_2^\sharp = \her(\cl A^\flat)$.
By the previous paragraph, 
$\cl A^\flat = \Delta(\cl B_2^\sharp)= \cl D_d \cap \cl B_2^\sharp$  
and, by Proposition \ref{Deltasharp}, 
\begin{equation}\label{eq_B2}
\cl A = \Delta (\cl B_2) = \cl D_d \cap \cl B_2. 
\end{equation} 
Equations (\ref{eq_B1}) and (\ref{eq_B2}) imply that any convex $M_d$-corner $\cl B$ with 
$\cl B_1 \subseteq \cl B \subseteq \cl B_2$ is a non-commutative lift of $\cl A$.

(i)$\Rightarrow$(ii) 
Suppose that $\cl B$ is a non-commutative lift of $\cl A$. 
By the hereditarity of $\cl B$, we have $\cl B_1 \subseteq \cl B$.  
By Proposition \ref{Deltasharp}, $\cl A^\flat = \cl D_d \cap \cl B^\sharp.$  Thus 
$\cl A^\flat \subseteq \cl B^\sharp$ and so $\her( \cl A^\flat) \subseteq \cl B^\sharp$ by the hereditarity of $\cl B^\sharp$.  
By Theorem \ref{2ndABe}, $\cl B= \cl B^{\sharp \sharp} \subseteq \cl B_2$, as required.
\end{proof}

\begin{remark}\label{r_AnkBk}
{\rm 
The condition $\Delta_V (\cl A) \subseteq \cl A$, or equivalently, $\Delta_V (\cl A) = \cl D_V \cap \cl A$, 
has appeared multiple times. 
We give some examples of convex corners for which it is satisfied.

\smallskip

(i) 
Let $N \in \cl D_V^+$ and $\lambda > 0$. 
Then $\Delta_V (\cl A_{N,\lambda} )\subseteq \cl A_{N,\lambda}$.  
Indeed, 
if $A \in \cl A_{N,\lambda}$ then, by Remark \ref{diagon}, 
$\ip{\Delta_V(A)}{N} = \ip{A}{\Delta_V(N)}= \ip{A}{N} \le \lambda$.  
The positivity of $N \in \cl D_V$ cannot be omitted; 
for example, if $N = \begin{pmatrix} 1 &-1 \\ -1 &1 \end{pmatrix}$ 
then $J = \begin{pmatrix} 1 & 1 \\ 1 &1 \end{pmatrix} \in \cl A_{N,1}$ but $\Delta(J) = I \notin \cl A_{N,1}.$   

(ii) 
Let $M \in \cl D_V^+$. Then $\Delta_V (\cl B_M )\subseteq \cl B_M$. 
As in (i), the condition $M \in \cl D_V^+$ is essential; for example, 
$J \in \cl B_M$, but $\Delta (J) = I \notin \cl B_M$. 
}
\end{remark}


\section{Entropy with respect to a convex corner} \label{sssmde}

In this section, we define the entropy of a quantum state with respect to a convex $M_d$-corner. 
Our motivation stems from the classical case, and parallels with it are drawn as we go along. 
We obtain non-commutative versions of several fundamental results about the entropy of a probability distribution 
with respect to a convex $\bb{R}^d$-corner \cite{Csis, Marton, Simonyi1}. 
Applications of those will be made in the subsequent sections.


\subsection{Background}

We let $$\cl R_d = \{\rho \in M_d^+: \Tr \rho = 1\}$$
be the (closed convex) set of all states in $M_d$, and recall that 
$\cl P_d$ stands for all probability distributions on $[d]$. Note that, up to a canonical 
identification, $\cl P_d = \cl R_d\cap \cl D_d$. 
If $A = \sum_{i=1}^d \lambda_i u_iu_i^*\in M_d^{++}$, 
where $\{u_1, \ldots, u_d\}$ is an orthonormal basis of $\bb C^d$ (and $\lambda_i > 0$, $i\in [d]$), 
the logarithm $\log A$ of $A$ is given by 
$$\log A = \sum_{i=1}^d (\log \lambda_i )u_iu_i^*;$$
it is clear that $\log A\in M_d^h$. 

Let $\rho, A\in M_d^+$, and write 
$A = \sum_{i=1}^d \lambda_i u_iu_i^*$, where $\{u_1, \ldots, u_d\}$ is an orthonormal basis of $\bb C^d$ and 
$\lambda_i\geq 0$, $i\in [d]$. Set
$$\Tr(\rho \log A) = 
\begin{cases}
\sum_{i=1}^d \ip{\rho u_i}{u_i} \log \lambda_i & \text{if } \ker(A) \subseteq \ker(\rho)\\
-\infty & \text{otherwise}
\end{cases}$$
(we recall the conventions $0 \log 0=0$ and $\log 0= -\infty$). 
The quantity $H(\rho) := - \Tr(\rho \log \rho)$  is the 
\emph{von Neumann entropy} of an element $\rho \in M_d^+$. 
Given $\rho, \sigma \in M_d^+$, the 
\emph{relative quantum entropy} of $\rho$ with respect to $\sigma$ is the quantity

$$D(\rho \| \sigma) = 
\begin{cases}
\Tr(\rho \log \rho) - \Tr(\rho \log \sigma) & \text{if } \ker (\sigma) \subseteq \ker (\rho)\\
+\infty & \text{otherwise.}
\end{cases} $$
We recall some basic properties of $D(\rho \| \sigma)$ that can be found as 
\cite[Theorem 11.9.2]{Wilde}, \cite[p.250]{Wehrl},  \cite[Theorem 7]{Ruskai2},
\cite[p.251]{Wehrl} and \cite{aud}.

\begin{lemma} \label{vn+}  
\begin{itemize}
\item[(i)] 
If $\rho, \sigma \in \cl R_d$ then
$D(\rho \| \sigma) \ge 0$ and equality holds if and only if $\rho = \sigma$; 

\item[(ii)] 
If $\rho = \sum_{k=1}^m \lambda_k \rho^{(k)} \in M_d^+$ and 
$\sigma = \sum_{k=1}^m \lambda_k \sigma^{(k)} \in M_d^+$, 
where $\lambda_k > 0$ and $\sum_{k=1}^m \lambda_k = 1$,
satisfy 
$\ker (\sigma^{(k)}) \subseteq \ker (\rho^{(k)})$, then 
$$D\left(\rho \| \sigma\right) \le \sum_{k=1}^m \lambda_k D\left(\rho^{(k)} \| \sigma^{(k)}\right).$$  
If $\rho^{(k)}, \sigma^{(k)} \in M_d^{++}$, $k \in [m]$, 
equality holds if and only if $ \log \rho - \log \sigma = \log \rho^{(k)} - \log \sigma^{(k)}$ for all $k\in [m]$;


\item[(iii)] 
For a fixed $\rho \in \cl R_d$, the function $\sigma \to D(\rho \| \sigma)$, from 
 $M_d^+$ to the extended real line, is convex and lower semi-continuous.
\end{itemize}
\end{lemma}

We next state a form of the well-known von Neumann minimax theorem that will be needed in the sequel.
A proof of this version of the theorem can be obtained along the lines of \cite{Pollard}, 
and can be found in \cite{bore-thesis}.

\begin{theorem} \label{Sion} 
Let $K$ be a convex, compact subset of a normed vector space $X$, and let $C$ be a convex subset of vector space $Y$.  
Let $f: K \times C \To \bb R \cup \{ \infty \}$ be a function, satisfing the conditions
\begin{enumerate} 
\item [(i)]  $x \To f(x,y)$  is convex and lower semi-continuous for each $y \in C,$ and  
\item [(ii)] $y \To f(x,y)$ is concave for each $x \in K.$ 
\end{enumerate}
Then 
$$\inf_{x \in K}  \sup_{y \in C}  f(x,y) = \sup_{y \in C}  \inf_{x \in K}  f(x,y).$$
\end{theorem}


\subsection{Quantisation of entropy}

We use the notation of  
\eqref{ccAX} and \eqref{ccAX3} to write  the  $M_d$-unit corner as  
$\cl A_{I_d} = \{ T \in M_d^+: \Tr T \le 1 \}$, and the $M_d$-unit cube as 
$\cl B_{I_d} = \{ T \in M_d^+: T \le I \}$.

\begin{lemma} \label{minA2} 
Let $\rho \in \cl R_d$ and $\cl A$ be a bounded convex $M_d$-corner.
The function $f : \cl A\rightarrow \bb R \cup \{+\infty\}$, given by $f(A)= -\Tr (\rho \log A)$, 
attains a minimum value $f(A_0)$ for some $A_0 \in \cl A$.  
If $\rho > 0$ and $\cl A$ has non-empty relative interior then $A_0$ is unique 
and $f(A_0) < + \infty$. 
\end{lemma}  

\begin{proof} 
Let  
$$\cl A^0(\rho) = \left\{A \in \cl A:  \ker(A)\subseteq \ker(\rho)\right\}.$$ 
By Lemma \ref{vn+}, $f$ is lower semi-continuous, and since $\cl A$ is compact, it attains a 
minimum. 
Suppose that $\rho > 0$ and $\cl A$ has non-empty relative interior. 
Then $\cl A^0 (\rho) = \cl A \cap M_d^{++}$; by Lemma \ref{l_ri}, $\cl A^0 (\rho) \ne \emptyset$. 
Since $A_0\in \cl A^0 (\rho)$, we have that $f(A_0) < + \infty$. 

Assume, towards a contradiction, that there exist distinct $A_0, B_0 \in \cl A^0 (\rho)$ 
satisfying $f(A_0) = f(B_0) = \min_{A \in  \cl A} f(A)$.  
Since 
$$f(A) = D( \rho \| A) - \Tr (\rho \log \rho), \ \ \ A\in \cl A,$$
Lemma \ref{vn+} (ii) implies
$$f \left(\frac{A_0 +B_0}{2}\right) < \frac{1}{2} f(A_0) + \frac{1}{2} f(B_0) = \min_{A \in  \cl A} f(A).$$
Since $\cl A$ is convex, $\frac{A_0+B_0}{2} \in \cl A$, yielding a contradiction.  
It follows that the minimum is achieved for a unique $A_0 \in \cl A$.  
\end{proof}

\begin{definition} \label{minAr} 
Let $\cl A$ be a bounded convex $M_d$-corner and  $\rho \in \cl R_d$ be a state.  
The parameter
$$H_{ \cl A}(\rho) = \min_{A \in \cl A} -\Tr \rho  \log A$$  
is called the \emph{entropy} of $\rho$ over $\cl A$.
\end{definition}

Let $\cl A$ be a convex $\bb{R}^d$-corner and $p\in \cl P_d$. The entropy of $p$ with respect to $\cl A$
was introduced in \cite{Csis} as the quantity
$$H_{ \cl A}(p) = \min\left\{\sum_{i=1}^d p_i \log\frac{1}{v_i} : v\in \cl A, v > 0\right\}.$$  
Thus, the parameter 
$H_{ \cl A}(\rho)$, introduced in Definition \ref{minAr}, can be viewed as a non-commutative version of $H_{ \cl A}(p)$.
This viewpoint will be made more rigorous in Theorem \ref{noncomm15} below.

\begin{remark}\label{HAmin}
(i) It is clear that, if 
$\cl A$ and $\cl B$ are convex $M_d$-corners and 
$\cl A \subseteq \cl B$ then $H_{\cl A}(\rho) \ge H_{\cl B}(\rho)$, $\rho \in \cl R_d$.  

\smallskip

\noindent (ii)
Let $\rho\in \cl R_d$ and $\cl A$ be a convex $M_d$-corner. 
We have that $H_{\cl A}(\rho) = +\infty$ if and only if $\cl A^0(\rho) = \emptyset$.

\smallskip

\noindent (iii)
If $\cl A$ has empty relative interior
then, by Lemma \ref{l_ri}, $\cl A$ has no strictly positive element, 
and there exists $\rho \in \cl R_d$, for example the maximally mixed state $\frac{1}{d}I$, 
such that $\cl A^0(\rho) = \emptyset$. 
In this case, $H_{\cl A}(\rho) = +\infty$.  
On the other hand, if $\cl A$ has non-empty relative interior
then $\cl A^0(\rho) \ne \emptyset$ and, by (ii), 
$H_{\cl A}(\rho)$ is finite for every $\rho \in \cl R_d$.

\smallskip

\noindent (iv) 
If $\cl A$ is a standard convex corner then 
\begin{equation}\label{eq_for++}
H_{ \cl A}(\rho) = \inf_{A \in \cl A^{++}} -\Tr \rho  \log A.
\end{equation}
Indeed, by (iii), $H_{ \cl A}(\rho)$ is finite and hence 
there exists a minimiser $A$ for $H_{ \cl A}(\rho)$ in $\cl A^0(\rho)$. 
Setting $A_n = \left(1-\frac{1}{n}\right)A + \frac{1}{n} I$, we have that 
$A_n\in \cl A^{++}$, $n\in \bb{N}$, and 
$$\lim_{n\to\infty} {\rm Tr}(\rho \log(A_n)) = {\rm Tr}(\rho \log(A)),$$
implying (\ref{eq_for++}). 

\smallskip

\noindent (v) 
Fix $\rho\in \cl R_d$.
It is not difficult to see that the minimising element of 
$\cl A_{I_d}$ in the definition of $H_{\cl A_{I_d}}(\rho)$ has unit trace.
Thus, 
$$H_{\cl A_{I_d}}(\rho) = \min_{\sigma \in \cl R_d} - \Tr (\rho \log \sigma)$$
and hence it coincides with the von Neumann entropy 
$H(\rho)$ of $\rho$ (see e.g. \cite{opetz}). 

\smallskip

\noindent (vi) 
Since the elements of $\cl B_{I_d}$ have eigenvalues in the interval $[0,1]$, 
we have  $-\Tr (\rho \log A) \ge 0$ for all $A \in \cl B_{I_d}$.  
Thus, 
$H_{\cl B_{I_d}}(\rho) = \Tr(\rho \log I) = 0$.

\smallskip

\noindent (vii) 
By (i), (iv) and (v), 
\begin{equation} \label{genent} 
0 \le H_{\cl A}(\rho) \le H(\rho) \mbox{ whenever } \cl A_{I_d} \subseteq \cl A \subseteq \cl B_{I_d}. 
\end{equation}
There exist convex $M_d$-corners $\cl B$ and $\cl C$ 
satisfying $H_{\cl B}(\rho) < 0$ and $H_{\cl C}(\rho) > H(\rho)$ for all $\rho \in \cl R_d$. 
For an example, let $\lambda > 1$ and $\cl B = \lambda\cl B_{I_d}$; 
then 
$$H_{\cl B}(\rho) = H_{\cl B_{I_d}}(\rho) - \log \lambda = -\log \lambda < 0.$$
Similarly, if $\cl C = \frac{1}{\lambda} \cl A_{I_d}$ then 
$$H_{\cl C}(\rho) = H_{\cl A_{I_d}}(\rho) + \log \lambda = H(\rho) + \log \lambda > H(\rho).$$
\end{remark}

For the next theorem, note that, if $\cl A$ is a standard convex corner then $N(\cl A) > 0$ and hence, 
by Theorem \ref{NMgamma2}, the logarithms 
in its statement are well-defined. 

\begin{theorem} \label{noncomm13} 
Let $\cl A$ be a standard convex $M_d$-corner. Then
$$\max_{\rho \in \cl R_d}  H_{\cl A} (\rho) = -\log N(\cl A) = \log M(\cl A) = \log \gamma (\cl A^\sharp).$$
\end{theorem}

\begin{proof}
Note that $\cl R_d$ and $\cl A$ are compact and convex subsets of $M_d^+$.
Let $g : \cl R_d \times \cl A \To \bb R\cup \{+\infty\}$ be the function, given by
$g(\rho,A) = -\Tr(\rho \log A)$.  
For a fixed $A \in \cl A$, the function $\rho \to g(\rho, A)$ is linear, and hence concave.   
On the other hand, 
$g(\rho, A) = D( \rho \| A) - \Tr \rho \log \rho$ and so, by Lemma \ref{vn+},  
for a fixed $\rho \in \cl R_d$, the function $A \to g(\rho,A)$ 
is convex and lower semi-continuous.  

Let $\lambda_{{\rm min}}(A)$ denote the smallest eigenvalue of a positive matrix $A$ and set 
$\mu = \sup_{A\in \cl A} \lambda_{\min}(A)$. 
Since $N(\cl A)I \in \cl A$, we have that $\mu \ge N(\cl A)$.
On the other hand, for every  $\ep > 0$, 
there exists $A \in \cl A$ such that $\mu - \ep <  \lambda_{{\rm min}}(A)$ and hence $(\mu - \ep) I \le A$.  
By hereditarity, $(\mu - \ep)I \in \cl A$.  
Thus $N(\cl A) \ge \mu - \ep$ for all $\ep > 0$, and so $N(\cl A) \ge \mu$.  
Hence $\mu = N (\cl A)$.  
Using Theorem \ref{Sion}, we now have
\begin{eqnarray*}
\max_{\rho \in \cl R_d} H_{\cl A}(\rho) 
& = & 
\sup_{\rho \in \cl R_d} \inf_{A \in \cl A} g(\rho, A)
= 
\inf_{A \in \cl A}\sup_{\rho \in \cl R_d}  g(\rho, A)\\
& = & 
\inf_{A \in \cl A^{++}}\sup_{\rho \in \cl R_d}  g(\rho, A)
= 
\inf_{A \in \cl A^{++}}\sup_{\rho \in \cl R_d} \Tr (\rho \log A^{-1})\\
& = &  
\inf_{A \in \cl A^{++}} \left\|\log A^{-1}\right\| 
= 
\inf_{A \in \cl A^{++}} - \log \lambda_{\min}(A)\\
& = & 
- \log \left({\sup_{A \in \cl A^{++}} \lambda_{{\rm min}}(A)}\right)
= 
- \log \left({\sup_{A \in \cl A} \lambda_{{\rm min}}(A)}\right)\\
& = & 
-\log N(\cl A).
\end{eqnarray*} 
The remaining equalities follow from Theorem \ref{NMgamma2}.
\end{proof}

Recall that $\phi : \bb R_d^+ \to \cl D_d^+$ is the canonical bijection, given by \eqref{phi}.


\begin{theorem} \label{noncomm15} 
Let $\cl A$ be a standard diagonal convex corner in $M_d$ and $\cl B$ be a non-commutative lift of $\cl A$. 
If $p \in \cl P_d$ and $\rho = \sum_{i=1}^d p_i e_ie_i^*$ then 
$H_{\phi^{-1}(\cl A)}(p) = H_{\cl B}(\rho).$
\end{theorem}

\begin{proof} 
Since $\cl A \subseteq \cl B$, we have 
$H_{\cl B}(\rho) \leq H_{\phi^{-1}(\cl A)}(p)$.
Since $\cl A$ is standard, so is $\cl B$ and, by Remark \ref{HAmin}, 
$H_{\cl B}(\rho) < + \infty$.
Let $B \in \cl B$ be a minimiser for $H_{\cl B}(\rho)$.
Write $B = \sum_{i=1}^d b_i v_i v_i^*$, 
where $\{v_1,\dots,v_d\}$ an orthonormal basis of $\bb C^d$ and $b_i \geq 0$, $i\in [d]$;
thus, 
$H_{\cl B}(\rho) = - \sum_{i=1}^d \langle \rho v_i, v_i\rangle \log b_i$.
Suppose that $b_k = 0$ for some $k$.
Then 
$\sum_{j=1}^d p_j |\langle v_k,e_j\rangle|^2 = \langle \rho v_k,v_k\rangle = 0$.
Thus, for all $j\in [d]$, either $p_j = 0$ or $\langle v_k,e_j\rangle = 0$. 

Note that
$$\Delta (B) = \sum_{i=1}^d \ip{Be_i}{e_i}e_ie_i^*= \sum_{i=1}^d \left( \sum_{j=1}^d b_j |\langle v_j,e_i\rangle|^2 \right) e_ie_i^*.$$
By the concavity of the logarithm and the fact that 
$\sum_{j = 1}^d |\langle v_j,e_i\rangle|^2 = 1$, $i\in [d]$, we have
\begin{align*} -\Tr(\rho \log (\Delta (B)))
=& - \sum_{i=1}^d p_i \log \left( \sum_{j=1}^d b_j |\langle v_j,e_i\rangle|^2 \right) \\ 
=& - \sum_{i : p_i > 0} p_i \log \left( \sum_{j=1}^d b_j |\langle v_j,e_i\rangle|^2 \right) \\
=& - \sum_{i : p_i > 0} p_i \log \left( \sum_{j : \langle v_j,e_i\rangle \neq 0} b_j |\langle v_j,e_i\rangle|^2 \right) \\
\le & - \sum_{i=1}^d p_i \sum_{j=1}^d |\langle v_j,e_i\rangle|^2 \log b_j = - \Tr ( \rho \log B).
\end{align*}
Since $\Delta (B) \in \cl A$, we have $H_{\phi^{-1}(\cl A)}(p) \le H_{\cl B} ( \rho)$, 
and the proof is complete. 
\end{proof}

In the special cases where $\cl A = {\rm vp}(G)$ and $\cl A = {\rm thab}(G)$, the 
next result was given in \cite{Simonyi1} and \cite{Marton}, respectively
(we refer the reader to Subsection \ref{ss_canonicalcc} for the definition of the latter convex corners).

\begin{cor} \label{maxent} 
Let $\cl A $ be a standard convex $\bb R^d$-corner.  Then
$$\max_{p \in \cl P_d}  H_{\cl A}(p) =- \log N(\cl A).$$  
\end{cor}

The next two propositions 
give straightforward but useful characterisations 
of the extreme values for the entropy over the convex corners lying between $\cl A_{I_d}$ and $\cl B_{I_d}$.

\begin{prop} \label{H_A=0} 
Let $\cl A$ be a convex $M_d$-corner with $\cl A_{I_d} \subseteq \cl A \subseteq \cl B_{I_d}$.  
The following are equivalent:
\begin{enumerate}  
\item[(i)] $H_{\cl A}( \rho) = 0$ for all $\rho\in \cl R_d$; 
\item[(ii)] $\gamma( \cl A^\sharp) = 1$; 
\item[(iii)]  $I \in \cl A$; 
\item[(iv)] $\cl A = \cl B_{I_d}$;
\item[(v)] $\gamma (\cl A) = d$.
\end{enumerate}
\end{prop}

\begin{proof}
(i)$\Leftrightarrow$(ii) follows from Theorem \ref{noncomm13}.

(ii)$\Rightarrow$(iii) By Proposition \ref{NMgamma2}, $N(\cl A) = 1$ and hence $I\in \cl A$. 

(iii)$\Rightarrow$(iv) By hereditarity, $\cl B_{I_d}\subseteq \cl A$, and now by assumption $\cl A = \cl B_{I_d}$.

(iv)$\Rightarrow$(v) is trivial.

(v)$\Rightarrow$(iv) The assumption implies that $I\in \cl A$ and hence 
$\cl A = \cl B_{I_d}$. 

(iv)$\Rightarrow$(ii) follows from the fact that $\cl B_{I_d}^{\sharp} = \cl A_{I_d}$.
\end{proof}

\begin{prop}\label{m(A)=d} 
Let $\cl A$ be a convex $M_d$-corner with $\cl A_{I_d} \subseteq \cl A \subseteq \cl B_{I_d}$.  
The following are equivalent:
\begin{enumerate} 
\item [(i)] $H_{\cl A}(\rho) = H(\rho)$ for all $\rho \in \cl R_d$; 
\item [(ii)] $\gamma( \cl A^\sharp) = d$; 
\item [(iii)] $\cl A = \cl A_{I_d}$; 
\item [(iv)] $\gamma (\cl A) = 1$.
\end{enumerate}
\end{prop}

\begin{proof}
(i)$\Rightarrow$(iv)   
Suppose that there exists $B \in \cl A$ with $\Tr B = t > 1$;
we have that $t^{-1}B \in \cl R_d$.  
Since $\frac{1}{d}I_d \in \cl A$, there exists $\ep > 0$ such that
$B' := (1 - \ep)B + \frac{\ep}{d}I_d \in \cl A \cap M_d^{++}$ satisfies $\Tr B' > 1$. 
Thus, without loss of generality, we may assume that $B \in M_d^{++}$.   
We have  
$$H(t^{-1}B) = - \Tr(t^{-1}B \log (t^{-1}B)) = \log t - \Tr (t^{-1}B \log B),$$ 
and
$$H_{\cl A}(t^{-1}B) = \min_{A \in \cl A} - \Tr(t^{-1}B \log A)  \le - \Tr(t^{-1}B \log B) < H(t^{-1}B),$$  
contradicting (i).

(iv)$\Rightarrow$(iii) follows from the assumption that $\cl A_{I_d} \subseteq \cl A$.

(iii) $\Rightarrow$ (i) This was proved in Remark \ref{HAmin} (iv).

(ii)$\Leftrightarrow$(iii)  
We have $\cl A_{I_d} \subseteq \cl A^\sharp \subseteq \cl B_{I_d}$. 
Thus, $\cl A = \cl A_{I_d} \iff \cl A^\sharp = \cl B_{I_d} \iff \gamma (\cl A^\sharp)=d$ by Proposition \ref{H_A=0}.
\end{proof}


\subsection{Dependence on the state and on the convex corner}

In this subsection, we examine the properties of the entropy as a function of the state and of the convex corner. 

\begin{prop} \label{Hconcave}
Let $\cl A$ be a bounded convex $M_d$-corner. 
Then the function 
$H_{\cl A} : \cl R_d \To \bb{R} \cup\{+\infty\}$, $\rho \to H_{\cl A}( \rho)$, 
is concave.  
If $\cl A$ is standard then $H_{\cl A}$
is upper semi-continuous and attains a finite maximum. 
\end{prop}

\begin{proof} 
Let $\rho_{i} \in \cl R_d$ and $\lambda_{i} \in \bb R^+, \, i=1,2$, with $\lambda_1+ \lambda_2=1$.
By Lemma \ref{minA2}, there exists $A_0 \in \cl A$ such that
\begin{align*} 
H_{\cl A}(\lambda_1 \rho_1 + \lambda_2 \rho_2)
=& 
\lambda_1 \Tr (- \rho_1 \log A_0) + \lambda_2 \Tr (- \rho_2 \log A_0) \\ 
\ge& 
\lambda_1 \min_{A \in \cl A}  \Tr(- \rho_1 \log A)+\lambda_2 \min_{A \in \cl A}  \Tr(- \rho_2 \log A) \\
=& 
\lambda_1 H_{\cl A}(\rho_1) + \lambda_2 H_{\cl A}(\rho_2).
\end{align*}

Assume $\cl A$ is standard. 
For $\rho \in \cl R_d$ and $B \in M_d^{+}$ satisfying $\ker(B) \subseteq \ker \rho$, let $g(\rho, B)= -\Tr (\rho \log B).$   
By Remark \ref{HAmin},
$H_{\cl A}(\rho) < +\infty$ for all $\rho \in \cl R_d$.
Let $(\rho^{(n)})_{n \in \bb N}$ be a sequence in $\cl R_d$ converging to $\rho \in \cl R_d$. 
Let $A \in \cl A $ and $A^{(n)}\in \cl A$, $n\in \bb{N}$, be the elements of 
$\cl A$ such that 
$H_{\cl A}( \rho) = g(\rho, A)$ and $H_{\cl A}( \rho^{(n)}) = g(\rho^{(n)}, A^{(n)})$, $n\in \bb{N}$. 
By Lemma \ref{l_ri}, there exists $r > 0$ such that $rI \in \cl A$.
Since $\cl A$ is convex, 
$B_{\mu} := (1-\mu)A + \mu r I \in \cl A\cap M_d^{++}$ for every $\mu \in (0,1)$. 

Write $A = \sum_{i=1}^d \lambda_i v_i v_i^*$, 
where $\{v_1, \ldots, v_d\}$ is an orthonormal basis of $\bb C^d$ and $\lambda_i \ge 0$, $i\in [d]$. 
Setting $\rho_{i,i} = \ip{ \rho v_i}{v_i}$, $i\in [d]$, we have that
$g(\rho,A) = - \sum_{i=1}^d \rho_{i,i} \log \lambda_i$.
Since $A_0\in \cl A^0(\rho)$ (see the proof of Lemma \ref{minA2}), 
\begin{equation}\label{eq_lrhoi}
\lambda_i = 0 \ \Longrightarrow \ \rho_{i,i} = 0, \ \ \ i\in [d].
\end{equation}
We have that
\begin{align*} g(\rho, A) \le g(\rho,B_{\mu}) 
=& -\Tr \left( \rho \log \left( (1-\mu)A+ {\mu} rI\right) \right) \\ 
=& -\sum_{i=1}^d \rho_{i,i} \log \left((1-\mu) \lambda_i + {\mu}{r} \right).
\end{align*}  
By (\ref{eq_lrhoi}), $g(\rho,B_{\mu}) \To_{\mu \To 0} g(\rho,A)$.  
For $\delta > 0$, let $\mu \in (0,1)$ be such that 
\begin{equation}\label{eq_grhoAB}
g(\rho,A) \le g(\rho,B_{\mu}) \le g(\rho,A) + \delta.  
\end{equation}

On the other hand,
$$H_{\cl A}( \rho^{(n)}) = g(\rho^{(n)},A^{(n)}) \le g(\rho^{(n)},B_{\mu}), \ \ \ n\in \bb{N}.$$  
Since $B_{\mu} > 0$, we have 
$$\limsup_{n \To \infty} H_{\cl A}( \rho^{(n)}) \le \limsup_{n \To \infty}g(\rho^{(n)},B_{\mu}) = g(\rho, B_{\mu}).$$  
By (\ref{eq_grhoAB}), 
$\limsup_{n \To \infty} H_{\cl A}( \rho^{(n)}) \le H_{\cl A}(\rho) + \delta,$
and $H_{\cl A}$ is upper semi-continuous as stated. 
By \cite[Theorem 2.43]{alip}, the compactness of $\cl R_d$ implies that a maximum value is attained.    
\end{proof}

\begin{theorem}\label{th_cecc}
Let $\rho\in \cl R_d$, and $\cl A$ and $\cl A_n$ be convex $M_d$-corners, $n\in \bb{N}$, such that 
$\cup_{n\in \bb{N}}\cl A_n$ is bounded. 
\begin{itemize}
\item[(i)] 
If $\limsup_{n\in \bb{N}} \cl A_n\subseteq \cl A$ then 
$H_{\cl A}(\rho) \leq \liminf_{n\in \bb{N}} H_{\cl A_n}(\rho)$;
\item[(ii)] 
If $\rho > 0$, $\cl A$ has non-empty relative interior and 
$\cl A \subseteq \liminf_{n\in \bb{N}} \cl A_n$ 
then 
$\limsup_{n\in \bb{N}} H_{\cl A_n}(\rho) \leq H_{\cl A}(\rho)$;
\item[(iii)] 
If $\rho > 0$, $\cl A$ has non-empty relative interior and
$\cl A = \lim_{n\in \bb{N}} \cl A_n$ then 
$H_{\cl A}(\rho) = \lim_{n\in \bb{N}} H_{\cl A_n}(\rho)$.
\end{itemize}
\end{theorem}

\begin{proof}
(i) 
Assume first that $H_{\cl A}(\rho) = \infty$. By Remark \ref{HAmin} (ii), 
$\cl A^0(\rho) = \emptyset$. 
Suppose, towards a contradiction, that there exists $C > 0$ and a sequence
$(A_k)_{k\in \bb{N}}\subseteq M_d^+$, such that $A_k\in \cl A_{n_k}$ 
and 
\begin{equation}\label{eq_AkC}
- \Tr \rho\log A_k \leq C, \ \ \ k\in \bb{N}.
\end{equation}
Assume, without loss of generality, that $A_k\to_{k\to \infty} A$ for some $A\in M_d$; 
by assumption, $A\in \cl A$. 
Write $A = \sum_{r=1}^l \lambda_r P_r$ in its spectral decomposition, where $(\lambda_r)_{r=1}^l$ is the 
family of distinct eigenvalues of $A$ in increasing order and 
$A_k = \sum_{r=1}^{l_k} \lambda_r^{(k)} P_r^{(k)}$ analogously. 
We have that, eventually, $l_k = l$, and hence we assume the latter equality holds for all $k\in \bb{N}$. 
By the continuity of the functional calculus, 
$P_r^{(k)}\to_{k\to\infty} P_r$ and $\lambda_r^{(k)} \to_{k\to\infty} \lambda_r$, $r\in [l]$.
Decomposing further $A = \sum_{r=1}^l \sum_{i=1}^{s_r} \lambda_r v_{r,i} v_{r,i}^*$, where $\{v_{r,i}\}_{i=1}^{s_r}$
is an orthonormal basis for the range of $P_r$, assume that 
$\delta := \langle \rho v_{r,i},v_{r,i}\rangle > 0$ but $\lambda_r = 0$, for some $r$ and $i$. 
We have that $\Tr(\rho P_r^{(k)}) > \frac{\delta}{2}$ for sufficiently large $k$, while 
$\lambda_r^{(k)}\to_{k\to \infty} 0$, contradicting (\ref{eq_AkC}).

Now suppose that $H_{\cl A}(\rho) < \infty$.  
If $\liminf_{n\in \bb{N}} H_{\cl A_n}(\rho) = \infty$ then the conclusion holds trivially, 
so suppose that $A_k\in \cl A_{n_k}$, $k\in \bb{N}$, satisfy (\ref{eq_AkC}) for some $C < \infty$. 
Assume, without loss of generality, that 
$A_k$ is the minimiser of $H_{\cl A_{n_k}}(\rho)$ and that 
$A_k\to_{k\to \infty} A$ for some $A\in \cl A$; 
thus, $A_k\in \cl A^0(\rho)$ for all $k$. 
Since $\cl A^0(\rho)$ is closed, the continuity of the functional calculus implies that 
$A\in \cl A^0(\rho)$. Now the lower semi-continuity of the function $X\to \Tr - \rho\log X$ implies that 
$H_{\cl A}(\rho) \leq C$. 
Thus, $H_{\cl A}(\rho) \leq \liminf_{n\in \bb{N}} H_{\cl A_n}(\rho)$. 

(ii) 
If $H_{\cl A}(\rho) = \infty$, the conclusion holds trivially; suppose thus that 
$H_{\cl A}(\rho) < \infty$. 
Let $A\in \cl A\cap M_d^{++}$ be such that $H_{\cl A}(\rho) = - \Tr \rho \log A$. 
Let $(A_n)_{n\in \bb{N}}$ be a sequence such that $A_n \in \cl A_n$, $n\in \bb{N}$, and 
$A_n \to_{n\to \infty} A$. We have that $A_n \in M_d^{++}$ eventually. 
Suppose that $H_{\cl A_{n_k}}(\rho) \to_{k\to \infty} \delta$
for some $\delta\in \bb{R}$. Then
$$H_{\cl A}(\rho) = - \lim_{n\to\infty} \Tr \rho \log A_n \geq \delta.$$

(iii) is a direct consequence of (i) and (ii). 
\end{proof}


\subsection{Entropy splitting}\label{ss_es}

This subsection is motivated by \cite[Section 2]{Csis}, and contains 
non-commutative analogues of the entropy splitting results obtained therein. 
If $V$ is an orthonormal basis of $\bb{C}^d$, we call a convex $M_d$-corner 
\emph{$V$-aligned} if $\Delta_V (\cl A) \subseteq \cl A$. 
Recall that $\cl A^{++}$ is the set of all invertible elements of a convex corner $\cl A$. 
We define the set
$$\log \cl A^{++} = \left\{ \log A: A \in \cl A^{++}\right\}.$$

\begin{lemma} \label{noncomm'} 
Let $V$ be an orthonormal basis of $\bb{C}^d$ and 
$\cl A$ be a bounded $V$-aligned convex $M_d$-corner. 
Then
\begin{itemize}
\item[(i)] $\Delta_V (\log \cl A^{++}) \subseteq \log \cl A^{++}$; 

\item[(ii)] If $\rho \in \cl R_d\cap \cl D_V$ then
there exists $A \in \cl A \cap \cl D_V$, such that $H_{\cl A}(\rho) = - \Tr (\rho \log A)$.  
\end{itemize}
\end{lemma} 

\begin{proof} 
(i) 
Write $V = \{v_1, \ldots, v_d\}$ 
and $A = \sum_{i=1}^d \lambda_i u_iu_i^* \in \cl A$, 
for a set $\{u_1,\ldots,u_d\}$ of orthonormal eigenvectors of $A$ and some $\lambda_i > 0$, $i\in [d]$. 
Then $\Delta_V (A) = \sum_{i,j=1}^d \lambda_i |\ip{u_i}{v_j}|^2 v_jv_j^*$
and $\log A = \sum_{i=1}^d \log\lambda_i u_iu_i^*$. 
Thus, 
$$\Delta_V (\log A) = \sum_{i,j=1}^d \log \lambda_i |\ip{u_i}{v_j}|^2 v_jv_j^*.$$ 
Set
$$A' = \sum_{j=1}^d  \exp\left(\sum_{i=1}^d|\ip{u_i}{v_j}|^2 \log \lambda_i \right) v_jv_j^*$$
and note that $\Delta_V (\log A) = \log A'$.
Since $\sum_{i=1}^d|\ip{u_i}{v_j}|^2=\|v_j\|^2=1$, the convexity of the exponential function in the extended real line 
implies that 
$A'\leq \Delta_V (A)$.
Since $\Delta_V (A) \subseteq \cl A$, it follows by hereditarity 
that $A' \in \cl A$, and hence $\Delta_V(\log A)= \log A' \in \log \cl A^{++}.$ 


(ii) 
If $H_{\cl A}(\rho) = +\infty$ then $\cl A^0(\rho) = \emptyset$, and we can pick any $A$ in $\Delta_V(\cl A)$. 
Suppose that $H_{\cl A}(\rho)$ is finite. 
Working with the extended real line $[-\infty,+\infty]$ and the conventions 
$0 \log 0 = 0$, $\log 0 = -\infty$ and ${\rm exp}(-\infty) = 0$, the operators
$\Tr(\rho\log A)$ and $\Tr(\rho\Delta_V(\log A))$ can be defined for any $\rho$ and $A$
(see e.g. \cite{opetz}). 
By Lemma \ref{minA2} and its proof, there exists $A \in \cl A$ such that 
$$-\Tr (\rho \log A)= H_{\cl A}(\rho).$$  
The operator $A' \in \cl A$ from (i) 
belongs to $\cl D_V$ and hence commutes with $\rho$. 
We have $\Delta_V (\rho) = \rho$ and so, by Remark \ref{diagon}, 
$$-\Tr (\rho \log A) = -\Tr \big (\Delta_V (\rho) \log A\big ) = - \Tr (\rho \log A').$$   
\end{proof}

The following result was proved in \cite{Csis} and will be needed below.

\begin{theorem} \label{noncomm20} \cite[Theorem 1]{Csis} 
If $\cl A, \cl B$ are convex $\bb R^d$-corners with $\cl A^\flat \subseteq \cl B$ 
then for any $p \in \cl P_d$ there exist $a \in \cl A$ and $b \in \cl B$ 
such that $p = ab$.  
\end{theorem}

\begin{prop} \label{noncomm10}    
Let $V$ be an orthonormal basis of $\bb{C}^d$,
$\rho \in \cl R_d \cap \cl D_V$ and 
$\cl A$ and $\cl B $ be convex $M_d$-corners.
\begin{itemize}
\item[(i)]
If $A \in \cl A$ and $B \in \cl B$, and $\rho = AB$ then
$H(\rho) \ge H_{\cl A}(\rho) + H_{\cl B}(\rho)$. 
Equality holds if and only if $A$ and $B$ are elements of $\cl A$ and $\cl B$ achieving the respective minima in Definition \ref{minAr}.   

\item[(ii)]
If $\cl A$ and $\cl B$ are $V$-aligned and
$\cl A^\sharp \subseteq \cl B$
then there exist $A \in \cl A$ and $B \in \cl B$ such that $\rho = AB$. 
\end{itemize}
\end{prop}

\begin{proof}
(i) 
Since $\rho = \rho^*$, we have that $AB = BA$.
Thus, 
\begin{align*} H_{\cl A}(\rho) +H_{\cl B}(\rho) &\le -\Tr (\rho \log A) - \Tr (\rho \log B) \\ &= -\Tr \big(\rho \log (AB)\big) 
= - \Tr \rho \log \rho = H(\rho).
\end{align*} 
The equality condition holds trivially. 

(ii) 
Let $\cl A_0 = \Delta_V(\cl A)$ and $\cl B_0 = \Delta_V(\cl B)$. 
By Remark \ref{diagon}, $\cl A_0 = \cl D_V \cap \cl A$ and $\cl B_0 = \cl D_V \cap \cl B$.
Let $\phi: \bb R_+^d \To \cl D_V \cap M_d^+$ be the bijection defined analogously to \eqref{phi};
then $\phi^{-1} (\cl A_0)$ and $\phi^{-1} (\cl B_0)$ are convex $\bb R^d$-corners. 
We claim that 
\begin{equation} \label{varphi} 
\cl D_V \cap \cl A_0^\sharp= \cl D_V \cap \cl A^\sharp. 
\end{equation}  
Since $\cl A_0 \subseteq \cl A$, we have $\cl D_V \cap \cl A^\sharp \subseteq \cl D_V \cap \cl A_0^\sharp$.  
Fix $M \in \cl D_V \cap \cl A_0^\sharp$ and $A \in \cl A.$  
By Remark \ref{diagon}, 
$$\ip{M}{A} = \ip{\Delta_V(M)}{A}= \ip{M}{\Delta_V(A)} \le 1.$$  
Thus, $M \in \cl D_V \cap \cl A^\sharp$ and \eqref{varphi} follows.  
We therefore have 
\begin{equation} \label{varphi2} 
\cl D_V \cap \cl A_0^\sharp \subseteq \cl D_V \cap \cl B = \cl B_0.
\end{equation}  
It is clear that 
$\phi^{-1}(\cl A_0)^\flat= \phi^{-1}(\cl D_V \cap \cl A_0^\sharp).$  
By \eqref{varphi2},
$\phi^{-1} ( \cl A_0)^\flat \subseteq \phi^{-1}(\cl B_0)$. 
For a state $\rho = \sum_{i=1}^d \rho_i v_iv_i^* \in \cl D_V$, 
we set $p = \phi^{-1}(\rho) \in \cl P_d$.
By Theorem \ref{noncomm20}, there exist $a \in \phi^{-1} (\cl A_0)$ and $b \in \phi^{-1} (\cl B_0)$ such that 
$\rho_i = a_i b_i$, $i\in [d]$.    
Then $\phi(a) = \sum_{i=1}^da_i v_iv_i^*\in\cl A_0 \subseteq \cl A$ and $\phi(b) = \sum_{i=1}^db_i v_iv_i^*\in \cl B_0 \subseteq \cl B$ 
satisfy $\phi(a)\phi(b) = \rho$  as required.
 \end{proof}

It was shown in \cite[Section 2]{Csis}
that if $\cl A$ is a convex $\bb R^d$-corner then
$$H(p) = H_{\cl A}(p) + H_{\cl A^\flat}(p) \mbox{ for all } p \in \cl P_n.$$
We provide a non-commutative version of this result.

\begin{theorem} \label{noncomm23}  
Let $V$ be an orthonormal basis of $\bb C^d$, $\rho \in \cl R_d \cap \cl D_V$ 
and $\cl A$ and $\cl B$ be $V$-aligned bounded convex $M_d$-corners.
\begin{itemize}
\item[(i)]
If $\cl B \subseteq \cl A^\sharp$ then $H(\rho) \le H_{\cl A}(\rho) + H_{\cl B}(\rho)$;

\item[(ii)]
If $\cl A^\sharp \subseteq \cl B$ then $H(\rho) \ge H_{\cl A}(\rho) + H_{\cl B}(\rho)$;

\item[(iii)]
$H(\rho) = H_{\cl A}(\rho) + H_{\cl A^\sharp}(\rho)$.
\end{itemize}
\end{theorem}

\begin{proof} 
(i)
By Lemma \ref{noncomm'}, there exist  $A \in \cl A\cap \cl D_V$ and $B \in \cl B \cap \cl D_V$ such that 
$H_{\cl A}(\rho)=-\Tr (\rho \log A)$ and $H_{\cl B}(\rho)=-\Tr (\rho \log B)$.
Write $V = \{v_1,\dots,v_d\}$, 
$$\rho = \sum_{i=1}^d p_i v_iv_i^*, \  A = \sum_{i=1}^d \lambda_i v_iv_i^* \mbox{ and } B = \sum_{i=1}^d \mu_i v_iv_i^*.$$  
We have
\begin{align*} 
H(\rho)- H_{\cl A}(\rho)- H_{\cl B}(\rho)
=& \Tr ( \rho  \log A) + \Tr ( \rho  \log B) - \Tr ( \rho  \log \rho)    \\ 
=& \sum_{i : p_i > 0} p_i \log \left( \frac{\lambda_i \mu_i}{p_i} \right) \le \log \left( \sum_{i : p_i > 0} \lambda_i \mu_i\right) \le 0,
\end{align*} 
where the first inequality follows from the concavity of the log function and the fact that 
$\sum_{i=1}^d p_i =1$, while the second one from the fact that 
$\sum_{i=1}^d \lambda_i \mu_i = \ip{A}{B} \le 1$.

(ii) follows from Proposition \ref{noncomm10}.

(iii) By Remark \ref{diagon},  $\Delta_V (\cl A^\sharp) \subseteq  \cl A^\sharp$.  
The result follows by setting 
$\cl B = \cl A^\sharp$ in (i) and (ii).  
\end{proof}

The following result is the non-commutative analogue of a bound established in 
\cite{bore} and \cite{Korner2}.

\begin{prop}  \label{lbAncg}  
Let $V$ be an orthonormal basis of $\bb{C}^d$, $\rho \in \cl R_d\cap \cl D_V$,  
and $\cl A$ be a $V$-aligned bounded convex $M_d$-corner.
Then 
\begin{equation}\label{eq_nbq}
H_{\cl A}(\rho) \ge H(\rho) - \log \gamma(\cl A).
\end{equation}
Equality holds in (\ref{eq_nbq}) 
if and only if $\gamma(\cl A) \rho \in \cl A.$ 
\end{prop}

\begin{proof}  
By Lemma \ref{noncomm'}, there exists 
$B \in \cl A\cap \cl D_V$ such that $H_{\cl A}(\rho)= - \Tr \rho \log B$.  
Write $\rho = \sum_{i=1}^d p_i v_iv_i^*$ and $B = \sum_{i=1}^d \mu_i v_iv_i^*$ with $p_i \geq 0$, $\mu_i \ge 0$, $i\in [d]$.  
Then 
$$H(\rho) = - \sum_{i=1}^d p_i \log p_i \ \mbox{ and } \ 
H_{\cl A}(\rho) =  -\sum_{i=1}^d p_i \log \mu_i.$$   
Hence 
$$\sum_{i=1}^d p_i \log \left( \frac {p_i}{\mu_i}\right)  \ge - \log \left(\sum_{i=1}^d \mu_i \right) \ge - \log \gamma(\cl A).$$ 
The equality condition follows as in \cite{bore}.  
\end{proof}

\noindent As in \cite{bore}, the lower bound (\ref{eq_nbq}) is attained.


\section{Tensor products of convex corners}\label{ss_pcc}

The behaviour of the entropy with respect to tensor products of convex $\bb{R}^d$-corners was 
examined in \cite[Section 5]{Csis}. In this section, we introduce tensor products of non-commutative convex corners, 
and discuss their behaviour in relation to the parameters defined earlier.

\begin{definition}\label{d_pcco}
Let $\cl A_i$ be a convex $M_{d_i}$-corner, $i = 1,2$. 
\begin{itemize}
\item[(i)] 
The \emph{maximal tensor product} of $\cl A_1$ and $\cl A_2$
is the convex $M_{d_1 d_2}$-corner
$$\cl A_1\otimes_{\max}\cl A_2 = {\rm C}\left(\{A_1\otimes A_2 : A_i\in \cl A_i, i = 1,2\}\right);$$

\item[(ii)] 
The \emph{minimal tensor product} of $\cl A_1$ and $\cl A_2$
is the convex $M_{d_1 d_2}$-corner
$$\cl A_1\otimes_{\min}\cl A_2 = \left(\cl A_1^{\sharp} \otimes_{\max}\cl A_2^{\sharp}\right)^{\sharp}.$$
\end{itemize}
\end{definition}

We note that 
\begin{equation}\label{eq_maxinmin}
\cl A_1\otimes_{\max}\cl A_2\subseteq \cl A_1\otimes_{\min}\cl A_2;
\end{equation}
the somewhat counterintuitive choice of notation becomes natural in view of the close
resemblance of these tensor products with the tensor products of operator systems 
as defined in \cite{kptt}. 
One defines tensor products of diagonal convex corners in an analogous way \cite[Section 5]{Csis}.

\begin{theorem}\label{th_partp}
Let $\cl A_i$ be a bounded convex $M_{d_i}$-corner, $i = 1,2$, $\tau\in \{\min, \max\}$,
and $\delta \in \{M,N,\gamma\}$. Then
$$\delta(\cl A_1\otimes_{\tau}\cl A_2) = \delta(\cl A_1) \delta(\cl A_2).$$
In addition, 
$$\Gamma(\cl A_1\otimes_{\min}\cl A_2) \leq \Gamma(\cl A_1\otimes_{\max}\cl A_2) 
\leq \Gamma(\cl A_1) \Gamma(\cl A_2).$$
\end{theorem}

\begin{proof}
We have
\begin{eqnarray*}
\gamma(\cl A_1)\gamma(\cl A_2) 
& = & 
\max\{\Tr(A_1\otimes A_2) : A_i\in \cl A_i, i = 1,2\}\\
& \leq &  
\max\{\Tr(A) : A\in \cl A_1\otimes_{\max}\cl A_2\}
= 
\gamma(\cl A_1\otimes_{\max}\cl A_2).
\end{eqnarray*}
The inequality 
$\gamma(\cl A_1\otimes_{\max}\cl A_2) \leq \gamma(\cl A_1)\gamma(\cl A_2)$ is straightforward from the definition of 
$\cl A_1\otimes_{\max}\cl A_2$ and Proposition \ref{l_gen}, and hence 
$\gamma(\cl A_1\otimes_{\max}\cl A_2) = \gamma(\cl A_1)\gamma(\cl A_2)$.
By Theorem \ref{NMgamma2}, 
\begin{eqnarray}\label{eq_mmin}
M\left(\cl A_1\otimes_{\min}\cl A_2\right)
& = & 
\gamma\left(\left(\cl A_1\otimes_{\min}\cl A_2\right)^{\sharp}\right) 
= 
\gamma\left(\cl A_1^{\sharp}\otimes_{\max}\cl A_2^{\sharp}\right)\\ 
& = & 
\gamma(\cl A_1^{\sharp}) \gamma(\cl A_2^{\sharp}) 
= 
M\left(\cl A_1\right) M\left(\cl A_2\right).
\end{eqnarray}

Suppose that $A_i\in \cl A_i$ and $\mu_i\geq 0$ are such that $\mu_i A_i\geq I$, $i = 1,2$.
Then $(\mu_1\mu_2)(A_1\otimes A_2) \geq I$ and hence 
$M(\cl A_1\otimes_{\max}\cl A_2) \leq \mu_1\mu_2$. 
After taking the infimum over all $\mu_1$ and $\mu_2$, we obtain
$$M(\cl A_1\otimes_{\max}\cl A_2) \leq M(\cl A_1) M(\cl A_2).$$
Inclusion (\ref{eq_maxinmin}) and equality (\ref{eq_mmin}) now imply
$$M(\cl A_1\otimes_{\max}\cl A_2) = M\left(\cl A_1\right) M\left(\cl A_2\right).$$
Thus,
$$\gamma\left(\cl A_1\otimes_{\min}\cl A_2\right)
= 
M\left(\cl A_1^{\sharp}\otimes_{\max}\cl A_2^{\sharp}\right)
=  
M(\cl A_1^{\sharp}) M(\cl A_2^{\sharp}) 
= 
\gamma(\cl A_1) \gamma(\cl A_2).$$
An application of Theorem \ref{NMgamma2} now completes the proof of the multiplicative identities.

Suppose that $(P_j^{(i)})_{j=1}^{m_i}$ is a PVM in $\cl A_i$, $i = 1,2$. 
Then $\{P_j^{(1)}\otimes P_k^{(2)} : j \in [m_1], k\in [m_2]\}$ is a PVM in $\cl A_1\otimes_{\max}\cl A_2$. 
Together with (\ref{eq_maxinmin}), this shows the inequality chain. 
\end{proof}

Suppose that $\rho$ is a state in $M_{d_1}\otimes M_{d_2}$. We denote by 
$\Tr_i\rho$ the reduced state of $\rho$ in $M_{d_i}$, $i = 1,2$; thus, 
$\Tr_1\rho \in M_{d_1}^+$,
$$\langle \Tr\mbox{}_1\rho, A_1\rangle = \langle \rho, A_1\otimes I\rangle, \ \ \ A_1\in M_{d_1},$$
and similar identities hold for $\Tr_2\rho$.

\begin{theorem}\label{th_enprod}
Let $\cl A_i$ be a standard 
convex $M_{d_i}$-corner, $i = 1,2$, $\tau\in \{\min, \max\}$, and $\rho$ be a 
state in $M_{d_1}\otimes M_{d_2}$. Then 
$$H_{\cl A_1\otimes_{\tau}\cl A_2}(\rho) \leq H_{\cl A_1}(\Tr\mbox{}_1\rho) + H_{\cl A_2}(\Tr\mbox{}_2\rho).$$
If $V_i$ is an orthonormal basis of $\bb{C}^{d_i}$, $\rho_i \in \cl R_{d_i}\cap \cl D_{V_i}$, 
and $\cl A_i$ is $V_i$-aligned, $i = 1,2$, then 
\begin{equation}\label{eq_proe2}
H_{\cl A_1\otimes_{\tau}\cl A_2}(\rho_1\otimes \rho_2) = H_{\cl A_1}(\rho_1) + H_{\cl A_2}(\rho_2).
\end{equation}
\end{theorem}

\begin{proof}
Let $A_i \in \cl A_i^{++}$, $i = 1,2$. 
If $B_1\in M_{d_1}$ then 
\begin{eqnarray*}
\langle (\Tr\mbox{}_1 \rho)\log A_1, B_1\rangle 
& = & 
\langle (\Tr\mbox{}_1 \rho), (\log A_1) B_1\rangle
= 
\langle \rho, (\log A_1) B_1\otimes I\rangle\\
& = & 
\langle \rho, ((\log A_1) \otimes I) (B_1\otimes I)\rangle\\
& = & 
\langle \rho \log (A_1\otimes I), B_1\otimes I\rangle\\
& = & 
\langle \Tr\mbox{}_1 (\rho \log (A_1\otimes I)), B_1\rangle;
\end{eqnarray*}
thus, $(\Tr\mbox{}_1 \rho)\log A_1 = \Tr\mbox{}_1 (\rho \log (A_1\otimes I))$
and, by symmetry,
$(\Tr\mbox{}_2 \rho)\log A_2 = \Tr\mbox{}_2 (\rho \log (I \otimes A_2))$.

Since $\cl A_1$ and $\cl A_2$ are standard, so are $\cl A_1\otimes_{\max} \cl A_2$ and 
$\cl A_1\otimes_{\min}\cl A_2$. 
Using Remark \ref{HAmin}, we have 
\begin{eqnarray*}
& & \hspace{-0.62cm} H_{\cl A_1\otimes_{\max}\cl A_2}(\rho)\\
& \leq & 
\inf \{-\Tr (\rho \log (A_1 \otimes A_2)) : A_i\in \cl A_i^{++}, i = 1,2\}\\
& = & 
\inf \{-\Tr (\rho \log (A_1 \otimes I)) - \Tr (\rho \log(I\otimes A_2)) : A_i\in \cl A_i^{++}, i = 1,2\}\\
& = & 
\inf \{-\Tr((\Tr\mbox{}_1 \rho)\log A_1) : A_1\in \cl A_1^{++}\}\\
& + & 
\inf \{-\Tr((\Tr\mbox{}_2 \rho)\log A_2) : A_2\in \cl A_2^{++}\}\\
& = & 
H_{\cl A_1}(\Tr\mbox{}_1\rho) + H_{\cl A_2}(\Tr\mbox{}_2\rho).
\end{eqnarray*}
The inequality in Theorem \ref{th_enprod} for the minimal tensor product now follows from (\ref{eq_maxinmin}). 
Using Theorem \ref{noncomm23}, we hence have 
\begin{eqnarray*}
H(\rho_1\otimes \rho_2) 
& = & 
H_{\cl A_1\otimes_{\max}\cl A_2}(\rho_1\otimes \rho_2) + H_{\cl A_1^{\sharp} \otimes_{\min} \cl A_2^{\sharp}}(\rho_1\otimes \rho_2)\\
& \leq & 
H_{\cl A_1}(\rho_1) + H_{\cl A_2}(\rho_2) + H_{\cl A_1^{\sharp}}(\rho_1) + H_{\cl A_2^{\sharp}}(\rho_2)\\
& = & 
H(\rho_1) + H(\rho_2) = H(\rho_1\otimes \rho_2).
\end{eqnarray*}
Equality (\ref{eq_proe2}) is now immediate. 
\end{proof}

\noindent {\bf Remark. } 
Tensor products of convex $\bb{R}^d$-corners were introduced in \cite{Csis} in an analogous way to Definition 
\ref{d_pcco}, where for the definition of the minimal tensor product one uses the classical anti-blocker 
$\flat$ instead of the non-commutative one $\sharp$. 
Equality (\ref{eq_proe2}) generalises \cite[Theorem 16]{Csis}, where the similar equality was shown for 
the entropy of product probability distributions with respect to products of classical convex corners.

\begin{prop}\label{p_liftspro}
Let $\cl A_i$ be a diagonal convex corner in $M_{d_i}$, 
$\cl B_i$ be a non-commutative lift of $\cl A_i$, $i = 1,2$, 
and $\tau\in \{\min,\max\}$. 
Then $\cl B_1\otimes_{\tau}\cl B_2$ is a non-commutative lift of $\cl A_1\otimes_{\tau}\cl A_2$. 
\end{prop}

\begin{proof}
Denote by $\Delta_i$ the conditional expectation onto $\cl D_{d_i}$;
we have that $\Delta_i(\cl B_i)\subseteq \cl B_i$, $i = 1,2$.
It follows that 
$$(\Delta_1\otimes\Delta_2)(\cl B_1\otimes_{\max}\cl B_2)\subseteq \cl B_1\otimes_{\max}\cl B_2,$$
and hence (see Remark \ref{diagon}),
$$(\Delta_1\otimes\Delta_2)(\cl B_1\otimes_{\max}\cl B_2) = (\cl B_1\otimes_{\max}\cl B_2)\cap (\cl D_{d_1}\otimes\cl D_{d_2}).$$
Thus, $\cl B_1\otimes_{\max}\cl B_2$ is a non-commutative lift of $\cl A_1\otimes_{\max}\cl A_2$. 
By Proposition \ref{Deltasharp}, 
$(\cl B_1^{\sharp}\otimes_{\max}\cl B_2^{\sharp})^{\sharp}$ is a non-commutative lift of 
$(\cl A_1^{\flat}\otimes_{\max}\cl A_2^{\flat})^{\flat}$, and the proof is complete. 
\end{proof}


\section{Convex corners from non-commutative graphs}\label{s_ccfncg}


\subsection{Motivation}

In this subsection, we recall some basic notions from zero-error information and quantum information theory; 
we refer the reader to \cite{NC} for some of the basic notions, such as completely positive maps and quantum channels.
Given a classical information channel $\cl N$ with an input alphabet $[d]$ and an output alphabet $[k]$, 
its confusability graph $G_{\cl N}$, as defined by Shannon in \cite{Shannon2}, has vertex set $[d]$, and 
two symbols $i,j\in [d]$ are adjacent if they may result in the same output from $[k]$ 
after transmission via $\cl N$. Shannon observed that the one-shot zero-error capacity of $\cl N$ 
-- that is, the size of a largest subset of $[d]$, no two elements of which can result in the same output 
after applying $\cl N$ 
-- is equal to the independence number $\alpha(G_{\cl N})$ of $G_{\cl N}$. 
The zero-error transmission properties of $\cl N$ were thus reduced to the study of various asymptotic 
combinatorial parameters of $G_{\cl N}$.
Given two information channels with confusability graphs $G_1$ and $G_2$, on vertex sets $[d_1]$ and 
$[d_2]$, respectively, the product channel has 
confusability graph equal to the \emph{strong product} $G_1 \boxtimes G_2$ of $G_1$ and $G_2$, that is, the graph with 
vertex set $[d_1]\times [d_2]$, in which $(i,k) \simeq (j,l)$
if and only if $i \simeq j$ in $G_1$ and $k \simeq l$ in $G_2$.
(Here, and in the sequel, we write $i\sim j$ to denote adjacency, and $i\simeq j$ if $i\sim j$ or $i = j$.)
Writing $G^{\boxtimes n}$ for the $n$-fold strong product of $G$, the \emph{Shannon capacity} \cite{Shannon2} of $G$ is 
the parameter
$$\Theta(G) = \lim_{n \to \infty}  \sqrt[n]{\alpha\left(G^{\boxtimes n}\right)}.$$

In the zero-error quantum communication task, Alice uses a quantum channel -- that is, a 
completely positive trace preserving linear map $\Phi : M_d\to M_k$ --
to send to Bob states from $\cl R_d$, received at Bob's site as states from $\cl R_k$. 
The \emph{one-shot zero-error capacity} of $\Phi$ is the maximum number $m$ of pure states 
$\xi_1\xi_1^*$, $\xi_2\xi_2^*,\dots,\xi_m\xi_m^*$ in $\cl R_d$ such that 
$\Phi(\xi_i\xi_i^*)\perp \Phi(\xi_j\xi_j^*)$ for $i\neq j$ 
(here, and in the sequel, for $\rho_1,\rho_2\in M_d$, we write $\rho_1\perp \rho_2$ if 
$\rho_1$ and $\rho_2$ are orthogonal in the Hilbert-Schmidt inner product). 
Let $\Phi$ have a Kraus representation 
$$\Phi(T) = \sum_{p=1}^r A_p T A_p^*, \ \ \ T\in M_d,$$
where $A_p : \bb{C}^d\to \bb{C}^k$, $p\in [r]$, are such that $\sum_{p=1}^r A_p^* A_p = I$.
Set 
$$\cl S_{\Phi} = {\rm span}\left\{A_p^* A_q : p,q\in [r]\right\}$$
and note that $\cl S_{\Phi}$ is an operator system in $M_d$, in the sense that 
$$I\in \cl S_{\Phi} \mbox{ and } S\in \cl S_{\Phi} \Rightarrow S^*\in \cl S_{\Phi}.$$
The operator system $\cl S_{\Phi}$ was shown in \cite{duan} to depend only on $\Phi$ -- and not on the 
particular Kraus representation of $\Phi$ used to define it -- 
and to capture many 
zero-error transmission properties of $\Phi$, 
playing the role of a confusability graph of $\Phi$ in the quantum setting. 
For example, it was observed that, for two unit vectors $\xi,\eta\in \bb{C}^d$, 
we have $\Phi(\xi\xi^*)\perp \Phi(\eta\eta^*)$ if and only if $\xi\eta^*\perp \cl S_{\Phi}$;
thus, the one-shot zero-error capacity of $\Phi$ coincides with the \emph{independence number}
$\alpha(\cl S)$ of $\cl S = \cl S_{\Phi}$, defined as 
$$\alpha(\cl S) = \max\left\{m : \exists \mbox{ unit vectors } \xi_i\in \bb{C}^d, i\in [m], \mbox{ s.t. } 
\xi_i\xi_j^*\perp \cl S \mbox{ if } i\neq j\right\}.$$
It is easy to note that, if $\cl S$ and $\cl T$ are operator systems in $M_d$ then 
$\alpha(\cl S\otimes\cl T) \geq \alpha(\cl S) \alpha(\cl T)$; by Fekete's Lemma, 
the \emph{Shannon capacity} 
$$\Theta(\cl S) = \lim_{n \To \infty} \sqrt[n]{\alpha\left(\cl S^{\otimes n}\right)}$$ 
of $\cl S$ is well-defined.

An arbitrary operator system in $M_d$ was hence called a \emph{non-commutative graph} in \cite{duan}. 
Given a graph $G$ with vertex set $[d]$, let 
$$\cl S_G = \Span\{ e_ie_j^*: i,j \in [n], ~i \simeq j \mbox{ in } G\}$$ 
be the \emph{graph operator system} of $G$. 
It was observed in \cite{duan} that $\alpha(\cl S_G) = \alpha(G)$ for every graph $G$. 
Since $\cl S_{G_1 \boxtimes G_2}= \cl S_{G_1} \otimes \cl S_{G_2}$, this implies that 
$\Theta(\cl S_G) = \Theta(G)$.

Identifying computable bounds on the Shannon capacity of a graph, together with questions about 
information sources equipped with non-uniform probability distributions that describe the likelihood of a 
particular symbol from $[d]$, leads naturally to the consideration of several convex $\bb{R}^d$-corners 
canonically associated with the graph $G$ \cite{relax}. 
In the next subsection, we recall these convex corners, their non-commutative counterparts \cite{BTW},
and establish some relations between them.


\subsection{Canonical convex corners from graphs}\label{ss_canonicalcc}

Let $G$ be a graph with vertex set $[d]$.
Recall that a subset $S\subseteq [d]$ is called independent (resp. a clique) 
if $i\not\sim j$ (resp. $i\simeq j$) whenever $i,j\in S$.
The complement $\bar{G}$ of $G$ has vertex set $[d]$, and $i\sim j$ in $\bar{G}$ 
if $i\not\simeq j$ in $G$. 
The \emph{vertex packing polytope} \cite{relax} of $G$ is the set
$$\vp(G) = \conv\left\{\chi_{S} : S\subseteq [d] \mbox{ an independent set}\right\},$$
while the \emph{fractional vertex packing polytope} \cite{relax} of $G$ is the set
$$\fvp(G) = \left\{x\in \bb{R}_+^d : \sum_{i\in K}x_i \leq 1,  \mbox{ for all cliques } K\subseteq [d]\right\};$$
note that $\fvp(G) = \vp(\bar{G})^{\flat}$.
(We denote by $\chi_S$ the characteristic function of a set $S$.)
We view these sets as diagonal convex corners in $M_d$ via the map (\ref{phi}). 

The notion of an $\cl S$-independent set in Definition \ref{Scliqfull} below
was first given in \cite{Paulsen}, while the notions of 
an $\cl S$-full set and an $\cl S$-clique were introduced in \cite{BTW}.

\begin{definition} \label{Scliqfull}  
Let $\cl S \subseteq M_d$ be a non-commutative graph. 
An orthonormal set $\{ v_1, \ldots, v_k\} \subseteq \bb C^d$ is called
\begin{enumerate} 
\item [(i)] \emph{$\cl S$-independent} if $v_iv_j^* \in \cl S^\perp$ for all $i \ne j$;
\item [(ii)] \emph{$\cl S$-clique} if $v_iv_j^* \in \cl S$ for all $i \ne j$, and
\item [(iii)] \emph{$\cl S$-full} if $v_iv_j^* \in \cl S$ for all $i,j \in [k]$.  
\end{enumerate} 
A projection $P\in M_d$ is called \emph{$\cl S$-abelian} 
(resp. \emph{$\cl S$-clique}, \emph{$\cl S$-full}) if its range is the span of an $\cl S$-independent set
(resp. an $\cl S$-clique, an $\cl S$-full) set.
\end{definition}

We let $\cl P_{\rm a}(\cl S)$ (resp. $\cl P_{\rm c}(\cl S)$, $\cl P_{\rm f}(\cl S)$) be the set of all 
$\cl S$-abelian (resp. $\cl S$-clique, $\cl S$-full) projections. 
We have that a projection $P$ is $\cl S$-abelian if and only if the set $P\cl S P$ consists of 
commuting operators; this fact was communicated to us by Vern Paulsen (see \cite{BTW}).

\medskip

\noindent {\bf Remark. } 
If $G$ is a graph with vertex set $[d]$
and $S\subseteq [d]$ is an independent set of $G$ then the set $\{e_i : i \in S\}$ is $\cl S_G$-independent. 
Similarly, 
if $K\subseteq [d]$ is a clique of $G$ 
then the set $\{e_i : i \in K\}$ is $\cl S_G$-full, and hence an $\cl S_G$-clique. 
The notion of an $\cl S$-independent set -- and that of an $\cl S$-abelian projection --
can thus be viewed a non-commutative version of the notion of an independent set of a graph. 
Similarly, $\cl S$-clique and $\cl S$-full projections are (distinct) non-commutative versions of  
the notion of a clique of a graph.

\medskip

Recall the following convex $M_d$-corners, associated with a non-commu- tative graph $\cl S \subseteq M_d$ \cite{BTW}:  
\begin{itemize}
\item $\ap(\cl S) = {\rm C}\left(\cl P_{\rm a}(\cl S)\right)$, the \emph{abelian projection convex corner};
\item $\cp(\cl S) = {\rm C}\left(\cl P_{\rm c}(\cl S)\right)$,
the \emph{clique projection convex corner};
\item $\fp(\cl S) = {\rm C}\left(\cl P_{\rm f}(\cl S)\right)$, the \emph{full projection convex corner}.
\end{itemize}

\begin{remark}\label{apcpbounds} 
Let $\cl S \subseteq M_d$ be a non-commutative graph.
\begin{itemize}
\item[(i)] Since every $\cl S$-full projection is $\cl S$-clique, we have $\fp(\cl S) \subseteq \cp(\cl S)$.
\item[(ii)] Since every rank one projection is trivially $\cl S$-abelian and $\cl S$-clique, 
$\cl A_{I_d} \subseteq \ap(\cl S) \subseteq \cl B_{I_d}$ and $\cl A_{I_d} \subseteq \cp(\cl S) \subseteq \cl B_{I_d}$.
\item[(iii)] The convex corners $\ap(\cl S)$ and $\cp(\cl S)$ are standard. This is not always true for 
$\fp(\cl S)$, which can reduce to $\{0\}$. 
\item[(iv)] If $\cl T\subseteq M_d$ is a non-commutative graph with 
$\cl S \subseteq \cl T$ then 
$\ap(\cl T) \subseteq \ap(\cl S)$, $\cp(\cl S) \subseteq \cp(\cl T)$ and $\fp(\cl S) \subseteq \fp(\cl T)$.
\end{itemize}
\end{remark}

Parts (i)-(ii) of the next proposition were established in \cite{BTW}, while (iii)-(iv) 
follow after an application of Proposition \ref{Deltasharp}. 

\begin{prop}  \label{apvp}
Let $G$ be a graph. The following hold:
\begin{itemize}
\item[(i)] 
$\ap(\cl S_G)$ is a non-commutative lift of $\vp(G)$;

\item[(ii)] 
$\cp(\cl S_G)$ and $\fp(\cl S_G)$ are non-commutative lifts of $\vp (\bar{G})$;

\item[(iii)] 
$\ap (\cl S_G)^\sharp$ is a non-commutative lift of $\fvp(\bar{G})$;

\item[(iv)] 
$\cp(\cl S_G)^\sharp$ and $\fp(\cl S_G)^\sharp$ are non-commutative lifts of $\fvp(G)$.
\end{itemize}
\end{prop}

Now Theorem \ref{diagdelta}, Remark \ref{apcpbounds} 
and Proposition \ref{apvp} imply the following.

\begin{cor} \label{apvpgraph} 
Let $G$ be a graph. The following hold:
\begin{itemize} 
\item[(i)]  $\her(\vp(G)) \subseteq \ap(\cl S_G) \subseteq  (\vp(G)^\flat)^\sharp$; 
\item[(ii)] $\her(\vp(\bar{G})) \subseteq \fp(\cl S_G) \subseteq \cp(\cl S_G) \subseteq  \fvp(G)^\sharp$; 
\item[(iii)]  $\her(\vp(G)^\flat) \subseteq \ap(\cl S_G) ^\sharp \subseteq  \vp(G)^\sharp$;
\item[(iv)]  $\her(\fvp(G)) \subseteq \cp(\cl S_G)^\sharp \subseteq \fp(\cl S_G)^\sharp \subseteq  (\fvp(G)^\flat)^\sharp$.
\end{itemize}  
\end{cor}

By Lemma \ref{flatsharp=}, the outer terms in Corollary \ref{apvpgraph} are distinct whenever $d > 1$. 
We next examine when the middle terms reduce to their extreme values. 
We denote by $K_d$ the complete graph with vertex set $[d]$, in which $i\simeq j$ for all $i,j\in [d]$. 
Its complement $\bar{K}_d$ is thus the empty graph on $[d]$, in which $i\simeq j$ precisely when $i = j$.

\begin{theorem} \label{emptyap}  
Let $G$ be a graph on $d$ vertices. 
\begin{itemize}
\item[(i)] $\her(\vp(G)) = \ap(\cl S_G)$  if and only if $G$ is empty;
\item[(ii)] $\ap(\cl S_G) =  (\vp(G)^\flat)^\sharp$ if and only if $G$ is complete;
\item[(iii)] $\cp(\cl S_G)= \fvp(G)^\sharp$ if and only if $G$ is empty;   
\item[(iv)]  $\cp(\cl S_G)= \her (\vp(\bar{G}))$ if and only if $G$ is complete;   
\item[(v)]  $\fp(\cl S_G)= \her (\vp(\bar{G}))$ for every graph $G$.
\end{itemize}
\end{theorem}

\begin{proof} 
(i) 
By \eqref{eqSKd}, $\ap\left(\cl S_{\bar{K}_d}\right) = \cl B_{I_d},$ and as $I_d \in \vp\left(\bar{K_d}\right)$ 
we have $\vp\left(\bar{K}_d\right) = \left\{ M \in M_d^+ \cap \cl D_d: M \le I_d \right\}$, 
giving 
\begin{equation} \label{hervpempty} 
\her\left(\vp\left(\bar{K}_d\right)\right)= \cl B_{I_d}. 
\end{equation}
Conversely, suppose that $G$ is non-empty with $i \sim j$ in $G$.  
Let $v = \frac{1}{ \sqrt{2}} (e_i + e_j)$; we have that $vv^* \in \ap(\cl S_G)$.  
Suppose that $vv^* \le Q \in \vp(G)$. 
Write $Q = (q_{i,j})_{i,j=1}^d = \sum_{k=1}^m \mu_k P_k$, where 
$P_k = \sum_{i \in S_k} e_ie_i^*$ for an independent set $S_k$ of $G$ and 
scalars $\mu_k > 0$, $k\in [m]$, with $\sum_k \mu_k = 1$.    
Then $q_{i,i} \ge 1/2 $ and  $q_{j,j}\ge 1/2 $.  
Since no independent set $S_k$ contains both $i$ and $j$, we have 
$q_{i,i} = q_{j,j} = 1/2$. Thus $\langle (Q-vv^*)e_i,e_i\rangle = \langle (Q-vv^*)e_j,e_j\rangle = 0$.
By Lemma \ref{PSD2} and the fact that $Q$ is diagonal, 
$$0 = \langle (Q-vv^*)e_i,e_j\rangle = - \langle (vv^*)e_i,e_j\rangle = -1/2,$$
a contradiction. 
It follows that $vv^* \notin \her(\vp(G))$ and hence $\her(\vp(G)) \ne \ap(\cl S_G)$.

(ii) 
We have $\cl S_{K_d}=M_d$ and so the $\cl S_{K_d}$-abelian projections are precisely 
the rank one projections; thus, $\ap(\cl S_{K_d}) = \cl A_{I_d}$.  
It is clear that $\vp(K_d) = \{M \in \cl D_d^+ : \Tr M \le 1\}$; by Lemma \ref{flatsharp=}, 
\begin{equation} \label{Kdflatsharp}  
\left(\vp(K_d)^\flat\right)^\sharp = \left\{M \in M_d^+: \Delta (M) \in \vp(K_d)\right\} = \cl A_{I_d}.
\end{equation}

Conversely, suppose that  $k \not\simeq l$ in $G$.  
Let $A = (e_k + e_l)(e_k +e_l)^*$ and note that
$I-A \not\ge 0$.  
Since $\ap(\cl S_G) \subseteq \cl B_{I_d}$, it follows that $A \notin \ap(\cl S_G)$.
However, $\Delta (A) = e_ke_k^*+e_le_l^* \in \vp(G)$. 
By Lemma \ref{flatsharp=}, $A \in (\vp(G)^\flat)^\sharp$ 
and hence $\ap(\cl S_G) \ne (\vp(G)^\flat)^\sharp$.

(iii)
We claim that $ \cp(\cl S_{\bar{K}_d})= \{ M \in M_d^+: \Tr M \le 1\}.$ 
To see this note that a projection $P$ lies in 
$\cp(\cl S_{\bar{K}_d})$ if and only if $\rank(P) =1$. 
To establish the latter assertion, suppose there exist orthogonal unit vectors
$u = (u_i)_{i=1}^d$ and $v = (v_i)_{i=1}^d$  
such that $uv^* \in \cl S_{\bar{K}_d} = \cl D_d$.  
Suppose $u_i \ne 0$; then $v_j \ne 0$ for some $j \ne i$.  
Thus $\ip{e_je_i^*}{uv^*} \ne 0$,  contradicting the fact that $uv^* \in \cl D_d$.
By \eqref{Kdflatsharp}, $\cp(\cl S_{\bar{K}_d}) = (\vp({K_d})^\flat)^\sharp$.   
Suppose that $k \sim l$ in $G$.  As in (ii), 
let $A = (e_k + e_l)(e_k +e_l)^*$; by Lemma \ref{flatsharp=}, $A \in (\vp(\bar{G})^\flat)^\sharp$.
On the other hand, since $A \not\le I$, we have that $A \notin \cp(\cl S_G)$.

(iv) 
By \eqref{hervpempty} and \eqref{eqSKd} below, 
$\cp(\cl S_{K_d})= \her (\vp(\bar{K}_d)).$
Suppose that
$i \not\simeq j$ and let $v = \frac{1}{ \sqrt{2}} (e_i +e_j)$.  
Using the argument from (i), we conclude that $vv^* \notin \her(\vp(\bar{G}))$.

(v) 
By Corollary \ref{apvpgraph}, $\her(\vp(\bar{G})) \subseteq \fp(\cl S_G)$;
we show the reverse inclusion.  
Let $\{v_1, \ldots, v_r\}$ be an $\cl S_G$-full set and  $P = \sum_{i=1}^r v_iv_i^*$.  
Set $v_i = \sum_{j=1}^d \lambda_j^{(i)} e_j$ with $\lambda_j^{(i)} \in \bb C$, $j\in [d]$, $i\in [r]$.  
Now $v_iv_j^* \in \cl S_G$ for all $i,j \in [r]$ 
and hence, if $\lambda_l^{(i)}\lambda_k^{(j)}  \ne 0$ for some $i,j \in [r]$ then 
$l \simeq k$ in $G$.  We conclude that the set 
$K = \left\{j\in [d] : \lambda_j^{(k)} \ne 0 \mbox{ for some } k\right\}$ is a clique of $G$.  
Thus, $Q := \sum_{j \in K} e_je_j^* \in \vp(\bar{G})$.  
Note that $v_1, \ldots, v_r \in \Span \{ e_j: j \in K\}$; thus,
$$\ran(P) = \Span \left\{ v_i: i \in [r]\right\} \subseteq \Span \left\{ e_j: j \in S_Q \right\} = \ran (Q).$$ 
Hence $P \le Q$ and so $P \in \her(\vp(\bar{G}))$.  
Since $\her\left(\vp(\bar{G})\right)$ is closed and convex,
Proposition \ref{l_gen} implies that 
$\fp(\cl S_G) \subseteq \her\left(\vp(\bar{G})\right)$, as required.  
\end{proof}

\begin{remark}\label{r_perr}
Recall that a graph $G$ is called perfect if every induced subgraph has 
equal clique and chromatic numbers. 
It is shown in \cite{Csis} that a graph $G$ is perfect if and only if $\vp(G)=\fvp(G)$.  
By Proposition \ref{apvp}, $\vp(G)= \cl D_d \cap \ap(\cl S_G)$ and 
$\vp(\bar{G})=\cl D_d \cap \fp(\cl S_G)= \Delta(\fp(\cl S_G))$, 
whence Propositions \ref{Deltasharp} and \ref{apvp} give
$$\fvp(G)= \vp(\bar{G})^\flat = \cl D_d \cap (\fp(\cl S_G)^\sharp).$$  
Thus, $G$ is perfect if and only if $\cl D_d  \cap \ap(\cl S_G) = \cl D_d \cap \fp(\cl S_G)^\sharp$. 
It is worthwhile to note that the latter condition is not equivalent to 
$\ap(\cl S_G) = \fp(\cl S_G)^\sharp$; 
in fact, $\ap(\cl S_G) = \fp(\cl S_G)^\sharp$ if and only if $G$ is complete.  
To see this, note first that, by Proposition \ref{apvp}, 
if $\vp(G) \ne \fvp(G)$ then $\ap(\cl S_G) \ne \fp(\cl S_G)^\sharp$.  
Combined with Theorem \ref{emptyap}, 
this means that if $G$ is perfect then $\fp(\cl S_G)=\her(\vp(G)^\flat)$ and  
$\fp(\cl S_G)^\sharp=\her(\vp(G)^\flat)^\sharp$.   
However, by Theorem \ref{emptyap},
$\ap(\cl S_G) = (\vp(G)^\flat)^\sharp=\her(\vp(G)^\flat)^\sharp$ if and only if $G$ is complete.  
\end{remark}

Recall \cite{BTW} that, for any non-commutative graph $\cl S$, we have 
\begin{equation}\label{eq_sand}
\ap(\cl S)\subseteq \cp(\cl S)^{\sharp} \subseteq \fp(\cl S)^{\sharp}.
\end{equation}
Equality in this chain, for graph operator systems, is characterised in the next proposition.

\begin{prop}\label{p_ap=cp=fp}
Let $G$ be a graph on $d$ vertices. The following are equivalent:
\begin{itemize}
\item[(i)] $\ap(\cl S_{\bar{G}})=\cp(\cl S_G)= \fp(\cl S_G)$;
\item[(ii)] $G = K_d$;
\item[(iii)] $\ap(\cl S_{G}) = \cp(\cl S_G)^{\sharp} = \fp(\cl S_G)^{\sharp}$.
\end{itemize}
\end{prop}

\begin{proof} 
(ii)$\Rightarrow$(i)
It is clear that $\{e_1, \ldots, e_d\}$ is an $\cl S_{\bar{K_d}}$-independent set, an 
$\cl S_{K_d}$-clique and an $\cl S_{K_d}$-full set.  
Thus $I_d = \sum_{i=1}^d e_ie_i^*$ is an $\cl S_{\bar{K_d}}$-abelian projection, 
an $\cl S_{K_d}$-clique projection and an $\cl S_{K_d}$-full projection, and hence
\begin{equation} \label{eqSKd} 
\ap(\cl S_{\bar{K}_d})= \cp(\cl S_{K_d})=\fp(\cl S_{K_d})= \cl B_{I_d}.
\end{equation}  

(i)$\Rightarrow$(ii)
Suppose that $G \ne K_d$ and let $i,j \in [d]$ such that 
$i \not \simeq j$.  
Let $v= \frac{1}{\sqrt{2}}(e_i + e_j)$; then 
$vv^* \in \ap(\cl S_{\bar{G}})$ and $vv^* \in \cp(\cl S_G)$. 

Consider an $\cl S_G$-full set $\{v_1, \ldots, v_k\}$ with associated $\cl S_G$-full projection $P$.  
Write $v_l = \sum_{r=1}^d \alpha_r^{(l)}e_r$, $l\in [k]$.
Now $v_lv_m^* = \sum_{r,s=1}^d \alpha_r^{(l)} \overline{\alpha}_s^{(m)} e_re_s^* \in \cl S_G$ for all $l,m \in [k]$.
Thus for all $l,m \in [k]$ we have $ \alpha_i^{(l)} \overline{\alpha}_j^{(m)} =0$, 
so either $\alpha_i^{(l)}=0$ for all $l \in [k]$, or $\alpha_j^{(m)}=0$ for all $m \in [k]$.  
Thus, $\ip{P}{e_ie_j^*} = \sum_{l,m \in [k]} \alpha_i^{(l)} \overline{\alpha}_j^{(m)} = 0$.
It follows that $\ip{A}{e_ie_j^*} = 0$ for all $A \in \overline{\conv}(\cl P_{\rm f}(\cl S_G))$. 
On the other hand, by (\ref{eq_sand}) and Theorem \ref{2ndABe}, 
$A\in \ap(\cl S_G)^{\sharp}$, and hence
$\ip{e_i}{Ae_i} + \ip{e_j}{Ae_j} \le 1$, whenever $A \in \overline{\conv}(\cl P_{\rm f}(\cl S_G))$.
Since $\frac{1}{2}\begin{pmatrix}1&1 \\ 1& 1 \end{pmatrix} \not \le \frac{1}{2}\begin{pmatrix}1&0 \\ 0& 1 \end{pmatrix}$, 
we have that $vv^* \not \le A$ for all $A \in \overline{\conv}(\cl P_{\rm f}(\cl S_G))$, and
we conclude $vv^* \notin \fp(\cl S_G).$ 

(ii)$\Leftrightarrow$(iii) is immediate from (\ref{eq_sand}) and Remark \ref{r_perr}. 
\end{proof}

We now turn to the theta corners of classical and non-commutative graphs.
Let $G$ be a graph with vertex set $[d]$. 
A family $(a_i)_{i\in [d]}$ of unit vectors in 
a finite dimensional complex Hilbert space is called an \emph{orthogonal labelling (o.l.)} of $G$ 
if 
$$i\not\simeq j \ \Rightarrow \ a_i \perp a_j.$$
Let 
$$\cl P_0(G) = \left\{\left(|\langle a_i,c\rangle|^2\right)_{i=1}^d : (a_i)_{i=1}^d \mbox{ is an o.l. of } G \mbox{ and } \|c\|\leq 1\right\},$$
viewed as a subset of $\cl D_d$, and
set
$\thab(G) = \cl P_0(G)^{\flat}$.
We note that the original definition of $\thab(G)$ was given in real Hilbert spaces, but inspection of the proofs shows that the 
results in \cite{relax, knuth, lovasz} are true for complex Hilbert spaces as well. 

Let $\cl S\subseteq M_d$ be an operator system. Set \cite{BTW}
$$\frak{C}(\cl S) = \left\{\Phi : M_d \to M_k \ : \ k\in \bb{N}, 
\Phi \mbox{ is a quantum channel with } \cl S_{\Phi} \subseteq \cl S\right\}$$
and
$$\thet(\cl S) = \left\{T\in M_d^+ : \Phi(T)\leq I \mbox{ for every } \Phi\in \frak{C}(\cl S)\right \}.$$
It was shown in \cite{BTW} that 
the set $\thet(\cl S)$ is a convex $M_d$-corner, which we call the \emph{theta corner} of $\cl S$.
Note that if $\cl S, \cl T \subseteq M_d$ are operator systems then 
\begin{equation}\label{monothet} 
\cl S \subseteq \cl T \ \Rightarrow \thet(\cl T) \ \subseteq \thet(\cl S).
\end{equation}

It was shown in \cite{BTW} that, if $G$ is a graph then $\thet(\cl S_G)$ is a non-commutative lift of ${\rm thab}(G)$.
By Proposition \ref{Deltasharp}, $\thet (\cl S_G)^\sharp$ is a non-commutative lift of ${\rm thab}(G)^\flat$.
Thus, Theorem \ref{diagdelta} has the following corollary.

\begin{cor} \label{ththabgraph} 
Let $G$ be a graph. Then   
\begin{itemize} 
\item[(i)]  $\her({\rm thab}(G)) \subseteq \thet(\cl S_G) \subseteq  ({\rm thab}(G)^\flat)^\sharp$, and 
\item[(ii)]  $\her({\rm thab}(G)^\flat) \subseteq \thet(\cl S_G) ^\sharp \subseteq {\rm thab}(G)^\sharp$. 
\end{itemize}  
\end{cor}

\noindent 
We examine when we have equalities in the inclusions of Corollary \ref{ththabgraph}.

\begin{theorem} \label{thabcomplete} 
Let $G$ be a graph. The following hold:
\begin{itemize}
\item[(i)] $\thet(\cl S_G) =  ({\rm thab}(G)^\flat)^\sharp$ if and only if $G$ is complete;
\item[(ii)] $\thet(\cl S_G) = \her ({\rm thab}(G))$ if and only if $G$ is empty.
\end{itemize}
\end{theorem}

\begin{proof} 
(i) 
It is easy to see that 
$\thet(\cl S_{K_d}) = \cl A_{I_d}$ and ${\rm thab}(K_d) = \cl D_d \cap \cl A_{I_d}$.
By Lemma  \ref{flatsharp=}, we hence have
\begin{eqnarray*}  
({\rm thab}(K_d)^\flat)^\sharp 
& = & 
\{ M \in M_d^+ : \Delta (M) \in {\rm thab}(K_d)\}\\
& = & \{ M \in M_d^+ : \Tr M \le 1 \} = \thet (\cl S_{K_d}). 
\end{eqnarray*}
Conversely, suppose that $G$ is not complete, and let $k \not \simeq l$.  
Let  $A= (e_k + e_l)(e_k + e_l)^*$.  
Then $I-A \not\ge 0$ and, since $\thet(\cl S_G) \subseteq \cl B_{I_d}$, we have that $A \notin \thet(\cl S_G)$.
It is straightforward that $A \in  (\vp(G)^\flat)^\sharp$.  
Since $\vp(G) \subseteq {\rm thab}(G)$, we have 
$\vp(G)^\flat \supseteq {\rm thab}(G)^\flat$,   and  $ (\vp(G)^\flat )^\sharp \subseteq  ({\rm thab}(G)^\flat)^\sharp$.  
Thus $A \in  ({\rm thab}(G)^\flat)^\sharp$ and so $\thet(\cl S_G) \ne  ({\rm thab}(G)^\flat)^\sharp$.

(ii) 
It is easy to see that
$$\her({\rm thab}(\bar{K}_d))=\cl B_{I_d}= \thet(\cl S_{\bar{K}_d}).$$
Conversely, assume that $i \sim j$ in $G$.  
Setting $v = \frac{1}{\sqrt{2}} (e_i + e_j)$ we have $\Tr (vv^*) = 1$ and  $vv^* \in \thet(\cl S_G)$.  
Choosing an o.l. $(a^{(i)})_{i \in [d]}$ with $a^{(i)}=a^{(j)}$ and 
$\ip{a^{(i)}}{a^{(l)}}=0$ when $l \notin \{ i,j\}$ and
letting $c = a^{(i)}$ gives $ e_ie_i^*+e_je_j^* \in \cl P_0(G).$  
Suppose that $vv^* \in \her(\thet(G))$, that is $vv^* \le Q$ for some $Q \in {\rm thab}(G) \subseteq \cl D_d$.    
This requires $\ip{e_i}{Qe_i} > \frac{1}{2}$ and $\ip{e_j}{Qe_j} > \frac{1}{2}$.  
(Indeed, note that, since $Q \in \cl D_d$, we have that 
$\ip{e_i}{(Q-vv^*)e_j}=- \frac{1}{2}.$  But  if $\ip{e_i}{Qe_i }= \frac{1}{2}$, we have $e_i^*(Q-vv^*)e_i=0$, 
and since $Q \ge vv^*$, Lemma \ref{PSD2} implies that $\ip{e_i}{(Q-vv^*)e_j}=0$.  
A similar argument applies for $j$.) Thus, $\ip{Q}{e_ie_i^*+e_je_j^*} > 1$ and so
$Q \not \in \cl P_0(G)^\flat = {\rm thab}(G)$, a contradiction.  
We conclude $ vv^* \notin \her(\thet(G))$.
\end{proof}


\section{Non-commutative graph entropy}\label{s_nge}

In this section, we provide a quantum version of the 
notion of graph entropy, introduced by J. K\"{o}rner in \cite{Korner} and 
a non-commutative analogue of the fractional chromatic number of a graph.
We examine the continuity properties of non-commutative graph entropy and 
show its connection to the fractional chromatic number, extending to the non-commutative case a classical 
optimisation result from \cite{Simonyi2}.


\subsection{Entropy and fractional chromatic number}\label{ss_dbp}
Let $G$ be a graph with vertex set $[d]$ and let $p\in \cl P_d$ be a probability distribution
over its vertices.
The entropy $H(G,p)$ of $p$ with respect to $G$ was defined in \cite{Korner} as
the optimal coding rate of the source $\left([d],p\right)$ in the presence of ambiguity 
between the symbols from $[d]$, captured by the adjacency relation of $G$ 
(two symbols $i,j\in [d]$ are distinguishable if $i\sim j$ in $G$). 
The entropy functional $H(G,p)$ is thus defined as an asymptotic parameter, 
whose computation requires knowledge of the limiting behaviour 
of a sequence of chromatic numbers of powers of $G$.  
An elegant closed formula for $H(G,p)$, reminiscent of the 
definition of the classical Shannon entropy of $p$, was obtained in \cite{Korner}:
$$H(G,p) = \min \left\{\sum_{i=1}^d p_i \log\frac{1}{v_i} : v = (v_i)_{i=1}^d \in \vp(G), v > 0\right\}$$
or, equivalently, 
\begin{equation}\label{eq_HGp}
H(G,p) = \min_{v\in \vp(G)} - \Tr(p \log v).
\end{equation}

Let $\cl S\subseteq M_d$ be a non-commutative graph and $\rho$ be a state in $M_d$.
Since $\ap(\cl S)$ is a quantum version of $\vp(G)$,
taking (\ref{eq_HGp}) as a starting point in the non-commutative case, it is natural to 
make the following definition.

\begin{definition} \label{Hrho} 
The \emph{entropy} 
$H(\cl S, \rho)$ of a non-commutative graph $\cl S \subseteq M_d$ with respect to a state $\rho \in \cl R_d$  
is the quantity
$H(\cl S,\rho) = H_{\ap(\cl S)}(\rho).$
\end{definition}

It follows from Theorem \ref{noncomm15} and Proposition \ref{apvp} that, if $p \in \cl P_d$ and
$\rho = \sum_{i=1}^d p_ie_ie_i^*$ then $H(G,p) = H(\cl S_G, \rho)$. 
Thus, the parameter introduced in Definition \ref{Hrho} can be 
viewed as a non-commutative version of classical graph entropy.

\begin{remark}\label{nonneg}
Let $\cl S$ be a non-commutative graph in $M_d$ and $\rho\in \cl R_d$.
It follows from \eqref{genent} and Remark \ref{apcpbounds} that 
$0 \le H(\cl S, \rho) \le H(\rho).$
It is clear that 
$$\cl S_1 \subseteq \cl S_2 \ \Longrightarrow \ H( \cl S_1 , \rho) \le H( \cl S_2 , \rho).$$
\end{remark}

Let $G$ be a graph with vertex set $[d]$ and $\cl S\subseteq M_d$ be a non-commutative graph. 
Recall that the chromatic number $\chi(G)$ of $G$ is given by
$$\chi(G) = \min\left\{k\in \bb{N} : \ \exists \mbox{ indep. sets } S_1,\dots,S_k \mbox{ s.t. } 
\cup_{i=1}^k S_i = [d]\right\}.$$
Taking into account that the $\cl S$-abelian projections are the quantum analogue of 
independent sets, the following definition of a chromatic number of $\cl S$, given in \cite{Paulsen}, 
becomes natural:
$$\chi(\cl S) = \min\left \{k\in \bb{N} : \  P_1,\dots,P_k \in \cl P_{\rm a}(\cl S), \ \sum_{i=1}^k P_i = I\right\}.$$
It was shown in \cite{Paulsen} that, if $G$ is a graph then $\chi(\cl S_G) = \chi(G)$. 
Recalling the definitions made after Proposition \ref{capsharp}, we note that 
\begin{equation}\label{eq_chiGamma}
\Gamma(\ap(\cl S)) \leq \chi(\cl S).
\end{equation}

Similarly, recall that the \emph{fractional chromatic number} $\chi_{\rm f}(G)$ of $G$ 
is defined by letting 
\begin{equation}\label{eq_chrcl}
\chi_{\rm f}(G) = \min\left\{\sum_{S} \lambda_S : \lambda_S \geq 0, \sum_{S} \lambda_S\chi_S \geq 1\right\},
\end{equation}
where the summation is taken over independent sets $S$ of $G$.
By a duality argument, $\chi_{\rm f}(G)$ coincides with the \emph{fractional clique number} $\omega_{\rm f}(G)$ of $G$, 
defined by 
$$\omega_{\rm f}(G) = \max\left\{\sum_{i=1}^d \mu_i : \mu_i \geq 0, \sum_{i\in S} \mu_i \leq 1 \ \forall 
\mbox{ independent set } S\right\}.$$
In \cite{BTW}, we defined a non-commutative version of the fractional clique number by letting, 
for an operator system $\cl S\subseteq M_d$,
$$\omega_{\rm f}(\cl S) = \max\left\{\Tr(A) : A \in M_d^+, \Tr(AP) \leq 1 \mbox{ for all }
P \in \cl P_{\rm a}(\cl S)\right\}.$$ 
It is clear that 
\begin{equation}\label{eq_omegam}
\omega_{\rm f}(\cl S) = \gamma(\ap(\cl S)^{\sharp}),
\end{equation}
and it was shown in \cite{BTW} that $\omega_{\rm f}(\cl S_G) = \omega_{\rm f}(G)$. 

With the definition (\ref{eq_chrcl}) of the fractional chromatic number of a classical graph in mind, 
it is natural to define the \emph{fractional chromatic number} of a non-commutative graph $\cl S\subseteq M_d$ by setting 
\begin{equation} \label{eq_chif} 
\chi_{\rm f}(\cl S) 
= \inf\left\{\sum_{i=1}^k \lambda_i : \lambda_i > 0 \mbox{ and } \exists \    
P_1,\dots,P_k \in \cl P_{\rm a}(\cl S) \mbox{ s.t. }  \sum_{i=1}^k \lambda_i P_i \ge I \right\}.
\end{equation}

\begin{prop}\label{eq_chifGamma}
If $\cl S$ is a non-commutative graph then $\chi_{\rm f}(\cl S) = \Gamma_{\rm f}(\ap(\cl S))$.
\end{prop}

\begin{proof}
Since $\cl P_{\rm a}(\cl S)\subseteq \ap(\cl S)$, we have that 
$\Gamma_{\rm f}(\ap(\cl S)) \leq \chi_{\rm f}(\cl S)$. By \cite[Remark 2.7]{BTW}, 
the set $\cl P_{\rm a}$ is closed. Carath\'{e}odory's Theorem now implies that 
$\overline{{\rm conv}}(\cl P_{\rm a}) = {\rm conv}(\cl P_{\rm a})$. 
Suppose that $P\in \ap(\cl S)$ is a projection. Then 
$P \leq \sum_{i=1}^k \lambda_i P_i$ for some $P_i \in \cl P_{\rm a}$, $\lambda_i > 0$, 
$i = 1,\dots,k$, with $\sum_{i=1}^k \lambda_i = 1$. 
If $\xi$ is a unit vector with $\xi = P\xi$ then 
$1 \leq \sum_{i=1}^k \lambda_i \langle P_i\xi,\xi \rangle\leq 1$, and hence $\langle P_i\xi,\xi \rangle = 1$ for each $i\in [k]$. 
It follows that $P_i \xi = \xi$, and hence $P \leq P_i$, for each $i \in [k]$. 
Thus, if $I \leq \sum_{j=1}^l \mu_j Q_j$ for some positive scalars $\mu_j$ and some projections $Q_j\in \ap(\cl S)$
then $I \leq \sum_{r = 1}^m \nu_r P_r$, for some positive scalars $\nu_r$ and some $P_r\in\cl P_{\rm a}$, with 
$\sum_{r=1}^m \nu_r = \sum_{j=1}^l \mu_j$, completing the proof.
\end{proof}

As noted, if $G$ is a classical graph $G$ then
$\chi_{\rm f}(G) = \omega_{\rm f}(G)$. The non-commutative counterpart of this identity also holds,
but is much deeper and replies on the second anti-blocker theorem we proved in Section \ref{sect_AB}.

\begin{theorem}\label{th_fcc}
If $\cl S \subseteq M_d$ is a non-commutative graph then  
$\omega_{\rm f}(\cl S) =  \chi_{\rm f}(\cl S).$
\end{theorem}

\begin{proof}  
By Proposition \ref{eq_chifGamma} and Theorem \ref{NMgamma2}, $\chi_{\rm f}(\cl S) = M(\ap(\cl S))$. 
The claim now follows from (\ref{eq_omegam}) and Theorem \ref{NMgamma2}.
\end{proof}

It was shown in \cite[Lemma 4]{Simonyi2} that 
$$\max_{ p \in \cl P_n}H(G, {p}) = \log \chi_{\rm f}(G).$$
The next theorem, which is a direct consequence of 
Proposition \ref{eq_chifGamma} and
Theorems \ref{noncomm13} and \ref{NMgamma2}
establishes a quantum version of this identity. 

\begin{theorem}\label{chif7}  
Let $\cl S \subseteq M_d$ be an operator system. Then 
$$\max_{ \rho \in \cl R_d} H(\cl S,\rho)  = \log\chi_{\rm f}(\cl S).$$ 
\end{theorem}


\subsection{Further properties}\label{ss_sfp}

In this subsection, we include observations regarding the 
continuity, multiplicativity and extreme value properties of the
non-commutative graph entropy.

\begin{theorem}\label{th_contnceo}
Let $\cl S$ and $\cl S_n$ be non-commutative graphs in $M_d$, $n\in \bb{N}$, 
such that $\cl S\subseteq \liminf_{n\in \bb{N}}\cl S_n$. 
Then
$H(\cl S,\rho)\leq \liminf_{n\in \bb{N}} H(\cl S_n,\rho)$ for every $\rho\in \cl R_d$.
\end{theorem}

\begin{proof}
We first claim that 
\begin{equation}\label{eq_apc}
\limsup_{n\in \bb{N}} \ap(\cl S_n) \subseteq \ap(\cl S).
\end{equation}
Suppose that $(P_k)_{k\in \bb{N}}$ is a sequence of projections such that 
$P_k\in \cl P_{\rm a}(\cl S_{n_k})$, $k\in \bb{N}$, and $P_k\to_{k\to \infty} P$.  
Let $A,B\in \cl S$, and $(A_n)_{n\in \bb{N}}$ and $(B_n)_{n\in \bb{N}}$ be sequences such that
$A_n,B_n\in \cl S_n$, $n\in \bb{N}$, and $A_n\to_{n\to\infty} A$ and $B_n\to_{n\to\infty} B$. 
Then 
\begin{eqnarray*}
(PAP)(PBP) & = & 
\lim_{k\to \infty} (P_kA_{n_k}P_k)(P_kB_{n_k}P_k)\\ 
& = & 
\lim_{k\to \infty} (P_kB_{n_k}P_k)(P_kA_{n_k}P_k) = (PBP)(PAP);
\end{eqnarray*}
thus, $P \in \cl P_{\rm a}(\cl S)$. 

Now suppose that $(A_k)_{k\in \bb{N}}$ is a sequence with $A_k\in \ap(\cl S_{n_k})$, $k\in \bb{N}$, 
and $A_k\to_{k\to \infty} A$. 
Let $B_k = \sum_{j=1}^{m_k} \mu_j^{(k)} P_j^{(k)}$ be a convex combination of $\cl S_{n_k}$-abelian projections
$P_j^{(k)}$, $j \in [m_k]$, $k\in \bb{N}$, such that $A_k\leq B_k$. 
By Carath\'{e}odory's Theorem, we may assume that $m_k = 2d^2 + 1$ for all $k\in \bb{N}$. 
Passing to subsequences, we may assume that
$P_j^{(k)}\to_{k\to\infty} P_j$ and $\mu_j^{(k)} \to_{k\to\infty} \mu_j$, $j\in [2d^2 + 1]$.
By the previous paragraph, 
$$B := \sum_{j=1}^{2d^2 + 1} \mu_j P_j \in \ap(\cl S).$$
Since $A\leq B$, we conclude that $A\in \ap(\cl S)$, and (\ref{eq_apc}) is proved. 
The claim now follows from Theorem \ref{th_cecc}. 
\end{proof}

Let $G$ be  a graph with vertex set $[d]$.  
We note that $H(G,p) = 0$ if and only if there exists 
$v = (v_i)_{i=1}^d \in \vp(G)$ such that $p_i > 0 \Rightarrow v_i = 1$.  
This is equivalent to the condition that $\{i \in [d] : p_i > 0\}$ is an independent set of $G$.  
Note that $H(G,p) = 0$ for all $p \in \cl P_d$ if and only if $G = \bar{K}_d$. 
We now address the analogous questions in the non-commutative setting.

\begin{prop} \label{geiszero} 
Let $\cl S \subseteq M_d$ be an operator system.  
\begin{itemize}
\item[(i)]
Suppose that $\rho\in \cl R_d$. 
We have that $H(\cl S ,\rho)=0$ if and only if there exists 
an orthonormal basis $\{v_1, \ldots, v_d\}$ of $\bb C^d$ 
such that, if $T= \{i \in [d]: \ip{\rho v_i}{v_i} > 0\}$ then $\sum_{i \in T} v_i v_i^* \in \ap({\cl S})$. 
\item[(ii)]
$H(\cl S, \rho) = 0$ for all $\rho \in \cl R_d$ if and only if there exists an orthonormal basis $V$ of $\bb{C}^d$ 
such that $\cl S \subseteq \cl D_V$.
\end{itemize}
\end{prop}

\begin{proof} 
(i) 
Note that $H(\cl S,\rho)=0$ if and only if there exists $A \in \ap( \cl S)$ such that 
$-\Tr (\rho \log A)=0$.  Write
$A = \sum_{i=1}^d \lambda_i v_i v_i^*$ for some orthonormal basis 
$\{v_1, \ldots, v_d\}$ and $\lambda_i \in \bb R_+$, $i\in [d]$.
We have $\sum_{i=1}^d  \ip{\rho v_i}{v_i} \log\lambda_i = 0$, and hence
$\lambda_i = 1$ whenever $\ip{\rho v_i}{v_i} > 0$. It follows that $P := \sum_{i \in T} v_i v_i^* \leq A$ and 
so $P\in \ap(\cl S)$. 
Conversely, if $P\in \ap(\cl S)$ then $\Tr (\rho \log P) = 0$ and hence $H(\cl S,\rho) = 0$.

(ii) 
Choose $\rho > 0$. By (i), 
if $H(\cl S, \rho) = 0$ then  $I \in \ap(\cl S)$.  
Thus, for some orthonormal basis 
$V = \{v_1, \ldots, v_d\}$ of $\bb C^d$, we have that $v_iv_j^* \in \cl S^\perp $ for all $i \neq j$.  
We conclude that $\cl S$ is diagonal in basis $V$.  
Conversely, if $\cl S$ is diagonal in some orthonormal basis then
$I$ is an $\cl S$-abelian projection, and (i) gives $H(\cl S, \rho) = 0$ for all $\rho \in \cl R_d$. 
\end{proof}

We next consider the extremal cases for the values of $H(\cl S,\rho)$; 
Propositions \ref{Scommutes} and \ref{chif=d} should 
be compared to Propositions \ref{H_A=0} and \ref{m(A)=d}.

\begin{prop} \label{Scommutes} 
The following are equivalent for a non-commutative graph $\cl S \subseteq M_d$:
\begin{enumerate} 
\item[(i)] $\cl S$ is diagonal in some orthonormal basis;
\item[(ii)] $H(\cl S, \rho)=0$  for all states $\rho \in \cl R_d$;
\item[(iii)] $\chi_{\rm f}(\cl S)=1$;
\item[(iv)] $I \in \ap (\cl S)$;
\item[(v)] $\alpha (\cl S) =d$;
\item[(vi)] $\chi(\cl S)=1$.
\end{enumerate}
\end{prop}

\begin{proof}   
(i)$\Leftrightarrow$(ii) is Proposition \ref{geiszero}.

(ii)$\Leftrightarrow$(iii)$\Leftrightarrow$(iv)$\Leftrightarrow$(v)  
Apply Proposition \ref{H_A=0}, 
recalling that 
$$H(\cl S, \rho) = H_{\ap(\cl S)}(\rho)$$ 
and using that $\alpha(\cl S)= \gamma(\ap(\cl S))$ and 
$\chi_{\rm f}(\cl S)= \gamma(\ap(\cl S)^\sharp).$

(iv)$\Leftrightarrow$(vi) is clear from the definition of $\chi(\cl S)$.
\end{proof}

\begin{prop} \label{chif=d} 
The following are equivalent for non-commutative graph $\cl S \subseteq M_d$:
\begin{itemize} 
\item[(i)] $H(\cl S, \rho)=H(\rho)$  for all states $\rho \in \cl R_d$;
\item[(ii)] $\chi_{\rm f}(\cl S)=d$;
\item[(iii)]$\chi(\cl S)=d$;
\item[(iv)] $\ap (\cl S)= \cl A_{I_d}$;
\item[(v)] $\alpha (\cl S) =1$.
\end{itemize}
\end{prop}
\begin{proof} 
(iii)$\Rightarrow$(iv)  
All rank one projections are trivially $\cl S$-abelian. 
Suppose that $P$ is an $\cl S$-abelian projection with $\rank(P) \ge 2$. 
Then $I$ can be expressed as the sum of $P$ and at most $(d-2)$ rank one projections, 
giving $\chi(\cl S) \le d-1.$

(ii)$\Rightarrow$(iv) 
From their respective definitions, it is clear that $\chi_{\rm f}(\cl S)\leq \chi(\cl S)\leq d$.

(i)$\Leftrightarrow$(ii)$\Leftrightarrow$(iv)$\Leftrightarrow$(v) 
follow from Proposition \ref{m(A)=d}.
\end{proof}

\begin{remark}\label{completecase}
Clearly, the equivalent conditions of Proposition \ref{chif=d} are satisfied if $\cl S = M_d$. However, 
there exist proper operator subsystems of $M_d$ for which these conditions are also satisfied,
for example, the operator system $\cl S_d$ considered in Section \ref{ch4examples} for $d \ge 1$ 
(this follows from Propositions \ref{chif=d} and \ref{Snspaces}).
\end{remark}

We finish this section with noting the subadditivity of the entropy. 

\begin{prop}\label{p_submHS}
Let $\cl S_i\subseteq M_{d_i}$ be a non-commutative graph, $i = 1,2$, and 
$\rho \in \cl R_{d_1 d_2}$. Then
$$H(\cl S_1\otimes \cl S_2,\rho) \leq H(\cl S_1,\Tr\mbox{}_1\rho) + H(\cl S_2,\Tr\mbox{}_2\rho).$$
\end{prop}

\begin{proof}
It is clear that, if $P_i\in \cl P_{\rm a}(\cl S_i)$, $i = 1,2$, then $P_1\otimes P_2\in \cl P_{\rm a}(\cl S_1\otimes \cl S_2)$.
Thus, 
\begin{equation}\label{eq_apma}
\ap(\cl S_1)\otimes_{\max} \ap(\cl S_2) \subseteq \ap(\cl S_1\otimes \cl S_2).
\end{equation}
The statement now follows from Theorem \ref{th_enprod}.
\end{proof}




\section{Cliques and clique covering number}\label{ss_ccn}

In this section, we discuss the non-commutative versions of cliques and the 
clique covering number, and their entropic meaning, and provide a bound on the Shannon capacity 
of a non-commutative graph. 
Recall that the clique number $\omega(G)$ of a graph $G$ is defined as the size of a largest clique of $G$.
In the non-commutative case, clique and full projections both constitute a legitimate 
quantum version of a clique, and so we have two 
versions of $\omega(G)$ for an operator system $\cl S\subseteq M_d$ \cite[Corollary 3.9]{BTW}:
the \emph{clique number}
$$\omega(\cl S) = \max\left\{\rank P : P \mbox{ is an } \cl S\mbox{-clique projection}\right\}$$
of $\cl S$, and the \emph{full number} 
$$\tilde{\omega}(\cl S) = \max\left\{\rank P : P \mbox{ is an } \cl S\mbox{-full projection}\right\}$$
of $\cl S$.
Note that 
\begin{equation}\label{eq_omgafp}
\omega(\cl S) = \gamma\left(\cp(\cl S)\right) \ \mbox{ and } \ 
\tilde{\omega}(\cl S) = \gamma\left(\fp(\cl S)\right).
\end{equation}
The clique covering number of $G$, on the other hand, 
is the minimum number of cliques of $G$ whose union is equal to the 
vertex set of $G$. It is clear that the latter parameter coincides with the chromatic number 
$\chi(\bar G)$ of the complement $\bar G$ of $G$, which is often denoted by $\bar \chi(G)$.
We thus have the following natural non-commutative analogues of $\bar \chi(G)$ and its fractional versions:

\begin{definition} \label{cfcover} 
Let $\cl S\subseteq M_d$ be an operator system. We define 
\begin{itemize}
\item[(i)] 
the \emph{clique covering number} of $\cl S$ by 
$$\Omega(\cl S) = \min\left \{k\in \bb{N} : \  P_1,\dots,P_k \in \cl P_{\rm c}(\cl S), \
\sum_{i=1}^k P_i = I\right\};$$

\item[(ii)] 
the \emph{full covering number} of $\cl S$ by
$$\tilde{\Omega}(\cl S) = \min\left \{k\in \bb{N} : \  P_1,\dots,P_k \in \cl P_{\rm f}(\cl S), \
\sum_{i=1}^k P_i = I\right\}.$$
If the condition on the right hand side of the last equation cannot be satisfied, we set 
$\tilde{\Omega}(\cl S) = \infty$;

\item[(iii)] 
the \emph{fractional clique covering number} of $\cl S$ by 
$$\Omega_{\rm f}(\cl S) =  \inf\left\{\sum_{i=1}^k \lambda_i : k \in \bb N,\ \lambda_i > 0, \   
P_1,\dots,P_k \in \cl P_{\rm c}(\cl S),\   \sum_{i=1}^k \lambda_i P_i \ge I \right \};$$

\item[(iv)] 
The \emph{fractional full covering number} of $\cl S$ by
$$\tilde{\Omega}_{\rm f}(\cl S) =  \inf\left\{\sum_{i=1}^k \lambda_i : k \in \bb N,\ \lambda_i > 0, \  
 P_1,\dots,P_k \in \cl P_{\rm f}(\cl S),\   \sum_{i=1}^k \lambda_i P_i \ge I \right\}.$$
If the condition on the right hand side of the last equation cannot be satisfied, we set 
$\tilde{\Omega}_{\rm f}(\cl S) = \infty$.
\end{itemize}
\end{definition}

Similarly to Proposition \ref{eq_chifGamma}, one can show that 
\begin{equation}\label{eq_Om}
\Omega_{\rm f}(\cl S) = \Gamma_{\rm f}(\cp(\cl S)) \mbox{ and } \ 
\tilde{\Omega}_{\rm f}(\cl S) = \Gamma_{\rm f}(\fp(\cl S)).
\end{equation}
It now follows from Theorem \ref{NMgamma2} that 
$\Omega_{\rm f}(\cl S)$ (resp. $\tilde{\Omega}_{\rm f}(\cl S)$) 
coincides with the complementary fractional clique number 
(resp. the complementary fractional full number) defined in \cite{BTW}
and denoted therein by $\kappa(\cl S)$ (resp. $\varphi(\cl S)$).

We collect the main properties of these parameters in the next theorem.

\begin{theorem}\label{th_mainpro}
Let $G$ be a graph with vertex set $[d]$, and $\cl S$ and $\cl T$ be non-commutative graphs in $M_d$ 
with $\cl S\subseteq \cl T$. The following hold:

\begin{itemize}
\item[(i)]
$0 \le  \tilde{\omega}(\cl S) \le \omega(\cl S) \le \omega_{\rm f}(\cl S)  \le \chi(\cl S)  \le d$.

\item[(ii)]
$1 \le   \Omega_{\rm{f}}(\cl{S}) \le  \Omega(\cl{S}) \le d$ and 
$1 \le  \tilde{\Omega}_{\rm{f}}(\cl{S}) \le  \tilde{\Omega}(\cl{S}) \leq +\infty$;

\item[(iii)]
$\Omega(\cl S) \le \tilde{\Omega}(\cl S)$ and 
$\alpha(\cl S)  \leq \Omega_{\rm f}(\cl S) \le \tilde{\Omega}_{\rm f}(\cl S)$;

\item[(iv)]
$\Omega(\cl S) = 1$ $\Leftrightarrow$ $\omega(\cl S) = d$;

\item[(v)]
$\tilde{\Omega}_{\rm f}(\cl S) = 1$ $\Leftrightarrow$ $\tilde{\Omega}(\cl S) = 1$ $\Leftrightarrow$  $\tilde{\omega}(\cl S) = d$ $\Leftrightarrow$ $\cl S=M_d$;

\item[(vi)]
If  $\tilde{\omega}(\cl S) = 0$ then $\tilde{\Omega}_{\rm f}(\cl S) = \infty$;

\item[(vii)]
$\tilde{\Omega}_{\rm f}(\cl S) = \infty$ $\Leftrightarrow$ $\fp(\cl S)^\sharp$ is unbounded 
$\Leftrightarrow$ $\fp(\cl S)$ has empty relative interior;


\item[(viii)]
$\chi_{\rm f}(\cl S_G) = \chi_{\rm f}(G)$;

\item[(ix)]
$\tilde{\Omega}_{\rm f}(\cl S_G) =\Omega_{\rm f}(\cl S_G) = \chi_{\rm f}(\bar{G})$;

\item[(x)]
$\tilde{\Omega}(\cl S_G) = \Omega(\cl S_G)= \chi (\bar{G})$;

\item[(xi)]
If $\zeta \in \{\Omega_{\rm f}, \tilde{\Omega}_{\rm f}, \Omega, \tilde{\Omega} \}$, then $\zeta(\cl S) \ge \zeta(\cl T)$;

\item[(xii)]
If $\zeta \in \{ \omega,  \tilde{\omega}, \omega_{\rm f}, \chi    \}$, then $\zeta(\cl S) \le \zeta(\cl T)$;

\item[(xiii)]
$\alpha(\cl S) \chi (\cl S) \ge d$, $\omega(\cl S)\Omega( \cl S) \ge d$ and, if $\tilde{\omega}(\cl S) \ge 1$ then $\tilde{\omega}(\cl S)\tilde{\Omega}( \cl S) \ge d$. 
\end{itemize}
\end{theorem}

\begin{proof}
(i) 
Theorem \ref{2ndABe} and (\ref{eq_sand})
give $\fp(\cl S) \subseteq \cp(\cl S) \subseteq \ap(\cl S)^\sharp \subseteq \cl B_{I_d}$.
The assertion follows from (\ref{eq_omgafp}) and Theorem \ref{th_fcc}. 

(ii) 
Using (\ref{eq_Om}) and Remark \ref{r_spcnd}, we have 
$$\Omega_{\rm f}(\cl S) = \Gamma_{\rm f}({\rm cp}(\cl S)) \leq \Gamma ({\rm cp}(\cl S)) \leq \Omega(\cl S).$$
The rest of the statements follow from
Theorem \ref{NMgamma2} and the fact that 
$\cl A_{I_d} \subseteq \cp(\cl S) \subseteq \cl B_{I_d}$ and $\fp(\cl S) \subseteq \cl B_{I_d}$. 

(iii) Using (\ref{eq_sand}), (\ref{eq_Om}) and Theorem \ref{NMgamma2}, we have 
$$M({\rm fp}(\cl S)) \geq M({\rm cp}(\cl S)) = \Omega_{\rm f}(\cl S) \geq M({\rm ap}(\cl S)^{\sharp})
= \gamma({\rm ap}(\cl S)) = \alpha(\cl S).$$

(iv) follows from the fact that $\Omega(\cl S) = 1$ if and only if $I\in \cl P_{\rm c}(\cl S)$, 
if and only if $\omega(\cl S) = d$. 

(v) 
Since $\fp(\cl S) \subseteq \cl B_{I_d}$, we have that
$\cl A_{I_d} \subseteq \fp(\cl S)^\sharp.$ 
Thus, if $\tilde{\Omega}_{\rm f}(\cl S)=1$ then $\fp(\cl S)^\sharp = \cl A_{I_d} $, 
yielding $\fp(\cl S) =\cl B_{I_d}$ and $\tilde{\omega}(\cl S)=d$.    
It follows that $I \in \fp(\cl S)$ and so $I$ is an $\cl S$-full projection, implying $\cl S = M_d.$  
The proof is completed by noting that if $\cl S=M_d$, then $\fp(\cl S)= \cl B_{I_d}$ and 
$\tilde{\Omega}_{\rm f}(\cl S)= \gamma(\fp(\cl S)^\sharp)= \gamma(\cl A_{I_d})=1.$

(vi) The condition $\tilde{\omega}(\cl S)=0$ holds if and only if 
$\fp(\cl S)= \{0\}$ or, equivalently, $\fp(\cl S)^\sharp= M_d^+$, which yields 
$\tilde{\Omega}_{\rm f}(\cl S)= \infty.$ 

(vii) The second equivalence is immediate from  Proposition \ref{cc1}. 
On the other hand, if $\tilde{\Omega}_{\rm f}(\cl S) < \infty$ then there exist $\cl S$-full projections 
$P_1,\dots,P_k$ such that $\sum_{i=1}^k \lambda_i P_i \geq I$ for some positive scalars $\lambda_1,\dots,\lambda_k$. 
It follows by hereditarity that $\frac{1}{\sum_{i=1}^k \lambda_i} I\in \fp(\cl S)$, and hence 
$\fp(\cl S)$ has non-empty relative interior by Lemma \ref{l_ri}. Conversely, if $\fp(\cl S)$ has non-empty relative 
interior then, by Lemma \ref{l_ri}, $rI\in \fp(\cl S)$ for some $r > 0$. 
Thus, there exist $P_i \in \cl P_{\rm f}(\cl S)$ and $\lambda_i\geq 0$, $i\in [k]$, such that 
$\sum_{i=1}^k \lambda_i = 1$ and 
$rI \leq \sum_{i=1}^k \lambda_i P_i$. This implies that $\tilde{\Omega}_{\rm f}(\cl S) \leq \frac{1}{r}$.

(viii) follows from Theorem \ref{th_fcc} and the fact that $\omega_{\rm f}(\cl S_G) = \omega_{\rm f}(G)$ 
\cite[Corollary 3.9]{BTW}.

(ix) follows from \cite[Corollary 3.9]{BTW}.

(x) 
If $\{i_1, \ldots, i_k\}$ is a clique in $G$, then $\{e_{i_1}, \ldots, e_{i_k}\}$ is an 
$\cl S_G$-full set and hence $P= \sum_{j=1}^k e_{i_j}e_{i_j}^*$ is an 
$\cl S_G$-full, and thus an $\cl S_G$-clique, projection.
Thus, $\Omega(\cl S_G) \le \tilde{\Omega}(\cl S_G) \le \chi(\bar{G})$.

Let $G$ be a graph on $d$ vertices and let $\{v_1, \ldots, v_d\}$ be an orthonormal basis of $\bb C^d$.  
A standard combinatorial 
result (see \cite[Lemma 7.28]{Paulsen} and \cite[Lemma 13]{kim}) shows that there exists a 
permutation $\sigma$ on $[d]$ such that 
$\ip{e_{\sigma(i)}}{v_i} \ne 0$ for all $i \in [d]$ and so, for $j,k \in [d]$, we have that

\begin{equation} \label{onbsigma} 
\ip{v_jv_k^*}{e_{\sigma(j)}e_{\sigma(k)}^*}=\ip{e_{\sigma(k)}}{v_k}\ip{v_j}{e_{\sigma(j)}} \ne 0. 
\end{equation}  
Let $P = \sum_{i=1}^k v_i v_i^*$.  
If $P$ is an $\cl S_G$-clique, then $v_pv_q^* \in \cl S_G$ for distinct 
$p,q \in [k]$ and, by \eqref{onbsigma}, $e_{\sigma(p)}e_{\sigma(q)}^* \notin \cl S_G^\perp$.  
Thus, 
$\sigma(p) \sim \sigma(q)$ in $G$ and $\{ \sigma(1), \ldots, \sigma(k) \}$ 
is a clique in $G$.  Then, corresponding to any family of $n$ $\cl S_G$-clique 
projections which sum to $I$, there is a family of $n$ cliques in $G$ which partition 
$V(G)$, and  $\chi(\bar{G}) \le \Omega(\cl S_G).$

(xi) 
Note that $\Omega_{\rm f}(\cl S) = \gamma({\rm cp}(\cl S)^{\sharp})$ and 
$\tilde{\Omega}_{\rm f}(\cl S) = \gamma({\rm fp}(\cl S)^{\sharp})$, and then apply 
Remark \ref{apcpbounds} to obtain the results for $\Omega_{\rm f}$ and $\tilde{\Omega}_{\rm f}$.
For $\Omega$ and $\tilde{\Omega}$, it suffices to see that if $\cl S\subseteq \cl T$
then $\cl P_{\rm c}(\cl S) \subseteq \cl P_{\rm c}(\cl T)$ and 
$\cl P_{\rm f}(\cl S) \subseteq \cl P_{\rm f}(\cl T)$.

(xii)
The results for 
$\omega_{\rm f}$, $\omega$ and $\tilde{\omega}$ follow from 
Remark \ref{apcpbounds}. 
If $\cl S\subseteq \cl T$
then $\cl P_{\rm a}(\cl T) \subseteq \cl P_{\rm a}(\cl S)$ and the result for $\chi$ follows.

(xiii)
The first inequality follows from (\ref{eq_chiGamma}) and Theorem \ref{NMgamma2}, and the rest are similar. 
\end{proof}

\noindent {\bf Remark. } 
Part (xiii) of Theorem \ref{th_mainpro} can be viewed as a non-commutative version of the 
inequality $\alpha(G) \chi(G) \ge d$ for classical graphs $G$. 
Note that corresponding results for operator anti-systems are considered in \cite[Section 3.1]{kim}.

The following fact -- an immediate corollary of Theorem \ref{noncomm13} -- gives an 
entropic significance to the parameters $\Omega_{\rm f}$ and $\tilde{\Omega}_{\rm f}$.

\begin{theorem}\label{entmean}
Let $\cl S\subseteq M_d$ be an operator system. Then 
$$\max_{\rho \in \cl R_d} H_{\cp(\cl S)}(\rho) = \log \Omega_{\rm f}(\cl S)
\ \mbox{ and } \ 
\max_{\rho \in \cl R_d} H_{\fp(\cl S)}(\rho) = \log \tilde{\Omega}_{\rm f}(\cl S).$$
\end{theorem}

\begin{theorem}\label{th_contomegaf}
Let $\cl S$ and $\cl S_n$ be non-commutative graphs in $M_d$, $n\in \bb{N}$.
\begin{itemize}
\item[(i)] 
If $\cl S\subseteq \liminf_{n\in \bb{N}}\cl S_n$ then
$\chi_{\rm f}(\cl S)\leq \liminf_{n\in \bb{N}} \chi_{\rm f}(\cl S_n)$;

\item[(ii)] If $\tilde{\Omega}_{\rm f}(\cl S) < \infty$ and $\limsup_{n\in \bb{N}}\cl S_n \subseteq \cl S$ then
$\tilde{\Omega}_{\rm f}(\cl S)\leq \liminf_{n\in \bb{N}} \tilde{\Omega}_{\rm f}(\cl S_n)$.
\end{itemize}
\end{theorem}

\begin{proof}
(i) By (\ref{eq_apc}) and Proposition \ref{p_ccc}, 
$\ap(\cl S)^{\sharp} \subseteq \liminf_{n\in \bb{N}} \ap(\cl S_n)^{\sharp}$, and the claim 
now follows from (\ref{eq_omegam}) and the proof of Corollary \ref{c_cpa}.

(ii) It is straightforward that
$\limsup_{n\in \bb{N}} \fp(\cl S_n) \subseteq \fp(\cl S)$. By Theorem \ref{th_mainpro}, 
$\fp(\cl S)$ has non-empty relative interior, and the statement follows from 
(\ref{eq_Om}) and Corollary \ref{c_cpa}.
\end{proof}

\noindent 
{\bf Remark. } 
Operator systems satisfying the conditions of Theorem \ref{th_mainpro} (vi) are 
precisely those for which no unit vector $v$ satisfies $vv^* \in \cl S$
(for example, $\Span \{I_d\}$ for $d>1$).  
Note that the converse of Theorem \ref{th_mainpro} (vi) does not hold.  
Indeed, let $d \ge 3$ and $\cl K = \Span\{ I_d, e_1e_1^* \}\subseteq M_d$.
It is straightforward to see that the only $\cl K$-full projection is $e_1e_1^*$.
Thus, $\fp(\cl K)= \{ M \in M_d^+:M \le e_1e_1^*\}$ and $\tilde{\omega}(\cl K) = 1$.
By Lemma \ref{ccinM_d2} we have $\fp(\cl K)^\sharp= \{ M \in M_d^+: \Tr (Me_1e_1^*) \le 1 \}$ 
and so $ke_2e_2^* \in \fp(\cl S)^\sharp$ for all $k \in \bb R_+$, 
giving that $\tilde{\Omega}_{\rm f}(\cl K) = \infty.$


\section{The Witsenhausen rate}\label{s_wit}

In this section, we define the Witsenhausen rate of a non-commutative graph, 
extending the well-known Witsenhausen rate of a 
classical graph \cite{wit}. En route, we examine the multiplicativity of some of the
non-commutative graph parameters discussed earlier. 
Some of our bounds are more conveniently expressed in terms of 
orthogonal complements of non-commutative graphs, 
already employed in \cite{Stahlke} and \cite{kim}. 
More specifically, a subspace $\cl T \subseteq M_d$ is called an 
\emph{operator anti-system} \cite{BTW}
if there exists an operator system $\cl S \subseteq M_d$ 
such that $\cl T = \cl S^\perp.$  
(Such subspaces are called trace-free non-commutative graphs in \cite{Stahlke}.)  
As was pointed out in \cite[Proposition 8]{kim}, 
a subspace $\cl T \subseteq M_d$ is an operator anti-system precisely when it is  
self-adjoint and traceless, in the sense that $\Tr T=0$ whenever $T \in \cl T$.
Given a graph $G$ with vertex set $[d]$, its operator anti-system \cite[Equation (7)]{Stahlke}, \cite[Definition 6]{kim}
is the space
$$\cl T_G = \Span \{e_i e_j^*: i \sim j \mbox{~in~}G\}.$$   
Note that
\begin{equation} \label{opsysanti}  
\cl T_G = (\cl S_{\bar{G}})^\perp. 
\end{equation} 

Let $\cl T \subseteq M_d$ be an operator anti-system. 
An orthonormal set $\{ v_1, \ldots, v_k \}$ in $\bb C^d$ is called 
\emph{$\cl T$-independent}  (resp. \emph{strongly $\cl T$-independent}) 
if $v_i v_j^* \in \cl T^\perp$ for all $i,j\in [k]$ with $i \ne j$ (resp. for all $i,j\in [k]$). 
It is clear that 
a set is $\cl T$-independent (resp. strongly $\cl T$-independent) 
if and only if it is $\cl T^\perp$-clique (resp. $\cl T^\perp$-full). 
The \emph{chromatic number} $\chi(\cl T)$ and \emph{strong chromatic number} $\chi_{\rm s}(\cl T)$ 
of an operator anti-system $\cl T$ were 
introduced in \cite{kim} and can be expressed in our terms as follows:
\begin{equation} \label{chistrongs} 
\chi(\cl T)= \Omega(\cl T^\perp)\mbox{ and }\chi_{\rm s}(\cl T)= \tilde{\Omega}(\cl T^\perp).
\end{equation}  
Thus, $\tilde{\Omega}_{\rm f}(\cl S)$ (resp. $\Omega_{\rm f}(\cl S)$) 
can be regarded as the fractional version of 
$\chi_{\rm s}(\cl S^\perp)$ (resp. $\chi(\cl S^\perp)$).  
It was shown in \cite[Corollary 28 and Theorem 14]{kim}, and follows from
\eqref{opsysanti} and Theorem \ref{th_mainpro}, that 
$$\chi(\cl T_{\bar{G}})  = \chi_{\rm s}(\cl T_{\bar{G}}) = \chi(\bar{G}).$$

Recall that, 
if $G_1$ and $G_2$ are graphs with vertex sets $[d_1]$ and $[d_2]$, respectively, 
their disjunctive product $G_1\ast G_2$ has vertex set $[d_1]\times [d_2]$ and 
two pairs $(i,k), (j,l)$ of vertices are adjacent if $i\sim j$ in $G_1$ or $k\sim l$ in $G_2$. 
The \emph{co-normal product} of operator anti-systems \cite{Stahlke}
$\cl T_{i}\subseteq M_{d_{i}}$, $i = 1,2$, is the operator anti-system
$$\cl T_1 * \cl T_2= \cl T_1 \otimes M_{d_2} + M_{d_1} \otimes \cl T_2.$$     
It is straightforward that
$$\cl T_{G_1} \ast \cl T_{G_2} = \cl T_{G_1 \ast G_2}.$$
Note that, if $\cl S_1$ and $\cl S_2$ are operator systems then 
$(\cl S_1 \otimes \cl S_2)^\perp = \cl S_1^{\perp} \ast \cl S_2^{\perp}$.

The next theorem collects the submultiplicativity properties of the chromatic, 
the fractional chromatic, the clique and the clique covering numbers.
Part (i) answers \cite[Question 7.5]{BTW}.

\begin{theorem}\label{th_mulpco}
Let $\cl S_i\subseteq M_{d_i}$ be a non-commutative graph, and 
$\cl T_i\subseteq M_{d_i}$ be an operator anti-system, $i = 1,2$. 
\begin{itemize}
\item[(i)]
If $\zeta\in \{\chi, \chi_{\rm f}, \tilde{\Omega}, \tilde{\Omega}_{\rm f}\}$ then 
$\zeta(\cl S_1 \otimes \cl S_2) \leq \zeta(\cl S_1) \zeta(\cl S_2)$;

\item[(ii)]
$\tilde{\omega}(\cl S_1 \otimes \cl S_2) \ge \tilde{\omega}(\cl S_1) \tilde{\omega}(\cl S_2)$;

\item[(iii)]
$\omega(\cl S_1 \otimes \cl S_2) \ge \min \{\omega(\cl S_1), \omega(\cl S_2)\}$;

\item[(iv)] 
If $\tilde{\omega}(\cl S_2)\ge 1$ then $\omega(\cl S_1 \otimes \cl S_2) \ge \omega(\cl S_1).$ 
Thus, if $\tilde{\omega}(\cl S_i)\ge 1$, $i = 1,2$, then 
$\omega(\cl S_1 \otimes \cl S_2) \ge \max\{\omega(\cl S_1),\omega(\cl S_2)\};$

\item[(v)]
$\chi_{\rm s}(\cl T_1 * \cl T_2) \le \chi_{\rm s}(\cl T_1) \chi_{\rm s} (\cl T_2)$.
\end{itemize}
\end{theorem}

\begin{proof} 
(i)
Suppose that $\{P_i^{(k)}\}_{i=1}^{l_k}$ is a PVM consisting of projections in $\cl P_{\rm a}(\cl S_k)$, $k = 1,2$.
Then $\{P_i^{(1)}\otimes P_j^{(2)} : i\in [l_1], j\in [l_2]\}$ is a PVM consisting of projections in 
$\cl P_{\rm a}(\cl S_1\otimes \cl S_2)$; minimising over $l_1$ and $l_2$ proves the claim if 
$\zeta = \chi$. 
A similar argument shows the claim for $\zeta = \tilde{\Omega}$. 
For $\zeta = \chi_{\rm f}$, the statement 
follows from Theorems \ref{th_partp} and \ref{NMgamma2}, Proposition \ref{eq_chifGamma} and (\ref{eq_apma}).
The claims in the case $\zeta = \tilde{\Omega}_{\rm f}$ follow from
the -- straightforward to verify -- inclusion 
\begin{equation}\label{eq_fpma}
\fp(\cl S_1)\otimes_{\max} \fp(\cl S_2) \subseteq \fp(\cl S_1\otimes \cl S_2).
\end{equation}

(ii) follows from (\ref{eq_fpma}).

(iii) 
Without loss of generality, let $\omega(\cl S_1) = p \le q = \omega(\cl S_2)$, and choose 
an $\cl S_1$-clique $\{u_1, \ldots, u_p\}$ and an $\cl S_2$-clique $\{v_1, \ldots , v_q\}$.  
The set $\{ u_i \otimes v_i: i \in [p]\}$ is then an $\cl S_1\otimes \cl S_2$-clique.

(iv) Since $\tilde{\omega}(\cl S_2)\ge 1$, there exists an $\cl S_2$-full projection $vv^*$ of rank one.  
Let $\{u_1, \ldots, u_p\}$ be an $\cl S_1$-clique, where $p = \omega(\cl S_1)$.  
We have 
$$(u_i \otimes v) (u_j \otimes v)^* = u_iu_j^* \otimes vv^* \in \cl S_1 \otimes \cl S_2, \ \ \ i \ne j,$$
and hence the set $\{u_i \otimes v : i \in [p]\}$ is an $\cl S_1 \otimes \cl S_2$-clique. 

(v)
Using  \eqref{chistrongs} and (i), we have
$$\chi_{\rm s}(\cl T_1 * \cl T_2)
= 
\tilde{\Omega}(\cl T^\perp_1 \otimes \cl T^\perp_2) 
\le 
\tilde{\Omega}(\cl T^\perp_1)\tilde{\Omega}(\cl T^\perp_2)
= \chi_{\rm s}(\cl T_1)\chi_{\rm s}(\cl T_2).$$
\end{proof}



\noindent {\bf Remark. } 
It is well-known that the clique number of classical graphs is multiplicative with respect to strong graph products
\cite[Chapter 7, Exercise 13]{Godsil}.
The same does not hold true for non-commutative graphs; 
indeed, we will see in Section \ref{ch4examples} that there exist operator systems 
$\cl S$ and $\cl T$ such that $\omega(\cl S \otimes \cl T) < \omega(\cl T)$.   

\medskip

An application of Theorems \ref{th_mainpro} and \ref{th_mulpco} yields the following bound on the
Shannon capacity of a non-commutative graph: 

\begin{cor}\label{c_bounds}
Let $\cl S \subseteq M_d$ be a non-commutative graph. Then $\Theta(\cl S) \leq \tilde{\Omega}_{\rm f}(\cl S)$. 
\end{cor}



In \cite{wit}, Witsenhausen identified the zero-error capacity of noisy channels in the presence of 
side information.  
In this scenario, in addition to a noisy channel $\cl N : [d] \To [m]$, 
Alice can communicate to Bob using an identity channel $[k] \To [k]$ for any $k \in \bb N$ of her choice, 
which she runs in parallel with $\cl N$ so that Bob can retrieve with certainty
her input $i \in [d]$.   
Thus, Alice seeks a function $f: [d] \To [k]$, such that 
the output of the channel $\cl N$ applied to $i\in [d]$, together with the value $f(i)$, 
completely determine $i$.     
The minimum value of $k$ such that these constraints can be satisfied is denoted 
$\chi(\cl N)$ and known as the \emph{packing number} of $\cl N$.       
Witsenhausen showed that  
$\chi (\cl N)$ coincides with the chromatic number $\chi(G)$ of the confusability graph $G$ of $\cl N$. 
The zero-error capacity of $\cl N$ (or, alternatively, of $G$) in the presence of side information, 
is the \emph{Witsenhausen rate}
$$R(G) = \lim_{n \To \infty} \sqrt[n]{\chi(G^{\boxtimes n})}.$$

The quantum zero-error side information problem was examined in
\cite[Section 7.3]{Paulsen}. Given a quantum channel $\Phi : M_d \To M_k$,
here we seek an orthonormal basis $\{v_1, \ldots, v_d \} \subseteq \bb C^d$, $k \in \bb N$ 
and a function $f: [d] \To [k]$ such that  the outputs $(\Phi \otimes \cl I) ((v_i \otimes e_{f(i)})(v_i \otimes e_{f(i)})^*)$ are perfectly distinguishable for $i=1, \ldots, d$, where $\cl I$ is the identity channel and $\{e_1, \ldots, e_k\}$ is the canonical orthonormal basis of 
$\bb C^n$.  The least $k \in \bb N$ with this property is the \emph{packing number} $\chi ( \Phi)$ of $\Phi$.
It was shown on \cite[p. 59]{Paulsen} that, if $\cl S$ is the confusability graph of $\Phi$ then 
$\chi (\Phi) = \chi (\cl S)$.
Theorem \ref{th_mulpco} and Fekete's Lemma now show that the limit
$$R(\cl S) := \lim_{n \To \infty} \sqrt[n]{\chi(\cl S^{\otimes n})},$$ 
which we call the \emph{Witsenhausen rate} of $\cl S$, 
exists and coincides with the infimum of the sequence
$\left(\sqrt[n]{\chi(\cl S^{\otimes n})}\right)_{n\in \bb{N}}$.
It is immediate that, if $G$ is a graph then $R(\cl S_G) = R(G)$.

Let $\cl T^{\ast n}$ denote the disjunctive product of $n$ copies of 
an operator anti-system $\cl T$. 
It follows from Theorem \ref{th_mulpco} and Fekete's Lemma that the limit
$\lim_{n \To \infty} \sqrt[n]{\chi_{\rm s}(\cl T^{\ast n})}$ exists and is equal to 
$\inf_{n \in \bb N} \{\sqrt[n]{\chi_{\rm s}(\cl T^{ n})}\}$. 
To appreciate the significance of this limit, recall that, by
\cite[Corollary 3.4.3]{schnei},
\begin{equation} \label{posmcel} 
\lim_{n \to \infty} \sqrt[n]{\chi(G^n)}= \chi_{\rm f}(G),
\end{equation} 
where $G^n$ denotes the disjunctive product of $n$ copies of a graph $G$. 
An application of \eqref{opsysanti} and Theorem \ref{th_mainpro} shows that  
$$\lim_{n \To \infty} \sqrt[n]{\chi_{\rm s}(\cl T_G^{n})} 
= \tilde{\Omega}_{\rm f}(\cl T_G^\perp).$$  
Since $\tilde{\Omega}_{\rm f}(\cl T^\perp)$ is a fractional version of 
$\chi_{\rm s}(\cl T)$, the following question about a non-commutative version of 
\eqref{posmcel} is natural:

\begin{question} 
Let $\cl T$ be an operator anti-system.
Is it true that 
$$\lim_{n \To \infty} \sqrt[n]{ \chi_{\rm s}(\cl T^{n})}= \tilde{\Omega}_{\rm f}(\cl T^\perp)?$$  
\end{question}


\section{Some examples} \label{ch4examples}

In this subsection, we consider some examples of non-commutative graphs and evaluate 
the parameters we introduced. 
For a graph $G$, let 
$$\theta(G) = \gamma(\thab(G)) = \max\left\{\Tr(A) : A\in \thab(G)\right\}$$
be the \emph{Lov\'asz number} of $G$ \cite{lovasz}. 
The non-commutative versions $\theta(\cl S)$ and $\hat{\theta}(\cl S)$ of the Lov\'{a}sz number 
were introduced in \cite{BTW}; we refer the reader to \cite{BTW} for their definitions and note here that, by  
\cite[ Corollary 4.8 and Theorem 5.2]{BTW}, if $\cl S \subseteq M_d$ is an operator system then 
\begin{equation} \label{eq_thandhatth} 
\alpha(\cl S) \le \theta(\cl S) \le \hat{\theta}(\cl S) \le d.  
\end{equation}
It was shown in \cite{BTW} that, if $G$ is a graph then $\theta(\cl S_G) = \hat{\theta}(\cl S_G) = \theta(G)$. 
It follows from \eqref{eq_Om} and Theorem \ref{NMgamma2} 
(and was shown in \cite{BTW}) that 
$\tilde{\Omega}_{\rm f}(\cl S) = \gamma (\fp(\cl S)^\sharp)$ and $ \Omega_{\rm f}(\cl S) = \gamma (\cp(\cl S)^\sharp)$. 
For completeness, whenever they are known,
we include in the following the values of $\theta$ and $\hat{\theta}$.

\begin{prop} \label{propCIpara}  
Let $d \in \bb N$. 
The following hold:
\begin{itemize}
\item[(i)]
$\alpha (\bb C I_d)= \theta(\bb C I_d) = \hat{\theta}(\bb C I_d)
= \Omega_{\rm f}( \bb CI_d)=\Omega(\bb CI_d)=d$; 

\item[(ii)]
$\omega (\bb C I_d) = \chi_{\rm{f}}(\bb C I_d) =  \chi(\bb C I_d) =1$;

\item[(iii)]
$\tilde{\Omega}_{\rm f}(\bb C I_d)= \tilde{\Omega}(\bb CI_d)= 1$ if $d = 1$, and 
$\tilde{\Omega}_{\rm f}(\bb C I_d)= \tilde{\Omega}(\bb CI_d)= \infty$ if $d \geq 2$;

\item[(iv)]
$\tilde{\omega}( \bb CI_d)= 1$ if $d = 1$ and $\tilde{\omega}( \bb CI_d)= 0$ if $d \geq 2$;

\item[(v)]
$H(\bb CI_d, \rho)=0 \mbox{ for all } \rho \in \cl R_d$;

\item[(vi)]
$\Theta(\bb CI_d) = d \ \mbox{ and } \ R(\bb CI_d) = 1$.
\end{itemize}
\end{prop}

\begin{proof} 
(i), (ii)  For orthonormal $u,v \in \bb C^d$ we have $\ip{u}{v}= \ip{uv^*}{I_d}=0$, and so $uv^*\in \bb CI_d^\perp$.  
It follows that a projection in $M_d$ is $\bb CI_d$-clique if and only if it has rank one.  Thus
$\omega(\bb CI_d)=1$, whence $\Omega(\bb CI_d)=d$.  
Proposition \ref{Scommutes} implies $\alpha(\bb CI_d) = d$
and $\chi(\bb CI_d)= \chi_{\rm f}(\bb CI_d)=1$.  
It is immediate that $\Omega_{\rm f}(\bb CI_d)=d$,
and \eqref{eq_thandhatth} yields $\theta(\bb C I_d) = \hat{\theta}(\bb C I_d) = d$. 

(iii), (iv) Note that if $d=1$ we have $e_1e_1^* \in \fp(\bb CI_1)$ and  $\fp( \bb C I_1)= [0,1]= \fp ( \bb C I_d)^\sharp$.  
This gives $\tilde{\omega}(\bb CI_1)=\tilde{\Omega}(\bb CI_1)=\tilde{\Omega}_{\rm f}(\bb C I_d)= 1$.  
However, if $d \ge 2$, no unit vector $v$ satisfies $vv^* \in \bb C I_d$.  
Thus  $\fp(\bb C I_d)= \{0\}$ and $\fp(\bb C I_d)^\sharp= M_d $, 
giving $\tilde{\omega}(\bb CI_d)=0$ and $\tilde{\Omega}(\bb CI_d) = \tilde{\Omega}_{\rm f}(\bb C I_d)= \infty.$   

(v) This follows from Proposition \ref{Scommutes}.   

(vi) We have 
$\alpha\left((\bb CI_d)^{\otimes n}\right)=\alpha(\bb CI_{d^n})=d^n$, giving $\Theta(\bb CI_d)=d.$  
Similarly, 
$\chi\left((\bb CI_d)^{\otimes n}\right)=\chi(\bb CI_{d^n})=1$, and so $R(\bb CI_d)=1.$
\end{proof}

Letting $J_d$ be the $d \times d$ matrix all of whose entries are equal to one, we define the operator system 
$\cl T_d = \bb CI_d + \bb C J_d$.

\begin{prop} \label{alphaTn} 
Let $d \in \bb N$. The following hold:
\begin{itemize}
\item[(i)]
$\alpha ( \cl T_d)= \theta  ( \cl T_d)= \hat{\theta} ( \cl T_d)= \Omega_{\rm f}(\cl T_d)= \Omega(\cl T_d)=d;$ 

\item[(ii)]
$\omega (\cl T_d) =  \chi_{\rm{f}}(\cl T_d) =  \chi(\cl T_d) =1;$  

\item[(iii)]
$H ( \cl T_d, \rho)=0 \mbox{ for all }\rho \in \cl R_d;$ 

\item[(iv)]
$\Theta(\cl T_d)=d$ and $R(\cl T_d) = 1$;

\item[(v)]
$\tilde{\Omega}(\cl T_2) = \tilde{\Omega}_{\rm f}(\cl T_2)= 2$ and $\tilde{\omega}(\cl T_2)=1$; 

\item[(vi)]
If $d \ge 3$ then $\tilde{\Omega}(\cl T_d) = \tilde{\Omega}_{\rm f}(\cl T_d)= \infty$ and $\tilde{\omega}(\cl T_d)=1.$  
\end{itemize} 
\end{prop}  

\begin{proof} 
(i)-(iii)  As $\cl T_d$ is commutative,  Proposition \ref{Scommutes} gives 
(iii) and the fact that $\alpha(\cl T_d)=d$ and $\chi_{\rm{f}}(\cl T_d) =  \chi(\cl T_d) =1$.  
Theorem \ref{th_mainpro} and \eqref{eq_thandhatth} give the remaining results.

(iv)
That $\Theta(\cl T_d) = d$ follows from (i) and the fact that 
$\alpha(\cl S)  \le \Theta(\cl S) \le \hat{\theta}(\cl S)$ (see \cite[Corollary 5.5]{BTW}). 
Theorem \ref{th_mulpco} gives $\chi(\cl T_d^{\otimes n})=1$ for all $n \in \bb N$, whence we have $R(\cl T_d)=1.$

(v) 
Suppose unit vector $v = (v_i)_{i=1}^2 \in \bb C^2$ satisfies $vv^* \in \cl T_2$, so that $\{v\}$ is $\cl T_2$-full.  
Then $ |v_1|^2=|v_2|^2= 1/2$.  Since we have
$v_1\overline{v_2}=v_2\overline{v_1}$, it follows that $v_1= \pm v_2$.  
Setting $v_1= e^{i \theta} / \sqrt{2}= \pm v_2$ for $\theta \in [0,2\pi)$, gives
\begin{equation} \label{l1.50} 
vv^* = \frac{1}{2} \begin{pmatrix} 1 & \pm 1 \\ \pm 1 & 1 \end{pmatrix} \in \cl T_2, 
\end{equation}  
and we conclude that the $\cl T_2$-full singleton sets are those of the form 
$$\left \{ \frac{e^{i\theta}}{\sqrt{2}}\begin{pmatrix} 1 \\ \pm 1 \end{pmatrix} \right\}, \ \theta \in [0,2\pi) .$$
By (ii) and Theorem \ref{th_mainpro}, $\tilde{\omega}(\cl T_2) \le \omega(\cl T_2)=1$, and so $\tilde{\omega}(\cl T_2)=1.$  
It then follows from Theorem \ref{th_mainpro} that $\tilde{\Omega}(\cl T_2) \ge 2.$  
Now let $u=\frac{1}{\sqrt{2}} \begin{pmatrix} 1 \\ 1 \end{pmatrix}$,  
$v=\frac{1}{\sqrt{2}} \begin{pmatrix} 1 \\  -1 \end{pmatrix}$, 
$P_1 = uu^*$ and $P_2 = vv^*$. 
Noting that $\{u\}$ and $\{v\}$ are $\cl T_2$-full sets and that $P_1 + P_2 = I$ yields 
$\tilde{\Omega}(\cl T_2) = 2$.

By \eqref{l1.50}, $P_1$ and $P_2$ are the  only $\cl T_2$-full projections.
Thus by Proposition \ref{l_gen}, 
a matrix $M= \begin{pmatrix} a & b \\ \overline{b} & d \end{pmatrix} \in M_2^+$ 
belongs to $\fp (\cl T_2)^\sharp$ if and only if 
$$\frac{1}{2}(a + b + \overline{b}+d) \le 1 \ \mbox{ and } \ \frac{1}{2}(a - b - \overline{b}+d) \le 1.$$
It follows that if $M \in \fp(\cl T_2)^\sharp$ then $\Tr M = a+d \le 2.$  
Since $I \in \fp(\cl T_2)^\sharp$, we have that 
$\tilde{\Omega}_{\rm f}(\cl T_2) = 2$.

(vi)
Let unit vector $v= (v_i)_{i=1}^d \in \bb C^d$ satisfy $vv^* \in \cl T_d$.  
This requires that  $|v_i |^2= 1/ d$ for all $i \in [d]$.  
Letting $i, j, k \in [d]$ be pairwise distinct, we require  $v_i \overline{v_k}= v_j \overline{v_k},$ and so  $v_i=v_j$.  
Then $v= \frac{e^{i \theta}}{\sqrt{d}} \mathbbm 1$ for some $\theta \in \bb [0, 2\pi )$, and  
\begin{equation} \label{eq_vvstar}vv^* = \frac{1}{d} J_d \in \cl T_d. \end{equation}

Thus for $d \ge 3$, the  $\cl T_d$-full singleton sets are precisely those of the form 
$\left\{\frac{e^{i\theta}}{\sqrt{d}} \mathbbm 1 \right\}$,  $\theta \in [0, 2 \pi).$  
As in (v), for $d \ge 3$ we have $\tilde{\omega}(\cl T_d) \le \omega (\cl T_d)=1$ and we conclude that $\tilde{\omega}(\cl T_d)=1$.   
From \eqref{eq_vvstar} we see  that the only $\cl T_d$-full projection is $\frac{1}{d}J_d.$ 
Then for $M \in M_d^+$ we have $M \in \fp (\cl T_d)^\sharp$ if and only if $\Tr (MJ_d) \le d$.   Let unit vector $w= (w_i)_{i=1}^d \in \bb C^d$ satisfy $\sum_{i=1}^d w_i=0$, and thus $\ip{w}{\mathbbm 1}=0.$ For  $k \in \bb R_+$  form $M= k ww^* \in M_d^+,$  giving that $ \Tr M=k$ and $\Tr(MJ_d)= k | \ip{w}{ \mathbbm 1}| ^2=0.$ Hence we have $M \in \fp( \cl T_d)^\sharp$ for all 
$k \in \bb R^+$, and
$\tilde{\Omega}_{\rm f}(\cl T_d)= \gamma(\fp(\cl T_d)^\sharp)= \infty$.  Finally note by Theorem \ref{th_mainpro} that $\tilde{\Omega}(\cl T_d) = \infty$.    
\end{proof}

\begin{example} 
Here we give some quantum channels whose related operator systems are of the form $\cl T_d$ for some $d \in \bb N$.

(i)   
Consider a quantum channel $\Phi:M_2 \to M_2$ with Kraus representation 
$\Phi(T)= \sum_{i=1}^2 A_i T A_i^*, \, \rho \in M_2$, where   
$$A_1= \frac {1}{\sqrt{2}}
\begin{pmatrix} 
1 & 1 \\ 0 & 0 
\end{pmatrix} 
\mbox{ and }   
A_2 = \frac {1}{\sqrt{2}}
\begin{pmatrix} 
0 & 0 \\ 1 & -1 
\end{pmatrix}.$$
It is easy to verify that $\cl S_{\Phi}=\cl T_2$.

(ii)  The  operators $$B_1= \frac{1}{\sqrt{2}}\begin{pmatrix}1 & 0 & 0 \\ 0 & 1 & 0 \\ 0 & 0 & 1 \\ 0 & 0 & 0 \\ 0 & 0 & 0 \\ 0 & 0 & 0  \end{pmatrix}, \ \ \ B_2= \frac{1}{\sqrt{8}}\begin{pmatrix}0 & 1 & 1 \\ 1 & 0 & 1 \\ 1 & 1 & 0 \\ 1 & -1 & 0 \\ 0 & 1 & -1 \\ 1 & 0 & -1  \end{pmatrix} $$  satisfy $B_1^*B_1=B_2^*B_2= \frac{1}{2}I_3$ and $B_2^*B_1=B_1^*B_2= \frac{1}{4}(J_3-I_3).$  It follows that the channel $\Psi:M_3 \to M_6$ given by $\Psi(\rho)= B_1\rho B_1^*+B_2 \rho B_2^*, \, \rho \in M_3$ is a quantum channel with $S_\Psi = \cl T_3.$ 
\end{example}

\begin{prop} \label{prop_omegaR}
Consider operator systems $\cl R_i \subseteq M_{d_i}$ where $\chi(\cl R_i)=1$ for $i=1, \ldots, m$.  
Then \begin{itemize} 
\item[(i)] $ \chi (\bigotimes_{i=1}^m \cl R_i) =\chi_{\rm f}(\bigotimes_{i=1}^m \cl R_i) = \omega(\bigotimes_{i=1}^m \cl R_i) = R(\bigotimes_{i=1}^m \cl R_i)
= 1;$ 
\item[(ii)]  
$\alpha(\bigotimes_{i=1}^m \cl R_i) =\Omega(\bigotimes_{i=1}^m \cl R_i)=\Omega_{\rm f}(\bigotimes_{i=1}^m \cl R_i)  
= \theta (\bigotimes_{i=1}^m \cl R_i) \newline
= \hat{\theta} (\bigotimes_{i=1}^m \cl R_i) 
 = c (\bigotimes_{i=1}^m \cl R_i)
= d_1 \ldots d_m. $ \end{itemize} 
\end{prop}

\begin{proof} 
(i) By  Theorem \ref{th_mulpco}, $\chi (\bigotimes_{i=1}^m \cl R_i) = 1$, whence it is immediate that $R(\bigotimes_{i=1}^m \cl R_i)=1$.  The remaining equalities follow from 
Theorems \ref{th_fcc} and \ref{th_mainpro}. 

(ii) By (i) and Proposition \ref{Scommutes}, $\alpha(\bigotimes_{i=1}^m \cl R_i)=d_1\ldots d_m$, which implies 
$\Theta(\bigotimes_{i=1}^m \cl R_i) = d_1\ldots d_m$.  Theorem \ref{th_mainpro} and  \eqref{eq_thandhatth}  give the 
rest of the equalities.
\end{proof}

Propositions \ref{alphaTn} and \ref{prop_omegaR} have the following corollary. 

\begin{cor} 
We have that 
\begin{itemize} \item[(i)] $\chi (\bigotimes_{i=1}^m \cl T_{d_i})=\chi_{\rm f}(\bigotimes_{i=1}^m \cl T_{d_i}) 
=\omega(\bigotimes_{i=1}^m \cl T_{d_i})=R(\bigotimes_{i=1}^m \cl T_{d_i}) = 1;$ 
\item[(ii)] 
$\alpha(\bigotimes_{i=1}^m \cl T_{d_i}) 
= \Omega(\bigotimes_{i=1}^m \cl T_{d_i}) 
= \Omega_{\rm f}(\bigotimes_{i=1}^m \cl T_{d_i}) 
 =\theta (\bigotimes_{i=1}^m \cl T_{d_i}) \newline = \hat{\theta}(\bigotimes_{i=1}^m \cl T_{d_i})
 = \Theta(\bigotimes_{i=1}^m \cl T_{d_i})= d_1 \ldots d_m.$
\end{itemize}
\end{cor}

Next we discuss an operator system that has  been widely considered in the literature, see for example \cite{kim} and \cite{LPT}, namely 
$$\cl S_d = {\rm span}\{e_ie_j^*, I_d : i \ne j \} \subseteq M_d,  \ \ d \in \bb N.$$ 
For $d \ge 2$, $\cl S_d$ is not commutative, and so it does not reduce to the rather trivial case of 
Proposition \ref{Scommutes}, and nor is it equal to $\cl S_G$ for any graph $G$.  
In \cite{LPT} it was shown that $\alpha(\cl S_2) = 1$, and in 
\cite[Examples 4, 22]{kim}  that $\chi(\cl S_d) = \chi_{\rm s}(\cl S_d^{\perp}) = d$, 
while the parameters $\alpha$, $\omega_{\rm f}$, $\chi$, $\tilde{\omega}$ were 
identified in \cite[Proposition 3.12]{BTW}.
Here we extend these results by identifying 
the values of some  of the parameters introduced in Sections \ref{ss_ccn} and \ref{s_wit}.

\begin{prop} \label{Snspaces}
Let $d_1,\dots,d_m\in \bb{N}$. Then

\begin{itemize}
\item[(i)] $R(\bigotimes_{i=1}^m \cl S_{d_i}) = d_1, \ldots, d_m;$

\item[(ii)] $\tilde{\Omega}(\cl S_2)=\tilde{\Omega}_{\rm f}(\cl S_2)=2$ and $\tilde{\Omega}(\bigotimes_{i=1}^m \cl S_{d_i}) \ge \tilde{\Omega}_{\rm f}(\bigotimes_{i=1}^m \cl S_{d_i}) \geq d_1\dots d_m;$

\item[(iii)] $\Omega_{\rm f}(\cl S_d)=\Omega(\cl S_d)=1$.
\end{itemize}
\end{prop}

\begin{proof}
(i) 
In  \cite[Proposition 3.12]{BTW} we have  
$\chi\left(\bigotimes _{i=1}^m \cl S_{d_i}\right) = d_1 \ldots d_m$, and  the result follows.

(ii)
It follows from Proposition \ref{l_gen} and  the expression for $\fp(\bigotimes_{i=1}^m \cl S_{d_i})$ given in \cite[Proposition 3.12]{BTW} that  $I_{d_1 \ldots d_m}\in \fp(\bigotimes_{i=1}^m \cl S_{d_i})^{\sharp}$. 
Theorem \ref{th_mainpro} then
gives that  $\tilde{\Omega}(\bigotimes_{i=1}^m \cl S_{d_i}) \geq \tilde{\Omega}_{\rm f}(\bigotimes_{i=1}^m \cl S_{d_i}) \geq d_1 \dots d_m$. As $\cl T_2 \subseteq \cl S_2$,
Theorem \ref{th_mainpro} and Proposition \ref{alphaTn} give 
$\tilde{\Omega}_{\rm f}(\cl S_2) \le \tilde{\Omega}_{\rm f}(\cl T_2)=2$ 
and $\tilde{\Omega}(\cl S_2) \le \tilde{\Omega}(\cl T_2)=2$, and we conclude that 
$\tilde{\Omega}(\cl S_2)=\tilde{\Omega}_{\rm f}(\cl S_2)=2.$

(iii) It is clear that $\{ e_1, \ldots, e_d\}$ is an $\cl S_d$-clique.  
Thus $I_d$ is an $\cl S_d$-clique projection and hence 
$\Omega(\cl S_d)=1$.  
By Theorem \ref{th_mainpro}, $\Omega_{\rm f}(\cl S_d) = 1$.
\end{proof}

We conclude with an example of an interesting phenomenon pointed out   
at the end of Section \ref{s_wit}. 

\begin{example} \label{omega5} 
Consider the operator system $\bb CI_2 \otimes \cl S_2$.  Recall from \cite[Proposition 3.12]{BTW} that $\omega(\cl S_2)=2$  and observe that $\omega(\bb CI_2)=1$ by Proposition \ref{propCIpara}. We claim that $\omega(\bb CI_2 \otimes \cl S_2)= 1 < \omega  (\cl S_2)$.  Since $\{u\}$ is an $\bb C I_2 \otimes \cl S_2$-clique for any unit vector  $u \in \bb C^4$, it suffices to show that no $(\bb CI_2 \otimes \cl S_2)$-clique has cardinality greater than 1.  To establish this, we show that if $uv^* \in \bb CI_2 \otimes \cl S_2$, then $u=0$ or $v=0$.   We note that
$$\bb CI_2 \otimes \cl S_2
= \left \{ \begin{pmatrix} 
\lambda & a & 0 & 0 \\ 
b & \lambda & 0 & 0 \\ 
0 & 0 & \lambda & a \\ 
0 & 0 & b & \lambda 
\end{pmatrix}
:\   \lambda, a, b \in \bb C \right \}.$$  
For $u,v \in \bb C^4$, write $u=(u_i)_{i =1}^4$ and $v=(v_i)_{i =1}^4,$ and suppose that $uv^* = (u_i \overline{v}_j)_{i,j =1}^4\in \bb CI_2 \otimes \cl S_2.$ This requires $u_1\overline{v}_3=u_1\overline{v}_4=u_2\overline{v}_3=u_2\overline{v}_4=0$, giving $u_1=u_2=0$ or $v_3=v_4=0.$
  Since for $uv^* \in \bb CI_2 \otimes \cl S_2$ we also have $$u_1\overline{v}_1=u_2\overline{v}_2=u_3\overline{v}_3=u_4\overline{v}_4,$$ it must then hold that all these terms vanish. Similarly,  $u_1\overline{v}_2=u_3\overline{v}_4$ and  vanishes because either $v_4=0$ or $u_1=0$. Finally, $u_2\overline{v}_1=u_4\overline{v}_3$ and vanishes because $u_2=0$ or $v_3=0$.  We then have $uv^*=0$, and it follows that $u=0$ or $v=0$, and $\{ u,v\} $ is not a $\bb CI_2 \otimes \cl S_2$-clique. 
  \end{example}

\smallskip

\noindent 
{\bf Acknowledgement. } 
AW acknowledges financial support by the Spanish MINECO (projects 
FIS2016-86681-P and PID2019-107609GB-I00) with the support of FEDER 
funds, and the Generalitat de Catalunya (project CIRIT 2017-SGR-1127).
It is our pleasure to thank Giannicola Scarpa for
fruitful discussions on the topic of graph entropy, and
P\'eter Vrana for valuable comments concerning convex corners.


\end{document}